\input ./style/arxiv-general.cfg
\documentclass[aop,MSNbibl,seceqn,dvips]{arximspdf}
\makeatletter
   \@ifpackageloaded{graphicx}{}{\usepackage{graphicx}}
\makeatother
\usepackage{mathrsfs,mathbh}

\doi{10.1214/14-AOP923} 
\volume{43}
\issue{4}
\pubyear{2015}
\firstpage{1892}
\lastpage{1943}

\makeatletter
\renewcommand{\mid}{|}
\newcommand{\idotsint}{\mathop{\int\cdots\int}}
\newtheorem{theorem}{Theorem}[section]
\newtheorem{proposition}[theorem]{Proposition}
\newtheorem{lemma}[theorem]{Lemma}
\newtheorem{corollary}[theorem]{Corollary}
\newproclaim{example}[theorem]{Example}
\newproclaim{remark}[theorem]{Remark}
\newproclaim{definition}[theorem]{Definition}

\newcommand{\Vol}{\operatorname{vol}}
\newcommand{\dint}{\mathrm{d}}

\def\BB{\mathbb{B}}
\def\CC{\mathbb{C}}
\def\EE{\mathbb{E}}
\def\HH{\mathbb{H}}
\def\LL{\mathbb{L}}
\def\NN{\mathbb{N}}
\def\PP{\mathbb{P}}
\def\RR{\mathbb{R}}
\def\SS{\mathbb{S}}
\def\TT{\mathbb{T}}
\def\ZZ{\mathbb{Z}}

\def\bC{\mathbf{C}}
\def\bD{\mathbf{D}}
\def\bG{\mathbf{G}}
\def\bN{\mathbf{N}}
\def\bP{\mathbf{P}}
\def\bPi{\bolds{\Pi}}
\def\bQ{\mathbf{Q}}
\def\br{r}
\def\bS{\mathbf{S}}
\def\bT{\mathbf{T}}

\def\cA{\mathcal{A}}
\def\cB{\mathcal{B}}
\def\cG{\mathcal{G}}
\def\cH{\mathcal{H}}
\def\cU{\mathcal{U}}
\def\cV{\mathcal{V}}

\def\sC{\mathscr{C}}
\def\sD{\mathscr{D}}
\def\sG{\mathscr{G}}
\def\sP{\mathscr{P}}
\def\sV{\mathscr{V}}

\def\1{\mathbh{1}}
\def\bpr{\bolds{\pi}}

\def\obpr{\bar{\bpr}}
\def\tpr{\tilde{\pi}}
\def\obC{\overline{\bC}}
\def\obP{\overline{\bP}}
\def\obQ{\overline{\bQ}}
\def\oBT{\overline{\BB\TT}}

\def\Liso{{\Lambda_{\mathrm{iso}}}}
\def\iso{\mathrm{iso}}
\def\loc{\mathrm{loc}}
\def\Prob{\operatorname{Prob}}
\def\inn{\mathrm{in}}
\def\out{\mathrm{out}}
\def\av{\mathrm{av}}
\def\col{\operatorname{col}}
\def\hor{\mathrm{hor}}
\def\vert{\mathrm{vert}}
\def\diam{\mathrm{diam}}
\def\reg{\mathrm{reg}}
\makeatother

\begin{document}
\begin{frontmatter}

\title{Branching random tessellations with interaction:
A~thermodynamic view}
\runtitle{Branching random tessellations}

\begin{aug}
\author[A]{\fnms{Hans-Otto} \snm{Georgii}\corref{}\ead[label=e2]{georgii@math.lmu.de}},
\author[B]{\fnms{Tomasz} \snm{Schreiber}\thanksref{T1}}
\and
\author[C]{\fnms{Christoph} \snm{Th\"ale}\ead[label=e3]{christoph.thaele@rub.de}}
\runauthor{H.-O. Georgii, T. Schreiber and C. Th\"ale}
\affiliation{Ludwig-Maximilians University Munich, Nicolaus Copernicus University Toru\'n and Ruhr University Bochum}
\address[A]{H.-O. Georgii\\
Mathematisches Institut\\
Ludwig-Maximilians-Universit\"{a}t\\
Theresienstra{\ss}e 39\\
80333 M\"{u}nchen\\
Germany\\
\printead{e2}}
\address[B]{T. Schreiber\\
Faculty of Mathematics\\
\quad and Computer Science\\
Nicolaus Copernicus University\\
ul. Chopina 12/18\\
87-100 Toru\'n\\
Poland} 
\address[C]{C. Th\"ale\\
Fakult\"at f\"ur Mathematik\\
Ruhr-Universit\"at Bochum\\
Universit\"atsstra\ss e 150\\
44801 Bochum\\
Germany\\
\printead{e3}}
\end{aug}
\thankstext{T1}{This work has been initiated by the second author,
Tomasz Schreiber.
Shortly before his untimely death of cancer at the age of 35, he
entrusted to the first
author a draft manuscript that introduced the general setup of this
paper and proposed a
study of relative entropy densities for BRTs. To keep his memory, we
have tried to
realise his intentions with suitable modifications of concepts and
results. HOG, CT.}

\received{\smonth{4} \syear{2013}}
\revised{\smonth{11} \syear{2013}}

%
\begin{abstract}
A branching random tessellation (BRT) is a stochastic process that
transforms a coarse initial tessellation of $\mathbb{R}^d$ into a
finer tessellation by means of random cell divisions in continuous
time. This concept generalises the so-called STIT tessellations, for
which all cells split up independently of each other. Here, we allow
the cells to interact, in that the division rule for each cell may
depend on the structure of the surrounding tessellation. Moreover, we
consider coloured tessellations, for which each cell is marked with an
internal property, called its colour. Under a suitable condition, the
cell interaction of a BRT can be specified by a measure kernel, the
so-called division kernel, that determines the division rules of all
cells and gives rise to a Gibbsian characterisation of BRTs. For
translation invariant BRTs, we introduce an ``inner'' entropy density
relative to a STIT tessellation. Together with an inner energy density
for a given ``moderate'' division kernel, this leads to a variational
principle for BRTs with this prescribed kernel, and further to an
existence result for such BRTs.
\end{abstract}

%

\end{frontmatter}

\section{Introduction}

A central object of stochastic geometry and spatial stochastics are
tessellations of $\RR^d$ (with $d\geq1$), that is, locally finite
families of \mbox{$d$-}dimensional convex polytopes that cover $\RR^d$ and
have pairwise disjoint interiors. They are used in many practical
applications. For example, random tessellations serve as models for
cellular or polycrystalline materials, plant cells or influence zones,
for instance, in the modelling of telecommunication networks or animal
territories; see \cite{OkaBooSugChi00,SKM} for an overview.

The standard random tessellations usually considered in the literature
are the Poisson hyperplane tessellations, the Poisson--Voronoi and the
Poisson--Delaunay tessellations; cf.~\cite{SW} for definitions. These
have the property of being facet-to-facet (or side-to-side in the
planar case), which is to say that the intersection of any two of its
cells is either empty or a {common} face of {both} cells. However,
there are numerous applications for which models of this kind are
inappropriate, for example, network models for telecommunication
systems or models for crack structures in geology. Hence, there is a
growing demand for mathematically tractable models of
nonfacet-to-facet tessellations, which may serve as idealised
reference models. Only some years ago, the class of iteration-stable
random tessellations (called STIT tessellations for short) was
introduced by Nagel and Wei\ss\ in \cite{NW05}. These tessellations
are constructed by means of a temporal random process of cell division,
and thus live in space--time. They
have attracted considerable interest because of its analytical
tractability; see, for example, \cite{RedenbachThaele,ST2,ST3,ST7,ST4,ST5,ST6} or~\cite{TWN}.

Our objects of study here generalise the STIT models in two respects.
On the one hand, we consider coloured tessellations, for which each
cell is equipped with an individual colour. For example, the colour of
a cell could represent its nutrient content, its genotype, age, or
whatever else might be relevant to describe the state of a cell. (In a
different context, coloured tessellations have been studied by Arak and
Surgailis \cite{ArakSurgailis89,ArakSurgailis91}, e.g.)
On the other hand, and more importantly, we allow for an \emph
{interaction of cells} during their division process.
That is, our objects of interest can be viewed in two ways that are
equivalent but deal differently with space--time: either
\begin{itemize}[--]
\item[--] as Gibbsian spatial systems of interacting branching processes of
coloured cells,~or
\item[--] as temporal processes of tessellations in space.
\end{itemize}
The latter viewpoint can informally be described as follows.
At time zero, one starts with an initial
random tessellation of $\RR^d$ into coloured cells. Each cell lives
for a random time, which is determined by an interactive competition of
cells. Namely, the survival rate of a cell $c$ at any time $s>0$ may
not only depend on the cell's geometry and colour, but in fact on the
whole tessellation including its past evolution. When the lifetime has
run out, a hyperplane with coloured half-spaces is chosen randomly
according to some rule that may again depend
on the cell's geometry, colour and the past evolution of the surrounding
tessellation, and is used to cut $c$ into two polyhedral sub-cells
$c^+$ and $c^-$,
which inherit their colours from the respective half-spaces of the cutting
hyperplane. The daughter cells $c^+$~and~$c^-$ then replace $c$ in the
collective
division game, which is continued until time $1$, say. The resulting
tessellation of $\RR^d$ at a deterministic time $s\in[0,1]$ is
denoted by $T_{s}$, and the tessellation-valued stochastic process
$(T_s)_{s\in[0,1]}$ is what we call a \emph{branching random
tessellation} or BRT for short. The rule determining the splitting of
cells is given by a measure kernel, which will be called the associated
\emph{division kernel}.

In the special case when (i) the distribution of lifetimes is
exponential with parameter proportional to the mean width of the cells,
and (ii) the bi-coloured hyperplanes are chosen at random according to
the motion-invariant hyperplane measure and some reference measure on
the colour space, $(T_{s})_{s\in[0,1]}$ is a coloured STIT
tessellation of $\RR^d$ and its distribution is invariant under rigid
motions whenever so is the initial random tessellation. The coloured
STIT tessellations play an important role in the background of our
theory, in a way which is conceptually similar to that of the Poisson
point processes in the theory of Gibbsian point processes.

Let us note that the Gibbsian viewpoint, for which the BRTs are
considered as interacting branching processes of coloured cells,
parallels the Gibbsian treatment of interacting particle systems and
interacting diffusions developed in
\cite{CattiauxRZ,DaiPra,DaiPraRZ,Deuschel}, for example.
Let us also mention that different tessellation models with cell
interaction, namely Delaunay or Voronoi tessellations of Gibbsian type
(which undergo no time evolution), are studied in \cite
{BertinBilliotDrouilhet,DereudreDrouilhetGeorgii,DereudreGeorgii}.

The main results of this paper are the following.
%
\begin{itemize}[--]
\item[--] To begin, we discuss how the intuitive concept of ``cell
interaction'' that governs a BRT $\bP$ can be specified by a so-called
\emph{division kernel} $\Phi$.
We show that such a $\Phi$ can equivalently be used in two different
ways: either as the collection of instantaneous splitting rates of all
cells during their joint time evolution, or in the Gibbsian way, as a
means to determine the conditional distribution of the behaviour of all
cells within any bounded window when that of all other cells is given.
A third equivalent use of $\Phi$ involves a Campbell-like formula for
the jump intensity measure of $\bP$. We show further that a measure
kernel $\Phi$ as above exists as soon as $\bP$ satisfies a condition
of local absolute continuity (LAC) relative to a STIT model.

\item[--] We then turn to a kind of thermodynamic formalism for BRTs $\bP$
that are invariant under spatial translations.
The basic quantity is an inner entropy density $h^\inn(\bP)$, which
is defined as the limit of a conditional entropy per unit volume of~$\bP$ relative to a reference STIT model. The adjective ``inner'' refers
to the fact that only the cells completely inside the respective window
are taken into account, rather than all cells that hit the window. The
functional $h^\inn$ will be shown to share some familiar properties of
the entropy densities for the standard models of statistical mechanics,
at least with some natural adaptations.

\item[--] Finally, we consider an arbitrary division kernel $\Psi$ that
satisfies some mild regularity conditions, which roughly require that
$\Psi$ is not too far from a STIT kernel; such a $\Psi$ will be
called moderate. We introduce an associated inner energy density
$u^\inn(\bP;\Psi)$ as well as some sort of pressure $v^\inn(\bP;\Psi)$. The resulting inner excess free energy density $h^\inn(\bP;\Psi)$ gives rise to a variational principle, which states that the
minimisers of $h^\inn( \cdot;\Psi)$ are precisely the translation
invariant BRTs that admit $\Psi$ as their division kernel. It is
further shown that such minimisers do exist, for any prescribed
distribution $P$ of the time-zero tessellation. This proves the
existence of a BRT $\bP$ for any given initial distribution $P$ and
any moderate division kernel $\Psi$. For general $\Psi$, such a $\bP
$ is not necessarily unique.
\end{itemize}

The paper is organised as follows: Section~\ref{secPreliminaries}
introduces the setup and recalls some necessary facts.
Besides tessellations and BRTs, the main concepts are division kernels
and local conditional BRTs of Gibbsian type. This section also \mbox{includes}
some examples of division kernels to which our theory applies. The main
results together with their framework are stated in Section~\ref
{secResults}. These are Theorems~\ref{thmequivalence} and~\ref{thmexPhi} on the significance and existence of global division
kernels, Theorems~\ref{thmentropy} and~\ref{thmlevelsets} on the
existence of the inner entropy density and its properties, and Theorems
\ref{thmvarprinciple} and~\ref{thmexistence} on the variational
characterisation and the existence of invariant BRTs with given
moderate division kernels. All proofs are collected in the final
Section~\ref{ProofsSec}.

\section{Preliminaries}\label{secPreliminaries}

\subsection{Tessellations}\label{SecTessellations}

\subsubsection{Polytopes and tessellations}

Consider the Euclidean space ${\RR}^d$ of arbitrary dimension $d\ge
1$. We shall deal with certain random processes of coloured
tessellations\vspace*{1pt} of ${\RR}^d$ into (coloured) convex polytopes. Let us
specify these terms. First, a \emph{polytope} $p$ in ${\RR}^d$ is the
closed convex hull of a finite set of points and is always assumed to
have nonempty interior; the set of all such polytopes is denoted by
$\PP$. Each polytope $p\in\PP$ is equipped with a translation
covariant selector $m(p)$, called its ``\textit{centre}'' or ``\textit
{midpoint}'', for example, its barycentre, its Steiner point or its
circumcentre. We write $r(p)= \max_{x\in p}|x-m(p)|$ for its \emph
{radius} and $\partial p$ and $\operatorname{int}(p)$ for its topological
\emph{boundary}, respectively, \emph{interior}.

More generally, we will assume that each polytope is marked with some
internal property, called its \emph{colour}. So, we fix an arbitrary
Polish space $\Sigma$, which we call the \emph{colour space}. A
coloured polytope, called \textit{cell} in the sequel, is a pair
$c=(p,\sigma)$ with $p\in\PP$ and $\sigma\in\Sigma$. Let us
denote by
$\mathrm{sp}(c):=p$ and $\col(c):=\sigma$, respectively, the spatial
part and the colour of~$c$. The space of cells is thus $\CC:=\PP
\times\Sigma$. To simplify notation, \emph{we adopt the general
convention that spatial operations on cells} (\emph{and also on coloured
tessellations defined below}), \emph{such as intersections with subsets of
$\RR^d$ and translations}, \emph{solely refer to the spatial part and do not
affect their colours}. For example, $m(c):=m(\mathrm{sp}(c))$,
$r(c):=r(\mathrm{sp}(c))$, $\operatorname{int}(c):= \operatorname{int}(\mathrm
{sp}(c))$, $c\cap W:=(\mathrm{sp}(c)\cap W,\col(c))$ for $W\subset
\RR^d$, and $c-x:=(\mathrm{sp}(c)-x,\col(c))$ when $x\in\RR^d$. Finally,
$\Vol(c):=\Vol_d(\mathrm{sp}(c))$ is the ($d$-dimensional) volume of
the spatial part of $c$. Let us also define the space $\CC_0=\{c\in
\CC\dvtx m(c)=0\}$ of cells having their midpoint at the origin.

The cells are the constituents of the coloured tessellations which we
introduce now; for brevity we will omit the adjective ``coloured'' in the
following. (Note that letting $\Sigma$ be a singleton one recovers the
uncoloured case usually considered in the literature; cf.~\cite{SW,SKM}.)

\begin{definition}\label{deftessellation}
A (coloured) \emph{tessellation} $T$ of $\RR^d$ is a countable subset
of~$\CC$ such that:
\begin{itemize}
\item$T$ is locally finite, in that any bounded subset of $\RR^d$
only hits a finite number of cells from $T$,
\item two distinct cells of $T$ have disjoint interiors, that is,
$\operatorname{int}(c)\cap\operatorname{int}(c')=\varnothing$ for all $c,c'\in T$
with $c\neq c'$,
\item the cells cover the whole space, which is to say that $\bigcup_{c\in T}c=\RR^d$.
\end{itemize}
The space of all tessellations of $\RR^d$ will henceforth be denoted
by $\TT$.
\end{definition}

Besides tessellations of $\RR^d$, we will also consider tessellations
in local windows $W\subset\RR^d$, which will generally be chosen to
be polytopes, or sometimes also finite unions of polytopes.
So, we write $\PP_\cup$ for the set of all finite, not necessarily
connected unions of polytopes,
and for $W\in\PP_\cup$ we let $\CC_W$ be the set of cells that are
contained in $W$. We finally write $\TT_W$ for the set of all
tessellations of $W$, that is, of all finite collections $\{c_1,\ldots,c_n\}$ of cells with pairwise disjoint interiors and such that
$c_1\cup\cdots\cup c_n=W$.

\subsubsection{Measurability}

We need measurable structures on all spaces introduced above. We start
with the space $\PP$ of polytopes. As the sets in $\PP$ are compact
and nonempty, the natural metric on $\PP$ is the usual Hausdorff
distance $d_H$; cf. \cite{SW}, Chapter~12.3. Hence, the space $\PP$
can be equipped with the Borel $\sigma$-field $\cB(\PP)$ induced by $d_H$.
In fact, $\cB(\PP)$ is generated by the sets $\{p\in\PP\dvtx p\cap B\ne
\varnothing\}$ with $B\in\cB(\RR^d)$, the Borel $\sigma$-field on
$\RR
^d$; see \cite{SW}, Chapters~12.2--12.3.
The coloured counterpart $\CC$ is endowed with the product $\sigma$-field
$\cB(\CC)=\cB(\PP)\otimes\cB(\Sigma)$, where $\cB(\Sigma)$ is
the Borel $\sigma$-field on $\Sigma$. The space $\CC_0$ of centred cells
receives the trace $\sigma$-field.

We next need to introduce a suitable $\sigma$-field on $\TT$. As is usual
in point process theory, we let $\cB(\TT)$ be the $\sigma$-field
generated by the counting variables
%
\begin{equation}
\label{eqNA} N_A\dvtx \TT\rightarrow\NN\cup\{+\infty\},\qquad T\mapsto|T
\cap A|,\qquad A\in\cB(\CC),
\end{equation}
where $| \cdot|$ stands for the cardinality of the argument set,
that is, $N_A$ counts how many cells of $T$ belong to $A$. In view of
the structure of $\cB(\CC)$, $\cB(\TT)$ is also generated by the
random variables
\[
N_{B,S}\dvtx \TT\ni T\mapsto\bigl| \bigl\{c\in T\dvtx c\cap B\neq
\varnothing, \col(c) \in S \bigr\}\bigr|,
\]
with $B$ a bounded Borel set in $\RR^d$ and $S\in\cB(\Sigma)$.
Moreover, $\cB(\TT)$ is the Borel $\sigma$-field for the vague topology
on $\TT$, which is generated by the functions
\[
e_g\dvtx \TT\rightarrow[0,\infty),\qquad T\mapsto\sum
_{c\in T}g(c),
\]
where $g\geq0$ is a continuous function on $\CC$ with a bounded
support in the spatial coordinate; see \cite{KRM}, Appendix 15.7, or
\cite{KP}, Theorem~A2.3.

To deal with local properties of tessellations, we will often restrict
a tessellation to a local window $W\in\PP$.
We thus define the \emph{projection} to such a $W$ by
%
\begin{equation}
\label{pr} \pi_W\dvtx \TT\to\TT_W,\qquad T\mapsto
T_W:= \bigl\{c\cap W\dvtx c\in T, \operatorname{int}(c\cap W)\ne
\varnothing \bigr\}.
\end{equation}
In the same manner as above, we may introduce a $\sigma$-field $\cB
(\TT
_W)$ on $\TT_W$. One can then easily check that the mapping $\pi_W$
is measurable.

The culminating concept of this subsection is the following.

%
\begin{definition}
A probability measure $P$ on $(\TT,\cB(\TT))$ satisfying the
first-moment condition
$\int P(\dint T) |T_{W}| <\infty$ for all windows $W\in\PP$
is called a \emph{random tessellation}. The set of all such $P$ is
denoted by $\sP(\TT)$.
\end{definition}

\subsection{Branching tessellations}

\subsubsection{Cutting cells by hyperplanes}

We now turn to the main objects of our investigation: tessellations
which arise from a given initial tessellation by a successive splitting
of cells into two pieces by means of suitable hyperplanes. Recall that
a hyperplane $\eta$ with unit normal $u\in\SS_+^{d-1}$ (upper unit
half-sphere) and signed distance $r\in\RR$ to the origin can be
written in the form $\eta= \{ x \in\RR^d\dvtx  \langle x, u \rangle= r \}
$, where $\langle\cdot, \cdot\rangle$ stands for the usual
scalar product.\vspace*{1pt} So, the space of hyperplanes can be identified with
$\mathbb{S}^{d-1}_+ \times\RR$.
For $\eta$ as above, we write $\eta^+ = \{ x \in\RR^d\dvtx  \langle x, u
\rangle\ge r \}$ and $\eta^- = \{ x \in\RR^d\dvtx  \langle x, u \rangle
\le r \}$ for the associated half-spaces. More\vspace*{1pt} generally, we consider
\emph{bi-coloured hyperplanes} $H=(\eta,\sigma^+,\sigma^-)\in\HH:=\SS
^{d-1}_+\times\RR\times\Sigma^2$, for which each of the half-spaces
$\eta^\pm$ is equipped with a colour $\sigma^\pm$. We write
$\mathrm{sp}(H):=\eta$ and $\col^\pm(H):=\sigma^\pm$, respectively, for the
spatial part and the colours of $H$ and again \textit{adopt the convention
that spatial operations with bi-coloured hyperplanes only refer to the
spatial part}, for example, $c\cap H:=\mathrm{sp}(c)\cap\mathrm
{sp}(H)$ or $c\cap H^\pm:=(\mathrm{sp}(c)\cap\mathrm{sp}(H)^\pm,\col^\pm(H))$ for any $c\in\CC$. Moreover, for such a cell $c$, we let
%
\begin{equation}
\label{[c]} \langle c\rangle= \bigl\{H\in\HH\dvtx H\cap\operatorname{int}(c)
\ne \varnothing \bigr\}
\end{equation}
be the set of all bi-coloured hyperplanes which hit the interior of
(the spatial part of) $c$.
Each bi-coloured hyperplane $H$ defines a cell division operation
$\oslash$ on tessellations.
Namely, let $T\in\TT$, $c\in T$ and $H\in\langle c\rangle$. Then
$\oslash
$ is defined by
%
\begin{equation}
\label{oslash} \oslash_{c,H}(T):= \bigl(T\setminus\{c\} \bigr)\cup \bigl
\{c\cap H^+,c\cap H^- \bigr\}
\end{equation}
with $c\cap H^\pm$ as above.
Branching tessellations are now defined as follows. For simplicity, the
time interval will mostly be the unit interval $[0,1]$.

\begin{definition}\label{defBT}
(a) Let $W\in\PP_\cup$ be a finite union of polytopes. A \emph
{branching tessellation in the window $W$}
with bounded time interval $I=[a,b)$ or $[a,b]$ is a family $\bT
=(T_s)_{s\in I}$ of tessellations in $W$ such that:
\begin{itemize}
\item the function $s\mapsto T_s$ from $I$ to $\TT_W$ is piecewise
constant, right-continuous and has only a finite number of jumps,
\item at each point $s$ of discontinuity (so that $T_s\ne T_{s-}:=\lim_{r\uparrow s}T_r$), there exists a unique cell
$c\in T_{s-}$ and a bi-coloured hyperplane $H\in\langle c\rangle$
such that
\[
T_s = \oslash_{c,H}(T_{s-}).
\]
\end{itemize}
Further, $T_a$ is called the initial tessellation.
We write $\BB\TT_W$ for the set of all such branching tessellations
in $W$.

(b) A family $\bT=(T_s)_{0\leq s\le1}$ is called a \emph{branching
tessellation in $\RR^d$}
if for each window $W\in\PP$ the restricted process
$\bT_W=\bpr_W(\bT):= (\pi_W(T_s) )_{0\leq s\le1} $
is a branching tessellation in $W$.
Again, $T_0$ is then called the initial tessellation of $\bT$.
The set of all branching tessellations in $\RR^d$ is denoted by $\BB
\TT$.
\end{definition}

The following remark provides a further way of describing the time
evolution of a branching tessellation.

%
\begin{remark}\label{remBT}
(a) Let $\bT$ be a branching tessellation in a window $W\in\PP_\cup
$ with time interval $I=[0,1]$. (The case of other time intervals is similar.)
Keeping record of all jump times of $\bT$ together with the associated
cells that are divided and the respective cutting hyperplanes, one
arrives at the set
%
\begin{eqnarray}\label{division}
\sD(\bT) &=& \bigl\{(s,c,H)\in(0,1]\times\CC\times\HH\dvtx
T_{s-}\ne T_s,
\nonumber\\[-8pt]\\[-8pt]
&&\hspace*{20pt} c\in T_{s-}, H\in\langle c\rangle, T_s = \oslash
_{c,H}(T_{s-}) \bigr\}\nonumber
\end{eqnarray}
of all ``division events''. There is a one-to-one correspondence between
$\bT$ and the pair $(T_0,\sD(\bT))$, in that $\bT$ can be recovered
from the initial tessellation $T_0$ and the set $\sD(\bT)$ of
division events.
Indeed, labelling the elements of $\sD(\bT)$ with the indices
$1,\ldots,n:= |\sD(\bT)|$ according to the order of their time
coordinates so that $0=:s_0<s_1< \cdots<s_n\le s_{n+1}:=1$, one has
the recursion $T_s=T_0$ for $s\in[0,s_1)$ and
\[
T_s= \oslash_{c_i,H_i}(T_{s_{i-1}})\qquad\mbox{for $s\in
[s_i,s_{i+1})$, $i=1,\ldots,n$.}
\]
Finally, $T_1=T_{s_n}$.

This description also gives rise to a convenient way of visualising
$\bT$ as a graph in $[0,1]\times\CC_W$; see Figure~\ref{Tree}. The
set of vertices is
\[
\sV(\bT)= \bigl\{(0,c)\dvtx c\in T_0 \bigr\}\cup \bigl\{ \bigl(s,c
\cap H^\pm \bigr)\dvtx (s,c,H)\in\sD (\bT) \bigr\}.
\]
Moreover, each $(s,c)\in\sV(\bT)$ is equipped with a ``lifeline''
$[s,s^*)\times\{c\}$, where
$s^*=s'$ if $(s',c,H)\in\sD(\bT)$ for some $s'>s$ and $H\in\langle
c\rangle$,
and $s^*=1$ otherwise. If $s^*<1$, this lifeline is augmented by the
lines from $(s^*,c)$ to the two children
$(s^*,c\cap H^\pm)$ of $(s,c)$.
If $s^*=1$, the half-open line $[s,1)\times\{c\}$ is replaced by the
closed line $[s,1]\times\{c\}$.
In this way, one obtains a finite forest of binary ``family'' trees in
$\CC_W$ that evolve from the cells of $T_0$. So, these cells are
the roots, or ancestors, and the $|T_0|+|\sD(\bT)|$ leaves form the
tessellation $T_1$. This branching mechanism is strongly reminiscent of
the fragmentation processes considered in~\cite{Bertoin}.
%
\begin{figure}

\includegraphics{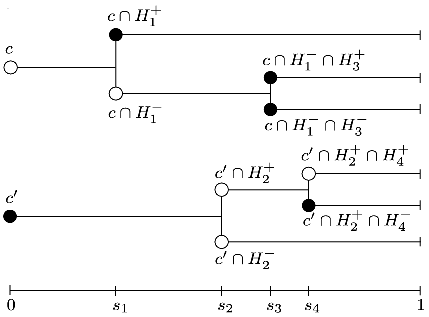}

\caption{Representation of a two-coloured branching tessellation in a
finite window with initial tessellation $T_0=\{c,c'\}$. The cells
living at a time $s$ constitute a tessellation $T_s$. At each time
$s_i$, a~cell that lives up to this moment is selected and cut in two
by a bi-coloured hyperplane $H_i$, which impresses its colours onto the
cell's pieces.}\label{Tree}
\end{figure}

(b) Branching tessellations in the whole space $\RR^d$ admit a similar
description in terms of division events. For each $\bT\in\BB\TT$,
we can then define
%
\begin{equation}
\label{eqglobaldivision} \sD(\bT)=\bigcup_{V\in\PP}\, \bigcap
_{W\in\PP\dvtx W\supset V}\sD (\bT_W).
\end{equation}
Conversely, for each $W\in\PP$ one can recover the division events in
$W$ from $\sD(\bT)$ via
\[
\sD(\bT_W)= \bigl\{(s,c\cap W,H)\dvtx (s,c,H)\in\sD(\bT), H\in
\langle c \cap W\rangle \bigr\}.
\]
It follows that $\bT$ is uniquely determined by $T_0$ and $\sD(\bT)$,
and $\bT$ can be regarded as a forest of infinitely many finite binary
family trees of coloured cells, the roots of which correspond to the
cells of the initial tessellation $T_0$ of $\RR^d$.
\end{remark}

Later on, it will be essential for us to keep track of the past of a
branching tessellation. So, instead of considering the evolution $\bT
=(T_s)_{0\leq s\leq1}$ in $\TT$, we will consider the process $(\bT
_s)_{0\le s\le1}$ in $\BB\TT$, which is given by $\bT_s=(T_u)_{0\le
u\le s}$.
Equivalently, $\bT_s$ can be thought of as being obtained from $\bT$
by removing from
$\sD(\bT)$ all elements
with time-coordinate larger than $s$. In this way, each $\bT_s$ can be
considered to be an element of $\BB\TT$, which is frozen at time $s$
(and thus remains constant thereafter). The set of all such branching
tessellations is denoted by $\BB\TT_s$. In particular, $\BB\TT
_1=\BB\TT$, and $\BB\TT_u\subset\BB\TT_s$ when $u<s$. We write
%
\begin{equation}
\label{eqbprs} \bpr_s\dvtx \BB\TT\to\BB\TT_s,\qquad \bT\mapsto
\bT_s,
\end{equation}
for the natural projection that removes the division events after time $s$.
As before, the nonbold $T_s$ stands for the tessellation at time $s$,
whereas a bold $\bT_s$ stands for an element of $\BB\TT_s$.

Besides this projection concerning time, we have also the projection to
a spatial window $W\in\PP$, which is given by
%
\begin{eqnarray}\label{eqbprW}
\bpr_W\dvtx \BB\TT\to\BB\TT_W,\qquad \bT
\mapsto \bT_W=(T_{W,s})_{0\le
s\le1}
\nonumber\\[-9pt]\\[-9pt]
\eqntext{\mbox{with } T_{W,s}= \pi_W(T_s),}
\end{eqnarray}
where $\BB\TT_W=\bpr_W(\BB\TT)$ and $\pi_W$ is as in (\ref{pr}).
We also write $\bpr_{W,s}=\bpr_W\circ\bpr_s$, $\bT_{W,s}=\bpr
_{W,s}(\bT)$ and $\BB\TT_{W,s}=\bpr_{W,s}(\BB\TT)$. So, to obtain
$\bT_{W,s}$ from $\bT$ one has to remove from $\sD(\bT)$ all
division events with a time coordinate exceeding $s$ or a hyperplane
not hitting the cell's intersection with $W$.

\subsubsection{Branching random tessellations}

Our main objects of interest are probability measures on $\BB\TT$.
So, we need to equip $\BB\TT$ with a $\sigma$-field. We know from Remark
\ref{remBT} that each $\bT\in\BB\TT$ is uniquely determined by
its initial tessellation $T_0$ together with the set
$\sD(\bT)$ of division events as given by (\ref{division}) and
(\ref{eqglobaldivision}). Since $\sD(\bT)$ is a locally finite
subset of $(0,1]\times\CC\times\HH$, one can proceed as usually in
point process theory by defining $\cB=\cB(\BB\TT)$ as the smallest
$\sigma$-field for which the counting variables
%
\begin{equation}
\label{bNA,B}\bN_{A,B}\dvtx \bT\mapsto|T_0\cap A| + \bigl|\sD(
\bT) \cap B\bigr|
\end{equation}
with $A\in\cB(\CC)$ and $B\in\cB((0,1])\otimes\cB(\CC)\otimes
\cB(\HH)$
are measurable; here $\cB((0,1])$ denotes the Borel $\sigma$-field on $(0,1]$.
By standard theory, $(\BB\TT,\cB)$ is a Borel space.
For any window $W\in\PP$, we define a $\sigma$-field $\cB_W=\cB
(\BB
\TT_W)$ on $\BB\TT_W$ in the same way.
To simplify notation, we will not distinguish between the $\sigma$-field
$\cB_W$ on $\BB\TT_W$ and its pre-image $\bpr_W^{-1}\cB_W$ on $\BB
\TT$, which will be denoted by the same symbol. Anyway, with these
definitions it is
clear that both the projection $\bpr_W$ in (\ref{eqbprW})
and the time restriction map $\bpr_\bullet\dvtx (s,\bT)\mapsto\bT_s$ of
(\ref{eqbprs}) are measurable.

%
\begin{definition}
A \emph{branching random tessellation} (\emph{BRT}) of $\RR^d$ is a
probability measure $\bP$ on $(\BB\TT,\cB)$ satisfying
the first-moment condition
%
\begin{equation}
\label{eqfirstmoment} \int\bP(\dint\bT) |T_{W,1}| <\infty\qquad\mbox{for all
windows }W\in\PP.
\end{equation}
The set of all such BRTs of $\RR^d$ is denoted by $\sP=\sP(\BB\TT
)$. BRTs within a window $W\in\PP_\cup$ are defined analogously.
\end{definition}

For every $\bP\in\sP$ and any of the projections $\bpr_\ast$ in
(\ref{eqbprs}) and (\ref{eqbprW}), we write $\bP_\ast=\bP\circ
\bpr_\ast^{-1}$ for the image of $\bP$ under $\bpr_\ast$.
In particular, each $\bP_s$ is a BRT. In fact, one can achieve that
$\bP_s$ depends measurably on $s$, in that the mapping $[0,1]\times
\cB\ni(s,A)\mapsto\bP_s(A)$ is a probability kernel, as will be
assumed throughout the following.
This can be seen by disintegrating the measure
%
\begin{equation}
\label{obP}\obP:=\int_0^1\dint s\int\bP(
\dint\bT ) \delta_{(s,\bT_s)}
\end{equation}
on
$
\oBT:= \{(s,\bT_s)\dvtx  s\in[0,1], \bT_s\in\BB\TT_s \}
$;
cf. \cite{KRM}, Appendix 15.3. Later on, we will also consider the projections
$\obpr_W=\mathrm{id}\otimes\bpr_W$ that act on the second
coordinate of $\oBT$ as in (\ref{eqbprW}) and leave the first
coordinate untouched, and the\vspace*{1pt} projection images $\obP_W=\obP\circ
\obpr_W^{-1}$, where $W\in\PP$. We also introduce the notation $\oBT
_W:=\obpr_W(\oBT)$.

\subsection{Division kernels}\label{SECCellDivisionKernels}

Consider a random element $\bT$ of $\BB\TT_W$ for a window $W\in\PP
_\cup$. The process $(\bT_s)_{0\le s\le1}$ is then automatically
Markovian because its ``past'' is part of the ``present''. In this paper,
we will focus on the ``nice'' case in which the evolution of this Markov
process is described by a rate kernel that specifies the jump times and
transitions of $(\bT_s)_{0\le s\le1}$. Since the only transitions are
single-cell divisions by bi-coloured hyperplanes, this means that the
rate kernels take the following form.

\begin{definition}\label{defdivisionkernel}
A \emph{division kernel}
is a measure kernel $\Phi$ from the set
\[
\bigl\{(s,\bT_s,c)\in\oBT\times\CC\dvtx c\in T_s \bigr\}
\]
to $\HH$ such that each $\Phi(s, \bT_s, c, \cdot)$ is a finite
measure supported on $\langle c\rangle\subset\HH$.
If $\Phi$ is only defined for arguments in $\oBT_W\times\CC_W$,
$\Phi$ is called a \emph{division kernel
for the window $W\in\PP_\cup$}.
\end{definition}

In the following, it will be convenient to work also with the \emph
{cumulative division kernel}
%
\begin{equation}
\label{eqcumulkernel} \widehat\Phi(s, \bT_s, \cdot) = \sum
_{c\in T_s} \delta_c\otimes \Phi(s, \bT_s,
c, \cdot)
\end{equation}
from $\oBT$ to $\CC\times\HH$. Note that, conversely, $\Phi(s, \bT
_s, c, \cdot) = \widehat\Phi(s, \bT_s, \{c\}\times\cdot) $.

The next remark describes how a division kernel determines the
evolution of a BRT within a bounded window.

%
\begin{remark}[(Local BRTs with prescribed division kernels)]\label{remconstruction}
Let $W\in\PP_\cup$ be a fixed window, $\Phi_W$ be a division kernel
for $W$, and
\[
\hat\phi_W(s,\bT_s):=\widehat\Phi_W
\bigl(s, \bT_s, T_s\times\langle W\rangle \bigr)
\]
the finite total mass of the cumulative kernel $\widehat\Phi_W(s, \bT_s,
\cdot)$. We construct a random element $\bT$ of $\BB\TT_W$ as follows:
\begin{longlist}[(R)]
\item[(I)] Pick an initial tessellation $T_0\in\TT_W$ according to
some probability law $P_W$ on
$\TT_W$,
and let $s_0=0$ and $\bT_0:= T_0$. Also, let $i=1$ and proceed with
the following random recursion over the number~$i$.
\item[(R)] Suppose that $i\ge1$ and both a random time $s_{i-1} \in
[0,1]$ and a BRT $\bT_{s_{i-1}}\in\BB\TT_{s_{i-1}}$ are already
realised. Then take a random time $s_i\in(s_{i-1},\infty]$ with
``survival'' probability
%
\begin{equation}
\label{survivalprobability} \Prob(s_i>s) =\exp \biggl[- \int
_{s_{i-1}}^s \hat\phi_W(u\wedge 1,
\bT_{s_{i-1}})\,\dint u \biggr]
\end{equation}
%
for $s>s_{i-1}$.
If $s_i\le1$, proceed to define an extension $\bT_{s_i}\in\BB\TT
_{s_i}$ of $\bT_{s_{i-1}}$ as follows: pick a random cell $c_i\in
T_{s_{i-1}}$ and a bi-coloured hyperplane $H_i$ according to the law
\[
{\widehat\Phi_W(s_i,\bT_{s_{i-1}}, \cdot)} /{
\hat \phi _W(s_i,\bT_{s_{i-1}})}.
\]
(Note that the denominator does not vanish for each possible choice of
$s_i$.) Then let $T_s=T_{s_{i-1}}$ for $s\in(s_{i-1},s_i)$ and
$T_{s_i}=\oslash_{c_i,H_i}(T_{s_{i-1}})$, that is,
\[
\sD(\bT_{s_{i}})=\sD(\bT_{s_{i-1}})\cup \bigl\{(s_i,c_i,H_i)
\bigr\}.
\]
Next, let $i:=i+1$ and go to (R).
In the case $s_{i}>1$, let $T_s=T_{s_{i-1}}$ for $s\in(s_{i-1},1]$,
set $n=i-1$, and stop.
\end{longlist}

One needs to ensure that this algorithm terminates after finitely many
steps. It is not difficult to show that this is the case if
%
\begin{equation}
\label{eqbdkernel} \sup_{s,\bT_s,c}\Phi_W \bigl( s,
\bT_s,c,\langle c\rangle \bigr)=: \phi <\infty;
\end{equation}
see the proof of Lemma~\ref{lemmoment-bound} below.
This lemma shows further that the process $(\bT_s)_{0\le s\le1}$ can
be characterised as the unique,
in general time-inhomogeneous pure jump (i.e., piecewise constant)
Markov process in $\BB\TT_W$ with
initial distribution $P_{W}$ and generator
%
\begin{equation}
\label{LocGen} \qquad\LL_{W,s}^{\Phi_W} g(\bT_s) = \int
_{T_s\times\langle W\rangle} \widehat\Phi_W \bigl(s,\bT_s,
\dint(c,H) \bigr) \bigl[g \bigl(\oslash_{s,c,H}(\bT _s) \bigr)
- g(\bT_s) \bigr]
\end{equation}
at time $s\in[0,1]$. Here, $\oslash_{s,c,H}(\bT_s)\in\BB\TT_s$
is the branching tessellation that coincides with $\bT_s$ for times
less than $s$ and equals
$\oslash_{c,H}(T_s)$ at time $s$, and $g$ is any bounded measurable
function on $\BB\TT_{W}$.
The distribution of $\bT$ is a BRT $\bP_W$ in $W\in\PP$, and this
$\bP_W$ is called the \emph{BRT in $W$ with division kernel $\Phi_W$
and initial distribution $P_{W}$}.
\end{remark}

The main objects of this paper are BRTs on the full space $\RR^d$ that
can be characterised in a similar way as the local BRTs in the remark
above. Namely, for any division kernel $\Phi$ and $0\le s\le1$ we
define an operator $\LL^\Phi_s$ by
%
\begin{equation}
\label{eqglobalgen} \LL^\Phi_s g(\bT_s)=\int
\widehat\Phi \bigl(s,\bT_s,\dint(c,H) \bigr) \bigl[g \bigl(
\oslash_{s,c,H}(\bT_s) \bigr)-g(\bT_s) \bigr].
\end{equation}
Here, $\oslash_{s,c,H}$ is as in the preceding remark, and $g$ is any
bounded local function on~$\BB\TT$, where \emph{local} means that
$g$ is $\cB_W$-measurable for some $W\in\PP$.

%
\begin{definition}\label{defBRTforPhi}
For a given division kernel $\Phi$, we will say that a BRT $\bP\in
\sP$ \emph{evolves according to $\Phi$} if the Markov process $\bT
=(\bT_s)_{0\le s\le1}$ in $\BB\TT$ with~distribution $\bP$
satisfies the forward equation with generators $\LL^\Phi_s$, in that
%
\begin{equation}
\label{eqglobalforwardeq}
\int_0^t \dint s\int\dint
\bP_{s} \LL_{s}^{\Phi}g = \int g\, \dint
\bP_{t}-\int g \,\dint\bP_{0}
\end{equation}
for all $t\in[0,1]$ and all bounded local functions $g$ on $\BB\TT$.
\end{definition}

Obviously, this definition refers to a BRT $\bP$ as a process evolving
in time, by saying that the Markov process with distribution $\bP$
evolves just as the local processes in Remark~\ref{remconstruction},
in that a cell $c$ in\vspace*{1pt} environment $\bT_s$ at time $s$ is split by a
bi-coloured hyperplane $H$ with instantaneous intensity $\widehat\Phi
(s,\bT_s,c,\dint H )\,\dint s$. Later we will study the spatial
aspects of $\bP$.

%
\begin{figure}[t]

\includegraphics{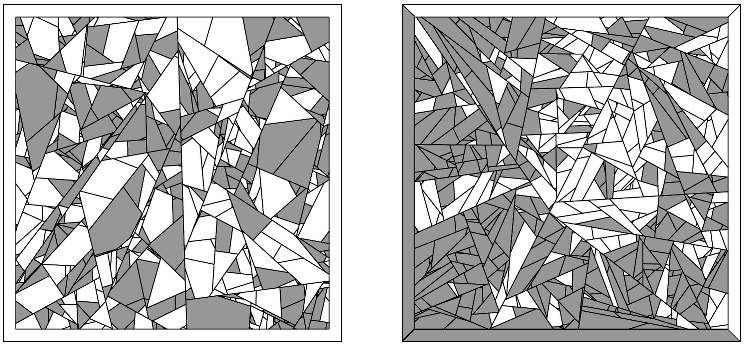}

\caption{Simulations of two BRTs with isotropic selection of lines,
two colours and the full window as single initial cell.
\emph{Left}: A STIT tessellation; colours are chosen at random.
\emph{Right}: Colour mutation and size balancing as in Example
\protect\ref{exmutations,balance,aging}, but without aging.
Here, $\varepsilon=0.025$, $\beta(\mathsf{s})=(1+\mathsf{s})/2$
and a mixed
boundary condition as indicated.}\label{Bild}
\end{figure}

\subsection{Examples of division kernels}

This section contains a few examples of division kernels; two
simulation pictures are shown in Figure~\ref{Bild}. The first is (by
now) classical and will be used as a reference model throughout the following.

%
\begin{example}[(STIT tessellations)]\label{exSTIT}
Let $\Lambda$ be a locally finite measure on the set $\HH$ of all
bi-coloured hyperplanes, which is invariant under all translations.
That is, under the identification of $H\in\HH$ with $(u,r,\sigma
^+,\sigma
^-)\in\SS^{d-1}_+\times\RR\times\Sigma^2$, $\Lambda$ can be
written in
the form
%
\begin{equation}
\label{eqLambda} \Lambda(\dint H)= \lambda(\dint u)\,\dint r \mu \bigl(u,\dint
\sigma ^+,\dint \sigma^- \bigr).
\end{equation}
Here, $\lambda$ is a measure on $\SS^{d-1}_+$,
and $\mu$ is a probability kernel from $\SS^{d-1}_+$ to $\Sigma^2$.
(The translation invariance is expressed by the fact that the
$r$-marginal is Lebesgue measure and $\mu$ does not depend on~$r$.)
A natural choice is the motion-invariant measure $\Liso$ for which
$\lambda$ is the normalised surface measure $\lambda_\iso$ on $\SS
^{d-1}_+$ and $\mu(u, \cdot) =\nu\otimes\nu$ for a reference
probability measure $\nu$ on $\Sigma$.
Then a STIT tessellation with driving measure $\Lambda$ is a BRT for the
division kernel
%
\begin{equation}
\label{DlaPhiLambda} \Lambda^{*}(s,\bT_s,c, \cdot):= \Lambda
\bigl( \cdot\cap \langle c\rangle \bigr).
\end{equation}
In the uncoloured case, this model has been introduced by Mecke, Nagel
and Wei\ss\ \cite{MNW,NW05}.
Since $\Lambda^{*}$ does not depend on the time $s$, the random
holding times
$s_i-s_{i-1}$ in Remark~\ref{remconstruction} above are exponentially distributed and can be
understood as minima over $c\in T_{s_{i-1}}$ of independent exponential
times with parameter $\Lambda(\langle c\rangle)$, which are
associated to the
presently existing cells. [In the isotropic case $\Lambda=\Liso$, the
parameter $\Lambda(\langle c\rangle)$ is precisely the mean width
of $c$.] In
other words, the tessellations evolve according to a continuous-time
branching process on $\CC_W$, $W\in\PP$, in which all cells $c$
behave independently of each other, live for an exponential time with
parameter $\Lambda(\langle c\rangle)$ and then split into two
parts according to
the conditional distribution $\Lambda( \cdot|\langle c\rangle
)$. In
particular, this implies that smaller cells live stochastically longer.

In view of this independence of the evolution in different cells, it is
clear that for each $T_0\in\TT$ there exists a unique
whole-space BRT $\bPi^\Lambda(T_0, \cdot)$, called STIT tessellation
of $\RR^d$, with initial tessellation $T_0$ and driving measure
$\Lambda$.
In fact, if the support of $\lambda$ contains a linear basis of $\RR
^d$, one can also construct a unique BRT $\bPi^{\Lambda,\infty
}=\bPi^\Lambda(T_0^\infty, \cdot)$ for the degenerate initial tessellations $T_0^\infty$
that consist of the single ``cell'' $\RR^d$ with any colour $\sigma$; see
\cite{MNW}, Theorem~1, and \cite{NW05}, Theorem~1. [By~(\ref{DlaPhiLambda}), $\bPi^{\Lambda,\infty}$ does not depend on
$\sigma$.]

Formally, $\bPi^\Lambda$ is a probability kernel from $\TT$ to $\BB
\TT
$. So, for each $P\in\sP(\TT)$,
$P\bPi^\Lambda=\int P(\dint T_0) \bPi^\Lambda(T_0, \cdot)$
is the unique
BRT for $\Lambda$ with initial distribution~$P$.
Its projections to arbitrary windows $W\in\PP$ are given by
%
\begin{equation}
\label{eqSTITlocalisation} \bigl(P\bPi^\Lambda \bigr)\circ\bpr_W^{-1}=
P_{W}\bPi^\Lambda_W
\end{equation}
for the restricted STIT kernel $\bPi^\Lambda_W(T_{W,0}, \cdot)$ from
$\TT_W$ to $\BB\TT_W$ with the restricted driving measure $\Lambda
(
\cdot\cap\langle W\rangle)$.
The abbreviation STIT stands for {st}ability under the operation of
{it}eration of tessellations. An explanation and further remarkable
properties can be found in \cite{MNW,NW05,RedenbachThaele,ST2,ST3,ST7,ST4,ST5,ST6} and \cite{TWN}.
\end{example}

A generalisation of the STIT models, which still keeps the independence
of the division process for distinct cells,
are the \emph{cell-driven BRTs}, which have division kernels of the form
%
\begin{equation}
\label{eqshape-driven} \Phi(s,\bT_s,c,\dint H) = \varphi(c,H) \Lambda(\dint
H)
\end{equation}
with a density function $\varphi(c,H)$ on $\CC\times\HH$ which vanishes
except when $H\in\langle c\rangle$. A~special case are the \emph
{shape-driven BRTs} investigated in \cite{ST4}; see also the examples therein.

The next example demonstrates the flexibility of modelling in the
present setting: it combines
an interaction between the colours of the cells with a geometric
homogenisation mechanism and an aging effect. The last feature takes
advantage of the fact that division kernels may also depend on the past.

%
\begin{example}[(Contact-induced mutations with size balancing and aging)]\label{exmutations,balance,aging}
Let the colour space be $\Sigma=\{-1,1\}$ and consider a division
kernel of the form
\[
\Phi(s,\bT_s,c,\dint H)= \varphi(c,\eta) \lambda_\iso(
\dint u)\,\dint r\, \mu \bigl(s,\bT_s,c, \dint\sigma^+ \bigr) \mu
\bigl(s, \bT_s,c, \dint\sigma^- \bigr),
\]
where $H=(\eta,\sigma^+,\sigma^-)$ with spatial part $\eta
=(u,r)\in\SS
^{d-1}_+\times\RR$ and colours \mbox{$\sigma^\pm\in\Sigma$}.
A special choice of the geometric pre-factor is
\[
\varphi(c,\eta)= \varepsilon \1_{\langle c\rangle}(\eta) +\varepsilon^{-1}
\1_{\langle\varepsilon
\star c\rangle}(\eta),
\]
for some small $\varepsilon>0$; here, $\varepsilon\star c=
m(c)+\varepsilon(c-m(c))$ is the $\varepsilon
$-retraction of $c$. This choice has the effect that the cutting
hyperplane will typically pass close to the midpoint $m(c)$ of $c$, so
that its two daughter cells have comparable size. One can further
choose the colouring rule
\[
\mu(s,\bT_s,c, \cdot)= \delta_{\col(c)} + \beta(\mathsf
{a}_{s,c,\bT_s},\mathsf{s}_{c,T_s}) \delta_{-\col(c)},
\]
where $\mathsf{a}_{s,c,\bT_s}=s-\min\{u\in[0,s]\dvtx c\in T_u\}$ is the
age of $c$ at time $s$,
\[
\mathsf{s}_{c,T_s}=\sum_{c'\in T_s\dvtx  \col(c')=-\col(c)} \Vol
_{d-1} \bigl(c\cap c' \bigr)/\Vol_{d-1}(\partial
c)
\]
is the opposite-type surface fraction (measured by the Hausdorff
measure of dimension $d-1$),
and $\beta\dvtx [0,1]^2\rightarrow(0,\infty)$ is a suitable positive function.
For instance, $\beta$ can be taken to be decreasing in $\mathsf{a}$ so
that increasing age reduces the willingness of splitting and mutating.
One can further let $\beta$ be increasing in $\mathsf{s}$. Then the
larger a cell's surface fraction is in contact with cells of opposite
type, the more the cell gets ``nervous'' and hurries to divide, and the
more likely it is that its daughter cells mutate to adapt their type to
that of the neighbours.
\end{example}

Our third example may seem somewhat exotic. It will be used in Remark
\ref{remuniqueness} to demonstrate that a BRT on the full space $\RR
^d$ is not necessarily uniquely determined by its initial distribution
and its division kernel.

%
\begin{example}[({Directional infinite-range interaction})]\label{exhor-vert}
This is an uncoloured model, for which $\Sigma$ is a
singleton.
We further confine ourselves to the planar case ${d=2}$. Let $\Lambda
_\hor
(\dint H)=\delta_{(0,1)}(\dint u)\,\dint r$ and $\Lambda_\vert(\dint
H)=\delta_{(1,0)}(\dint u)\,\dint r$ be the measures on $\HH=\SS
_+^1\times\RR$ for which all lines are horizontal, respectively, vertical.
For any cell $c\in\CC$ let $\diam_\hor(c)=\max_{x,y\in
c}|x_1-y_1|$ and $\diam_\vert(c)=\max_{x,y\in c}|x_2-y_2|$
be the horizontal and vertical diameters of $c$, where $x_i$ stands for
the $i$th coordinate of $x$. Also, let
\[
\CC_\hor= \bigl\{c\in\CC\dvtx \diam_\hor(c)>
\diam_\vert(c) \bigr\}
\]
be the set of all ``horizontal'' cells. Finally, writing $[n]$ for the
centred square of area $n^2$, let
\[
\rho_\hor(T)=\limsup_{n\to\infty} n^{-2} \bigl| \bigl
\{c\in T\cap \CC_\hor\dvtx m(c)\in[n] \bigr\} \bigr|
\]
be the upper density of horizontal cells for a tessellation $T\in\TT
$, and define $\rho_\vert(T)$ analogously. Then let
\[
\TT_\hor= \bigl\{T\in\TT\dvtx \rho_\hor(T) >
\rho_\vert(T) \bigr\}
\]
be the set of tessellations with a dominating fraction of horizontal
cells, and $\TT_\vert=\TT\setminus\TT_\hor$.
Consider the division kernel
%
\begin{equation}
\label{eqexhor-vert} \Phi(s,\bT_s,c, \cdot)= \1_{\TT_\hor}(T_s)
\Lambda_\hor \bigl( \cdot \cap\langle c\rangle \bigr)+
\1_{\TT_\vert}(T_s) \Lambda_\vert \bigl( \cdot\cap\langle
c\rangle \bigr).
\end{equation}
Since $\TT_\hor$ is invariant under translations and tail-measurable,
this $\Phi$ looks at the actual tessellation ``at infinity'' in order to
decide whether the cutting line should be horizontal or vertical.
\end{example}

\subsection{Gibbsian BRTs}

In this section, we introduce a Gibbsian perspective on~BRTs. As is
standard in the theory of Gibbs measures,
one aims at describing a macroscopic system by means of its local
conditional distributions that describe the behaviour inside a bounded
region when the remaining system is fixed. We first define such
conditional distributions in the context of BRTs. This will allow us
then to introduce Gibbsian BRTs.
Let $W\in\PP$ be a fixed window.

\subsubsection{Inner and outer projections}

Recall from (\ref{pr}) and (\ref{eqbprW}) that the projections $\pi
_W$ and $\bpr_W$ are defined by intersecting the cells with $W$, and
thus wipes off much information on the cell geometry (such as, e.g.,
the location of midpoints). To avoid this,
we introduce the ``\emph{inner}'' \emph{projection}
%
\begin{equation}
\label{eqinnerproj} \pi_W^\inn\dvtx \TT\to
\TT_W^\inn,\qquad T\mapsto T_W^\inn:=
\bigl\{c\in T\dvtx c\subset\operatorname{int}(W) \bigr\},
\end{equation}
which removes all cells which are not completely contained in the
interior of $W$.
It takes values in the set $ \TT_W^\inn$ of all possibly empty, not
necessarily connected collections of cells inside $W$ with pairwise
disjoint interiors. The counting variables $N_A$ in (\ref{eqNA}) are
even defined on $ \TT_W^\inn$ and generate a $\sigma$-field $\cB(
\TT
_W^\inn)$, for which $\pi_W^\inn$ is measurable. As the cells of
$T_W^\inn$ are even required to be contained in the interior of $W$,
$T_W^\inn$ is a measurable function of $T_W$.

In the same way, we define the \emph{inner projection}
%
\begin{equation}
\label{eqprojin} \bpr_W^\inn\dvtx \BB\TT\ni\bT\mapsto
\bT_W^\inn= \bigl(T_{W,s}^\inn
\bigr)_{0\le s\le1}
\end{equation}
on $\BB\TT$, where $T_{W,s}^\inn= \pi_W^\inn(T_s)$. Arguing as in
Remark~\ref{remBT}, one finds that
$\bT_W^\inn$ is uniquely determined by $T_{W,0}^\inn$, $\sD(\bT
_W^\inn)$, and the finite set of all
``immigration events'' $(s,c)$ with $T_{W,s}^\inn=T_{W,s-}^\inn\cup\{
c\}$. Consequently, one can generate a $\sigma$-field
on the range $\BB\TT_W^\inn$ of $\bpr_W^\inn$ by means of counting
variables similar to those in (\ref{bNA,B}),
so that $\bpr_W^\inn$ becomes measurable. Note also that $\bpr
_W^\inn=\bpr_W^\inn\circ\bpr_W$.

Complementary to the above, we also introduce an ``\emph{outer}''
\emph{projection for $W$} by
%
\begin{equation}
\label{eqouterproj} \pi_W^\out\dvtx \TT\ni T\mapsto
T_W^\out:=T\setminus T_W^\inn=
\bigl\{ c\in T\dvtx c\setminus\operatorname{int}(W)\ne\varnothing \bigr\},
\end{equation}
and a ``\emph{boundary}'' \emph{projection}
%
\begin{equation}
\label{eqboundaryproj} \pi_W^\partial\dvtx \TT\ni T\mapsto
\pi_W^\partial(T_W)= \{c\cap W\dvtx c\in T, c\cap
\partial W\ne\varnothing \}.
\end{equation}
Likewise, on the level of branching tessellations, we define
%
\begin{eqnarray}
\label{eqprojout} \bpr_W^\out \dvtx \BB\TT \ni\bT&\mapsto&\bT
_W^\out = \bigl(\pi_W^\out(T_s)
\bigr)_{0\le s\le1},
\\
\label{eqprojboundary} \bpr_W^\partial \dvtx \BB\TT \ni\bT &\mapsto& \bT
_W^\partial = \bigl(\pi_W^\partial(T_s)
\bigr)_{0\le s\le1} =\bpr_W^\out(\bT_{W}).
\end{eqnarray}
In the forest picture of Figure~\ref{Tree}, each $\bT_W^\out$ in the
range $\BB\TT_W^\out$ of $\bpr_W^\out$ corresponds to a forest of
binary trees from which all cells within $W$ are erased.
So, one can use the counting variables in (\ref{bNA,B}) to generate a
$\sigma$-field on $\BB\TT_W^\out$, and $\bpr_W^\out$
is then evidently measurable. The same applies to $\bpr_W^\partial$.
Furthermore, to keep the full information on the initial tessellation
in $\RR^d$, respectively, in $W$,
it will also be convenient to introduce the mappings
%
\begin{eqnarray}
\label{eqinitial+out} \bpr_W^{0,\out}\dvtx \bT &\mapsto& \bT_W^{0,\out
}:= \bigl(T_{W,0}^\inn,
\bT_W^\out \bigr),
\\
\label{eqinitial+partial} \bpr_W^{0,\partial}\dvtx \bT &\mapsto&\bT
_W^{0,\partial}:= \bigl(T_{W,0}^\inn,
\bT_W^\partial \bigr).
\end{eqnarray}
For each of the projections $\bpr_W^\ast$ in (\ref{eqprojin}),
(\ref{eqprojout}), (\ref{eqprojboundary}), (\ref
{eqinitial+out}) and (\ref{eqinitial+partial}), we write $\cB
_W^\ast
=\sigma(\bpr_W^\ast)$ for the $\sigma$-field on $\BB\TT$ that is
generated by this projection. By abuse of notation, we will use the
same symbol $\cB_W^\ast$ for the $\sigma$-field on the range of
$\bpr
_W^\ast$.

\subsubsection{Conditional BRTs}

Let $\bT\in\BB\TT$ any branching tessellation. Consider the
time-dependent ``inner'' window
%
\begin{equation}
\label{eqinnW} \inn_W \bigl(s,\bT_W^\partial
\bigr):= W\setminus\operatorname{int} \bigl(\cup \bigl\{c\dvtx c\in
T_{W,s}^\partial \bigr\} \bigr) = \cup \bigl\{c\dvtx c\in
T_{W,s}^\inn \bigr\},
\end{equation}
which is possibly empty and not necessarily connected. It is measurable
jointly in both arguments,
piecewise constant and right-continuous as a function of~$s$. Let
%
\begin{equation}
\label{eqjumptimes} 0<t_1=t_1 \bigl(\bT_W^\partial
\bigr) <\cdots<t_n=t_{n(\bT_W^\partial)} \bigl(\bT _W^\partial
\bigr)<1
\end{equation}
be the jump times of the path $s\mapsto\inn_W(s,\bT_W^\partial)$,
which depend measurably on $\bT_W^\partial$.
[Note that possibly $n(\bT_W^\partial)=0$. For the sake of
convenience, we also exclude the case that there is a jump at time $1$,
which occurs with probability zero.] At each $t_i$, $\bT_W^\partial$
creates a new cell $c_i$ inside $W$, namely
\[
c_i =c_i \bigl(\bT_W^\partial
\bigr):= \mathrm{cl} \bigl(\inn_W \bigl(t_i,\bT
_W^\partial \bigr)\setminus\inn_W
\bigl(t_{i-1},\bT_W^\partial \bigr) \bigr),
\]
where $t_0=0$. In other words, $\bT_{W}^\partial$ induces a process
of immigration of cells into~$W$.

%
\begin{definition}\label{defcondGibbs} Let $\Phi$ be a division
kernel and suppose that the following random process $\bS=(S_s)_{0\le
s\le1}$ with $\bS\cup\bT_{W}^\out:=(S_s\cup T_{W,s}^\out)_{0\le
s\le1}\in\BB\TT$ is well defined:
\begin{itemize}[--]
\item[--] Let $\bS_{[0,t_1)}=(S_s)_{0\le s<t_1}$ be the BRT in the
window $\inn_W(0,\bT_W^\partial)$ with time interval $[0,t_1)$,
initial tessellation $S_0=T_{W,0}^\inn$ and division kernel
\[
\Phi_W^\inn \bigl(s,\bS_s,c, \cdot|
\bT_W^\out \bigr):=\Phi \bigl(s,\bS_s\cup
\bT_{W,s}^\out,c, \cdot \bigr)
\]
for $c\in S_s, s\in[0,t_1)$.
\item[--] For $i=1,\ldots,n$ and conditional on $\bS_{t_i-}$ let
$\bS_{[t_i,t_{i+1})}=(S_s)_{t_i\le s<t_{i+1}}$ be the BRT in the
window $\inn_W(t_i,\bT_W^\partial)=c_i\cup\inn_W(t_{i-1},\bT
_W^\partial)$ with time interval $[t_i,t_{i+1})$, initial tessellation
$S_{t_i}=S_{t_i-}\cup\{c_i\}$ and division kernel
\[
\Phi_W^\inn \bigl(s,\bS_s,c, \cdot|
\bT_W^\out \bigr):=\Phi \bigl(s,\bS_s\cup
\bT_{W,s}^\out,c, \cdot \bigr)
\]
for $c\in S_s, s\in[t_i,t_{i+1})$.
Here $t_{n+1}=1$, and we finally set $S_1:=S_{1-} $.
\end{itemize}
The distribution of $\bS$ on $\BB\TT_W^\inn$ will be denoted by
$\bG_W^\Phi( \cdot|\bT_W^{0,\out})$ and is called the \emph
{conditional BRT for $\Phi$ in $W$ with initial tessellation
$T_{W,0}^\inn$ and boundary condition~$\bT_W^{\out}$}.
\end{definition}

By construction, $\bG_W^\Phi$ is a probability kernel from $(\BB\TT
_W^{0,\out},\cB_W^{0,\out})$ to
$(\BB\TT_W^\inn, \cB_W^\inn)$.

%
\begin{example}[({Conditional STIT tessellations})]\label{excondSTIT}
As in Example~\ref{exSTIT}, let $\Lambda$ be a locally finite
measure on
$\HH$
and $\Lambda^{*}$ be the associated division kernel; cf. (\ref
{DlaPhiLambda}). Then 
$\bG_W^\Lambda( \cdot|\bT_W^{0,\out}):=\bG_W^{\Lambda
^{*}}( \cdot|\bT
_W^{0,\out})$ is simply the distribution of
\[
\bigcup_{c\in T_{W,0}^\inn}\bS^{(c)}\cup\bigcup
_{i=1}^n \bS^{(i)}
\]
for independent random STIT tessellations $\bS^{(c)}$ and $\bS^{(i)}$
for $\Lambda$. Here, $\bS^{(c)}$ evolves in time $[0,1]$ from the
single-cell tessellation $S_0^{(c)}=\{c\}$ of the initial polytope
$\mathrm{sp}(c)$, whereas $\bS^{(i)}$ evolves in time $[t_i,1]$ from
the single-cell initial tessellation $S_{t_i}^{(i)}=\{c_i\}$ of the
``immigrated'' polytope $\mathrm{sp}(c_i)$ and is extended to the full
interval $[0,1]$ by setting $S^{(i)}_s=\varnothing$ for $s\in[0,t_i)$.
Since $\Lambda^{*}$ does not depend on the surrounding tessellation, it
follows that the measure $\bG_W^\Lambda( \cdot|\bT_W^{0,\out})$
depends only on~$\bT_W^{0,\partial}$.
\end{example}

Here is the natural counterpart of the concept of (macroscopic) Gibbs
measures in our setup of branching random tessellations.

%
\begin{definition}\label{defGibbs}
Let $\Phi$ be any division kernel. A BRT $\bP\in\sP$ is called a
\emph{Gibbsian BRT for $\Phi$} if, for all $W\in\PP$,
$\bG_W^\Phi$ is a regular version of its conditional probability
given $\cB_W^{0,\out}$. More explicitly, this means that
\[
\int f \,\dint\bP=\int\bP(\dint\bT)\int\bG_W^\Phi \bigl(
\dint\bS | \bT_W^{0,\out} \bigr) f \bigl(\bS\cup
\bT_{W}^\out \bigr)
\]
for all bounded measurable functions $f$ on $\BB\TT$ and all $W\in
\PP$.
\end{definition}

In contrast to Definition~\ref{defBRTforPhi} in which a BRT is
considered as a process in time, the preceding definition emphasises
the spatial aspects of a BRT, by saying that $\Phi$ describes the cell
splitting mechanism within an arbitrary local window when the evolution
of all other cells is given.

\subsection{Translation invariance}

A main focus of this paper is on BRTs that are invariant under spatial
translations. For each $x\in\RR^d$, the translation
$\vartheta_x$ by the vextor $-x$ acts:
\begin{itemize}[--]
\item[--] on cells $c\in\CC$ via $\vartheta_x\dvtx c\mapsto c-x:=(\mathrm
{sp}(c)-x,\col(c))$,
\item[--] on bi-coloured hyperplanes $H\in\HH$ via
\[
\vartheta_x\dvtx H\mapsto H-x:= \bigl(\mathrm{sp}(H)-x,\col^+(H),
\col^-(H) \bigr),
\]
\item[--] on tessellations $T\in\TT$ via
\[
\vartheta_x\dvtx T\mapsto T-x:=\{c-x\dvtx c\in T\},
\]
%
%
\item[--] on branching tessellations $\bT=(T_s)_{0\le s\le1}\in\BB\TT$ via
\[
\vartheta_x\dvtx \bT\mapsto\bT-x:=(T_s-x)_{0\le s\le1}.
\]
\end{itemize}
That is, only the spatial coordinates are shifted, but the colours
remain unchanged. Moreover, by abuse of notation
we use the same symbol $\vartheta_x$ for the translation on each
level, and we will also use it for the simultaneous translation of
pairs of objects as above.

%
\begin{definition}
A BRT $\bP\in\sP$ is called \emph{translation
invariant} if it is invariant under the action of the translation group
$\Theta=(\vartheta_x)_{x\in\RR^d}$ on $\BB\TT$, in that $\bP
\circ
\vartheta_x^{-1}=\bP$ for all $x\in\RR^d$.
We write $\sP_\Theta=\sP_\Theta(\BB\TT)$ for the set of all translation
invariant BRTs that satisfy the first-moment condition (\ref
{eqfirstmoment}), which by translation invariance is equivalent to the
requirement that the ``hitting intensity''
%
\begin{equation}
\label{eqiP} i_1(\bP):= \int\bP(\dint\bT) |T_{[1],1}|
\end{equation}
is finite. Here, $[1]:=[-1/2,1/2]^d$ stands for the centred unit cube.
\end{definition}

Translation invariance allows to investigate the behaviour of a random
tessellation ``around a typical cell'', which for convenience is located
``around the origin''. This is formalised by means of Palm calculus as
presented in \cite{KRM}, Chapter~12, and~\cite{SW}, Theorem~4.1.1.
Let $\bP\in\sP_\Theta$ be given. Then the \emph{Campbell measure} of
$\bP$ on $\BB\TT\times\CC$ is defined by
%
\begin{equation}
\label{Campbell} \bC^\bP= \int\bP(\dint\bT) \sum
_{c\in T_1} \delta_{(\bT,c)}.
\end{equation}
It captures the joint distribution of the (terminal) cells and the
complete history of their surrounding tessellation.
The Palm calculus now states that there exists a finite measure $\bP
^0$ on
$\BB\TT\times\CC_0$, the so-called \emph{Palm measure} of $\bP$,
such that the \emph{Palm formula}
%
\begin{eqnarray}
\label{Palm} && \int\dint\bC^\bP(\bT,c) f \bigl(m(c),c- m(c),
\bT-m(c) \bigr)
\nonumber\\[-8pt]\\[-8pt]
&&\qquad = \int\dint x \int\dint\bP^0(\bT,c) f(x,c,\bT)\nonumber
\end{eqnarray}
holds for any nonnegative measurable function $f$ on $\RR^d\times\CC
_0\times\TT$. Its normalised marginal on $\CC_0$ is called the \emph
{typical cell distribution}.

Later on, we will often consider the integral over time $s$ of the
Campbell measure and the Palm measure of the projected BRTs $\bP_s$,
and it will be convenient to have a shorthand notation for these
objects. So, we define the \textit{extended Campbell measure}
%
\begin{equation}
\label{eqextCampbell} \obC^{\bP} = \int_0^1\dint s \int\bP_s(\dint\bT_s) \sum
_{c\in T_s} \delta_{(s,\bT_s,c)}
\end{equation}
and the \textit{extended Palm measure}
%
\begin{equation}
\label{eqextPalm} \obP^0 = \int_0^1
\dint s \int\dint\bP^0_s(\bT_s,c) \delta
_{(s,\bT_s,c)}.
\end{equation}
For $W\in\PP$, we similarly define the \textit{extended local
Campbell measure}
%
\begin{equation}
\label{eqextlocCampbell} \obC^{\bP_W} = \int_0^1
\dint s \int\bP_{W,s}(\dint\bT_{W,s}) \sum
_{c\in T_{W,s}} \delta_{(s,\bT_{W,s},c)}.
\end{equation}
Also, we will often use the \emph{time-integrated} version of the
{Palm formula} (\ref{Palm}), where the Campbell measure $\bC^\bP$
and the Palm measure $\bP^0$ are replaced by their extended relatives
$\obC^\bP$ and $\obP^0$, respectively.
For example, combining the time-integrated Palm formula with the
first-moment condition (\ref{eqiP})
we find that the total mass of $\obP^0$ can be estimated by
%
\begin{equation}
\label{eqextPalmmass} \bigl\|\obP^0\bigr\| = \int\dint\obP(s, \bT_s) \bigl|
\bigl\{c\in T_s\dvtx m(c)\in[1] \bigr\} \bigr|\le i_1(\bP)<
\infty.
\end{equation}

We conclude this section with some comments on random, but not
branching, tessellations $P\in\sP(\TT)$. These can be considered as
BRTs by identifying the space $\TT$ with $\BB\TT_0$. In particular,
it is then clear what translation invariance means, and we can
introduce the set $\sP_\Theta(\TT)$ of all translation invariant random
tessellations $P$ that satisfy the first-moment condition
%
\begin{equation}
\label{eqi0P} i_0(P):= \int P(\dint T) |T_{[1]}| <\infty.
\end{equation}
Since $i_0(\bP\circ\bpr_0^{-1})\le i_1(\bP)$, the initial
distribution of each $\bP\in\sP_\Theta$ satisfies (\ref{eqi0P}).

\section{Results}\label{secResults}

Most of our results use a STIT tessellation as a reference model.
Therefore, we fix throughout a locally finite reference measure
$\Lambda$
on $\HH$ which is invariant under translations. Moreover, we write
$\bPi^\Lambda(T_0, \cdot)$ for the associated STIT kernel, as
introduced ibidem.

\subsection{The role of division kernels for BRTs}\label{sec31}

Definitions~\ref{defBRTforPhi} and~\ref{defGibbs} provide two
ways of describing how a BRT $\bP$ may depend on a division kernel
$\Phi$, by considering either the evolution in time or the division of
cells in space. Our first result implies that these two descriptions
are equivalent.

%
\begin{theorem}\label{thmequivalence}For each $\bP\in\sP$ and
every cell division kernel $\Phi$,
the following statements are equivalent.
\begin{longlist}[(a)]
\item[(a)] $\bP$ evolves according to $\Phi$ as specified in Definition
\ref{defBRTforPhi}.
\item[(b)] $\bP$ is Gibbsian for $\Phi$ in the sense of Definition~\ref
{defGibbs}.
\item[(c)] For all nonnegative measurable functions $f$ on $\oBT\times\CC
\times\HH$,
%
\begin{eqnarray*}
&&\int\bP(\dint\bT) \sum_{(s,c,H) \in\sD(\bT)} f(s,\bT
_{s-},c,H)
\\
&&\qquad =\int\dint\obP(s,\bT_s)\int\widehat\Phi \bigl(s,
\bT_s, \dint(c, H) \bigr) f(s,\bT_{s},c,H).
\end{eqnarray*}
\end{longlist}
\end{theorem}

If the above properties (a) to (c) hold, we will simply say that \emph
{$\bP$ admits the division kernel $\Phi$}, or that \emph{$\Phi$ is
a division kernel for $\bP$}. While statements (a) and (b) elucidate
the temporal and spatial roles of $\Phi$,
the equivalent statement (c) provides a characterisation of the ``jump
intensity measure'' of $\bP$ in terms of $\Phi$. In particular, one
finds that the division kernel of the (unconditioned) marginal process
in a local window $W$
is obtained by a natural averaging over the possible environments
outside~$W$. To state this fact, we recall that the extended measure
$\obP$ and the extended projections $\obpr_W$ have been introduced in
and after (\ref{obP}).
Further, we will need a projection that refers to the cell division
procedure. Namely, for $W\in\PP$ we introduce the projection
%
\begin{equation}
\label{eqtpr} \tpr_W\dvtx (c,H)\mapsto(c\cap W,H)
\end{equation}
on $\CC\times\HH$, which for each $T\in\TT$ maps the set
\[
\tpr_W^{-1}\Delta_W:= \bigl\{(c,H)\dvtx c\in
\CC, H\in\langle c\cap W\rangle \bigr\} 
\]
onto $\Delta_W:=  \{(c,H)\dvtx  c\in\CC_W, H\in\langle c\rangle
\}$.

\begin{corollary}\label{corPhiWfromPhi}
If\vspace*{2pt} a BRT $\bP\in\sP$ admits a cell division kernel $\Phi$, its
projection $\bP_W$ to a window $W\in\PP$
is a BRT in $W$ for the cumulative division kernel $\widehat\Phi_W$,
which is defined as a regular version of the conditional measure
\[
\widehat\Phi_W(s,\bT_{W,s},B):= \EE_{\obP}
\bigl[ \widehat\Phi \bigl( \cdot, \cdot,\tpr_W^{-1}B \bigr)
| \obpr_W=(s,\bT_{W,s}) \bigr].
\]
Here, $B$ is any measurable subset of $\Delta_W$.\vadjust{\goodbreak}
\end{corollary}

Next, we ask for conditions under which a given BRT $\bP\in\sP$
admits a division kernel~$\Phi$.
(The converse question of whether a BRT for a given division kernel
exists will be addressed in Theorem~\ref{thmexistence}.)
As we will see, this is the case whenever $\bP$ is
{l}ocally {a}bsolutely {c}ontinuous with respect to the
STIT model $P\bPi^\Lambda$ with initial distribution $P=\bP\circ
\pi
_0^{-1}$, in that
{\renewcommand{\theequation}{LAC}
\begin{equation}
\label{LAC} 
\bP_W\ll P_{W}
\bPi^\Lambda_W\qquad\mbox{for all $W\in\PP$;}
\end{equation}
}%
recall that $P_{W}\bPi_W^\Lambda=(P\bPi^\Lambda)\circ\bpr
_W^{-1}$ by (\ref{eqSTITlocalisation}).

\setcounter{equation}{1}
We note in passing that (LAC) also implies that the realisations of
$\bP$ almost surely exhibit a ``tame'' geometry. Namely, in the planar
case, they show exactly one type of vertices, the so-called $T$-vertices,
at which an endpoint of a line segment hits an inner point of another
line segment (provided this holds already for the initial
tessellation); see \cite{NW05,RedenbachThaele} and the references
cited therein.

%
\begin{theorem}\label{thmexPhi} For each $\bP\in\sP$ satisfying {(\ref{LAC})}
there exists a division kernel $\Phi$ for $\bP$. Moreover, if
$\bP$ is also invariant under translations, one can achieve that $\Phi
$ is covariant in the sense that
%
\begin{equation}
\label{eqhomogen} \widehat\Phi \bigl(s,\bT_s,\vartheta_x^{-1}
\cdot \bigr)=\widehat\Phi(s,\bT _s-x, \cdot)
\end{equation}
for all $x\in\RR^d$ and all $(s,\bT_{s})\in\oBT$.
\end{theorem}

Stated differently, the preceding theorem says that every $\bP\in\sP
$ satisfying {(\ref{LAC})} is Gibbsian for some $\Phi$. This is analogous to
similar results in standard Gibbs theory (cf.~\cite{GEO}, Theorem~2.30, or \cite{Israel}, Theorem~V.2.2a).
We note further that, by Corollary~\ref{corPhiWfromPhi}, the
covariance property (\ref{eqhomogen}) implies that also
the local division kernels can be chosen to be covariant in the sense that
%
\begin{equation}
\label{eqlocalhomogen} \widehat\Phi_{W} \bigl(s,\bT_{W,s},
\vartheta_x^{-1} \cdot \bigr)=\widehat\Phi
_{W-x}(s, \bT_{W,s}-x, \cdot)
\end{equation}
for all $x\in\RR^d$, $(s,\bT_{W,s})\in\oBT_{W}$ and $W\in\PP$.

\subsection{The inner entropy density}\label{EntrSec}

We now turn to a ``thermodynamic'' investigation of translation invariant
BRTs. Our goal in this subsection is an appropriate notion of entropy.
Recall that the relative entropy, or Kullback--Leibler divergence,
between two probability measures
$\mu$ and $\nu$ on a common measurable space is defined to be
$ \cH(\mu;\nu) = \int\log f \,\dint\mu$
if $\mu\ll\nu$ with Radon--Nikodym density $f$, and $+\infty$ otherwise.
It can also be written in the form
%
\begin{equation}
\label{relent} \cH(\mu;\nu) = \int\varrho(f)\,\dint \nu,
\end{equation}
where $\varrho$ is the nonnegative convex function
%
\begin{equation}
\label{rhofunctiondef} \varrho\dvtx a\mapsto1-a+ a\log a.
\end{equation}
The formula (\ref{relent}) readily shows that $\cH(\mu;\nu)\ge0$
with equality precisely when $\mu=\nu$.
We can also take it as the definition of relative entropy in the more
general case
when $\mu$ and $\nu$ are finite, not necessarily normalised measures.

Further, if $\cA$ is a sub-$\sigma$-field of the underlying $\sigma$-field
then the \emph{conditional relative entropy given $\cA$} is defined as
%
\begin{equation}
\label{eqcondrelent} \cH(\mu;\nu|\cA) = \int\cH \bigl(\mu_\cA( \cdot|x);
\nu_\cA ( \cdot|x) \bigr) \mu(\dint x),
\end{equation}
where $\mu_\cA( \cdot|x)$ and $\nu_\cA( \cdot|x)$ are
conditional measure kernels given $\cA$ for $\mu$ and $\nu$,
respectively (provided such kernels exist).

In our setup, we take the STIT model for $\Lambda$ as our reference measure
and introduce an ``inner'' entropy as follows.
Recall the definition (\ref{eqinitial+partial}) of $\bpr
_W^{0,\partial}$ and its associated $\sigma$-field
$\cB_W^{0,\partial}=\sigma(\bpr_W^{0,\partial})$, and Example
\ref
{excondSTIT} for the definition of the kernel $\bG_W^\Lambda$.

\begin{definition}\label{definnerentropy}
Let $\bP\in\sP$ be a BRT and $W\in\PP$. The \emph{inner entropy
of $\bP$ in $W$} is then defined by
%
\begin{equation}
\label{eqinnentr} \cH_W^\inn(\bP):= \cH \bigl(
\bP_W;P_{W,0}\bPi^\Lambda_W|\cB
_W^{0,\partial
} \bigr)=\cH \bigl(\bP_W;
\bP_W^{0,\partial} \otimes\bG_W^\Lambda \bigr).
\end{equation}
(According to physical convention we should add a minus sign, but here
we prefer to ignore this convention.)
\end{definition}

So, the attribute ``inner'' means that this entropy compares the
evolution of $\bP$ with that of the STIT model only for those
cells that are completely contained in~$W$, while the evolution of all
other cells hitting $W$ is ignored. The idea of using a conditional,
``inner'' entropy without boundary effects has been exploited before by
F\"ollmer and Snell \cite{FoeSn}
in the setup of Gibbs measures on general graphs.

Next, let $[n]:= [-n/2,n/2]^d$ denote the closed centred cube of
volume $n^d$.
For a translation invariant BRT $\bP\in\sP_\Theta$, one expects that
the limiting inner entropy per unit volume
\[
\lim_{n\to\infty}n^{-d} \cH_{[n]}^\inn(
\bP)
\]
exists, which is then called the \emph{inner entropy density} of $\bP
$ (relative to the reference STIT cutting rule $\Lambda$).
Indeed, our result is the following; see (\ref{eqextPalm}) for the
definition of the extended Palm measure $\obP^0$.

%
\begin{theorem}\label{thmentropy}
For each $\bP\in\sP_\Theta$, there exists the possibly infinite limit
\[
h^\inn(\bP):= \lim_{n\to\infty} n^{-d}
\cH_{[n]}^\inn(\bP).
\]
If this limit is finite, $\bP$ admits a translation covariant division
kernel $\Phi$, and
%
\begin{eqnarray}
\label{eqentropy} h^\inn(\bP)&=& \cH \bigl(\obP^0\otimes\Phi;
\obP^0\otimes\Lambda ^{*} \bigr)
\nonumber\\[-8pt]\\[-8pt]
&=& \int\dint\obP^0(s,\bT_s,c) \cH \bigl(\Phi(s,
\bT_s,c, \cdot); \1 _{\langle c\rangle
}\Lambda \bigr).\nonumber
\end{eqnarray}
\end{theorem}

So, the inner entropy density $h^\inn(\bP)$ is the conditional
relative entropy of its division kernel $\Phi$
with respect to $\Lambda^{*}$ when the branching\vspace*{1pt} tessellation and its
cell are
selected according to the extended Palm measure $\obP^0$.
In particular, if $h^\inn(\bP)$ is finite then the division kernel
$\Phi$ of $\bP$ admits a Radon--Nikodym density with respect to
$\Lambda^{*}$.

It is natural to expect that the relative entropy density is affine and
lower semi-continuous with compact level sets, at least under some
natural caveats. We show this for a topology that is finer than the
common weak topology, but is not metrisable.
Namely, we define the \emph{topology $\tau_\loc$ of local
convergence on} $\sP$
as the coarsest topology for which the mapping $\bP\mapsto\int f\,\dint\bP$
is continuous for every bounded local function $f$. It is then clear
that $\sP_\Theta$ is closed in~$\sP$.
Recalling the definition~(\ref{eqiP}) of the hitting intensity
$i_1(\bP)$,
we can then state the following.

%
\begin{theorem}\label{thmlevelsets}
The inner entropy density $h^\inn$ is affine and lower semi-continuous
in $\tau_\loc$. Moreover, for any two constants $0\le\beta,\gamma
<\infty$ and
every $P\in\sP_\Theta(\TT)$, the restricted level set
\[
\sP_{\Theta,P,\beta,\gamma}:= \bigl\{\bP\in\sP_{\Theta}\dvtx \bP \circ\bpr
_0^{-1}=P, i_1(\bP)\le\beta, h^\inn(
\bP)\le\gamma \bigr\}
\]
is compact and sequentially compact in $\tau_\loc$.
\end{theorem}

\subsection{Variational principle and existence}

Here, we change our perspective: rather than describing a given BRT in
terms of its division kernel $\Phi$, we will now suppose that a ``nice''
division kernel $\Psi$ is given in advance. As we will see, $\Psi$~gives rise to an ``inner energy'' functional on $\sP_\Theta$, and further
to an associated ``inner free energy'',
which in turn leads to a variational principle and an existence proof
for BRTs with division kernel $\Psi$.
Here are the conditions on $\Psi$ we need.

%
\begin{definition}\label{defmoderate}
Let us call a division
kernel $\Psi$ \emph{moderate} if there exists a measurable density
function $\psi$ on the set
\[
\bigl\{ (s,\bT_s,c, H)\dvtx 0\le s\le1, \bT_s \in\BB
\TT_s, c\in T_s, H\in{\langle c\rangle} \bigr\}
\]
satisfying
\[
\Psi(s,\bT_s,c,\dint H)= \psi(s,\bT_s,c, H)
\1_{\langle
c\rangle}(H) \Lambda(\dint H)
\]
such that the following holds for all arguments:
\begin{longlist}[(M4)]
\item[(M1)] $\psi$ is \emph{covariant} under translations, in that
\[
\psi(s,\bT_s,c, H)= \psi(s,\bT_s-x,c-x, H-x)
\]
for all $x\in\RR^d$.
\item[(M2)] $\psi$ has \emph{bounded range}, meaning that there
exists a constant $0\le\br=\br_\Psi<\infty$ such that
$\psi(s,\bT_s,c,H) =\psi(s,\bT_s',c,H)$ whenever $\bT_{c+ B_{\br
},s}=\bT_{c+ B_{\br},s}'$.
Here, $B_{\br}$ stands for the closed centred ball with radius $\br$.
\item[(M3)] $\psi$ is \emph{bounded and bounded away from zero},
that is, there exists
a constant $\kappa_\Psi<\infty$ such that $|\log\psi|\le\kappa
_\Psi$.
\item[(M4)]
$\Psi$ is \emph{approximately STIT for large cells}, which is to say
that there exists a constant $\kappa_\Psi'<\infty$ such that
\[
\int_{\langle c\rangle}\Lambda(\dint H) \bigl|\psi(s,\bT _{s},c,H)-1
\bigr| \le\kappa_\Psi'.
\]
[In view of the boundedness assumption \textup{(M3)}, this condition involves
only the cells $c$ for which $\Lambda(\langle c\rangle)$ is large.]
\end{longlist}
\end{definition}

To give some understanding of these assumptions, we set up an analogy
with the unbounded spin systems of classical statistical mechanics. A
branching tessellation $\bT_s$ at some time $s$ may be viewed as a
collection of (unbounded) ``spins'' that consist of cells together with
their prospective cutting hyperplanes and are located at the sites
$m(c)$, $c\in T_s$, of $\RR^d$. The interaction energy of a ``spin''
$(c,H)$ at time $s$ with its surrounding tessellation $\bT_s$ is given
by $-\log\psi(s,\bT_s,c, H)$. Assumption~(M1) then expresses a
natural spatial homogeneity, and (M3) the uniform boundedness of the
local energies. Assumption~(M2) stipulates that the range of interaction is
bounded---in the units of real space, not in the units of the graph of
sites which is random and difficult to handle. In particular, $ \psi
(s,\bT_s,c, H)$ may depend at least on the evolution of all cells
completely inside the $r$-neighbourhood $c+B_r$ of $c$ (which typically
contains most adjacent cells if $r$ is chosen large enough). It may
also depend on the colours of all cells that hit but are not contained
in $c+B_r$; this is because the colour remains unchanged if a cell is
intersected with a region. Finally, (M4) means that the interacting
system is close to the noninteracting reference system unless the
``spins'' are suitably confined. This type of assumption is quite common
for interacting systems of unbounded spins; we need it also here,
although it excludes the possibility that $\Psi$ is scale-invariant.

Obviously, the STIT kernel $\Psi=\Lambda^{*}$ of Example~\ref
{exSTIT} is moderate.
More generally, assumptions (M1)--(M3) hold for the cell-driven
division kernels in (\ref{eqshape-driven})
whenever the density $\varphi$ there is uniformly bounded from above and
away from zero;
(M4) can be achieved by setting $\Psi=\Lambda^{*}$ for cells with
large radius.
In Example~\ref{exmutations,balance,aging}, (M1) trivially holds,
(M2) holds for each $r>0$, and (M3) follows from the assumptions on
$\beta
$ stated there. Example~\ref{exhor-vert} violates the bounded-range
property (M2) in the most extreme way conceivable.

A moderate division kernel induces a functional on $\sP_\Theta$ which,
in analogy to the standard Gibbs theory, may be called
the (negative) \emph{inner energy in $W$ for $\Psi$}, and is defined by
%
\begin{equation}
\label{eqinnenergyW} \cU^\inn_W(\bP;\Psi)=\int\bP(\dint\bT)\sum
_{(s,c,H)\in\sD
(\bT)\dvtx  c\subset W} \log\psi(s,\bT_{s-},c,H);
\end{equation}
$\bP\in\sP_\Theta$.
(Note that in statistical mechanics the energy always appears
negatively in the exponent; so it should not be surprising that $\log
\psi$
shows up here. But we suppress the minus sign.) Likewise, there is a
term which comes from a normalisation (in our case of the distribution
of jump times),
and thus may be considered as an analog of the pressure in statistical
mechanics. In the present setup, however, this quantity is not only a
functional of $\Psi$, but also of the BRTs $\bP\in\sP_\Theta$, namely,
%
\begin{equation}
\label{eqpressureW} \qquad\cV^\inn_W(\bP;\Psi)=\int\dint\obP(s,
\bT_s)\sum_{c\in T_s\dvtx
c\subset W}\int_{\langle c\rangle}
\Lambda(\dint H) \bigl(\psi (s,\bT _{s},c,H)-1 \bigr).
\end{equation}

%
\begin{theorem}\label{thmenergydensity}
For every moderate division kernel $\Psi$ and every $\bP\in\sP
_\Theta
$ admitting a covariant division kernel $\Phi$, the following finite
limits exist and can be identified:
\begin{eqnarray*}
u^\inn(\bP;\Psi) &:=& \lim_{n\to\infty} n^{-d}
\cU _{[n]}^\inn (\bP_{[n]}; \Psi) =\int\log\psi\,\dint\obP^0 \otimes\Phi,
\\
v^\inn(\bP;\Psi) &:=&\lim_{n\to\infty} n^{-d}
\cV _{[n]}^\inn (\bP_{[n]}; \Psi) =\int(\psi-1)\,\dint \obP^0 \otimes\Lambda^{*}.
\end{eqnarray*}
In particular, $|u^\inn(\bP;\Psi)|\le\kappa_\Psi i_1(\bP)$ and
$|v^\inn(\bP;\Psi)|\le\kappa_\Psi' i_1(\bP)$.
\end{theorem}

The energy terms above can be combined with the inner entropy density
to define the \emph{inner excess free energy density of $\bP$ for
$\Psi$}, namely,
%
\begin{equation}
\label{eqfreeenergy} h^\inn(\bP;\Psi):=h^\inn(
\bP)-u^\inn(\bP;\Psi)+v^\inn(\bP;\Psi),
\end{equation}
where the right-hand side is set equal to $+\infty$ if $h^\inn(\bP
)=+\infty$. In fact, in the finite case it will turn out that
%
\begin{equation}
\label{eqrelentrdensityPsi} h^\inn(\bP;\Psi)= \int\dint\obP^0(s,
\bT_s,c) \cH \bigl(\Phi(s,\bT_s,c, \cdot); \Psi(s,
\bT_s,c, \cdot) \bigr),
\end{equation}
where $\Phi$ is a division kernel for $\bP$.
The following variational principle for BRTs is then immediate.

%
\begin{theorem}\label{thmvarprinciple}
Let $\Psi$ be any moderate division kernel. A BRT $\bP\in\sP
_\Theta$
then admits $\Psi$ as its division kernel if and only if $h^\inn(\bP;\Psi)=0$.
\end{theorem}

In particular, this can be used to prove the following result.\vadjust{\goodbreak}

%
\begin{theorem}\label{thmexistence}
For any moderate division kernel $\Psi$ and every $P\in\sP_\Theta
(\TT
)$, there exists a translation invariant BRT $\bP\in\sP_\Theta$ with
initial distribution $P$ and division kernel $\Psi$.
\end{theorem}

There is a large variety of initial random tessellations $P$ to which
this existence theorem applies. The most common examples are the
Poisson--Voronoi tessellation, the Poisson--Delaunay tessellation, and
the Poisson hyperplane tessellation, which are well known to satisfy
the moment condition (\ref{eqi0P}). Further examples are the
Delaunay tessellations that are constructed from tempered Gibbsian
point processes with tile interaction, as studied in \cite{DereudreGeorgii}.
Unfortunately, we cannot allow a start in a degenerate tessellation
with the full space $\RR^d$ as its only cell (of any colour), which
would be of major interest; cf. the discussion in Example~\ref
{exSTIT} on the STIT measure $\bPi^{\Lambda,\infty}$. However and
as already
indicated above,
we can choose the initial distribution $P$ to be the time-$\delta$
distribution of $\bPi^{\Lambda,\infty}$ for some small $\delta>0$.
Up to a time
shift, this means that there exists a BRT with degenerate start for any
moderate division kernel $\Psi$
with an initial cutoff of the form $\psi(s, \cdot, \cdot,
\cdot)=1$ for $0\le s<\delta$ and some small $\delta$.
In the special case of shape-driven tessellations as in~(\ref{eqshape-driven}), the existence of a BRT with degenerate initial
tessellation has been proved in~\cite{ST4} under regularity assumptions.

Since $h^\inn( \cdot;\Psi)$ is affine, the last two theorems
imply the following.

%
\begin{corollary}\label{corface}
For any moderate division kernel $\Psi$, the convex set $\sG_\Theta
(\Psi)$ of all translation invariant BRTs admitting $\Psi$ is a face
of $\sP_\Theta$. That is, the extremal elements of $\sG_\Theta
(\Psi)$
are in fact extremal in $\sP_\Theta$, and thereby ergodic under translations.
\end{corollary}

It is clear that for each ergodic $\bP\in\sG_\Theta(\Psi)$ its
initial distribution $P=\bP\circ\bpr_0^{-1}$ is also ergodic. The
converse holds whenever the correspondence between an initial
distribution $P\in\sP_\Theta(\TT)$ and its associated $\bP\in\sG
_\Theta(\Psi)$
is one-to-one. This, however, does not hold in general, as our
concluding remark shows.

%
\begin{remark}\label{remuniqueness}
\emph{Uniqueness and phase transition}.
It is natural to ask whether or not the convex set $\sG(P,\Phi)$ of
all BRTs with initial distribution $P\in\sP(\TT)$ and division
kernel $\Phi$ is a singleton.
In general, this is not the case. To provide an example, let $d=2$ and
consider the division kernel $\Phi$ defined in equation (\ref
{eqexhor-vert}) of Example~\ref{exhor-vert}. Let $T_\reg=\{
[1]+i\dvtx i\in\ZZ^2\}$ be the regular tessellation of $\RR^2$ into
unit squares and $P\in\sP_\Theta(\TT)$ be given by $P=\int_{[1]}\dint
x\, \delta_{T_\reg-x}$. Further, let $\bP^\hor$ be the STIT tessellation
with initial distribution $P$ and driving measure $\Lambda_\hor$ as
introduced in Example~\ref{exhor-vert}, and define $\bP^\vert$ analogously.
It is then clear that these BRTs live on the spaces $\TT_\hor$, respectively,
$\TT_\vert$ for all positive times. As a consequence, $\bP^\hor$
and $\bP^\vert$ are two distinct BRTs which both belong to $\sG
_\Theta
(\Phi)$ and have the same initial distribution $P$.

Although the infinite-range interaction of this example is somewhat
artificial, we learn that uniqueness does not hold automatically.
Instead, the phenomenon of nonuniqueness, or phase transition, which
is a central issue of statistical mechanics, shows up also in the
present setting. In analogy to standard results on Gibbs measures (cf.
\cite{GEO}, Section~8.3), we will show in Proposition~\ref
{prop1duniqueness} below that uniqueness does hold for suitable
division kernels of bounded range in one spatial dimension.
Uniqueness is also known in the noninteracting case (\ref
{eqshape-driven}) when the initial tessellation is degenerate and the
density $\varphi$ exhibits some regularity properties \cite{ST4}. We
leave it to the future to find sufficient conditions for uniqueness in
higher dimensions as well as examples of bounded-range division kernels
exhibiting phase transition. In fact, Figure~\ref{Bild} (right)
suggests that a phase transition might already occur for the (moderate)
model of Example~\ref{exmutations,balance,aging}.
\end{remark}

\section{Proofs}\label{ProofsSec}

\subsection{Some properties of local BRTs}

Before entering into the proofs of our results, we will establish some
auxiliary properties of local BRTs. First we will express the local evolution
of a BRT in a more explicit form. Throughout this section, we let $W\in
\PP_\cup$ be an arbitrary window.
For any division kernel $\Phi_W$ in $W$, we introduce the abbreviation
%
\begin{equation}
\label{eqtotalmass} \hat\phi_W(a, b;\bT)=\int_{a}^{b}
\hat\phi_W(s,\bT_s)\,\dint s,
\end{equation}
where $0\le a<b\le1$, $\bT\in\BB\TT_W$ and $\hat\phi_W(s,\bT
_s)=\widehat\Phi_W (s, \bT_s, T_s\times\langle W\rangle )$
is as in
Remark~\ref{remconstruction}.
For every $T_0\in\TT_W$, we define a measure on $\BB\TT_W$ by
%
\begin{eqnarray}\label{eqPi^PhiW}
&& \bPi^{\Phi_W}_W(T_0, \cdot)\nonumber
\\
&&\qquad= \sum_{n\ge0} \idotsint_{\{0\le s_1<\cdots<s_n\le1\}}
\dint s_1 \cdots\dint s_n \prod
_{i=1}^n\int\widehat\Phi_W
\bigl(s_i,\bT_{s_{i-1}},\dint (c_i,H_i)
\bigr)
\\
&&\quad\qquad{} \times\exp \bigl[-\hat\phi_W (0, 1;\bT)  \bigr]
\1_{\{\bpr_0(\bT
)=T_0, \sD(\bT)=\{(s_i,c_i,H_i)\dvtx 1\le i\le n\}\}} \delta_\bT,\nonumber
\end{eqnarray}
where $s_0:=0$ and the last indicator function simply means that $\bT$
is the unique branching tessellation which starts from $T_0$ and is
successively defined by the division events $(s_i,c_i,H_i)$; recall
(\ref{division}).

%
\begin{lemma}\label{lemlocalcharacterisation} Let $\Phi_W$ be a
division kernel and
$\bP_W\in\sP(\BB\TT_W)$ a BRT in $W$. Then the following
statements are equivalent.
\begin{longlist}[(a)]
\item[(a)] $\bP_W$ admits the division kernel $\Phi_W$.\vspace*{1pt}
\item[(b)] $\bPi^{\Phi_W}_W(T_0, \cdot)$ is the conditional
distribution of $\bP_W$ given $\bpr_0(\bT)=T_0$.
\item[(c)] For every nonnegative measurable function $f$ on $\oBT_W\times
\CC_W\times\langle W\rangle$,
%
\begin{eqnarray}
\label{eqPotimesPhi} &&\int\bP_W(\dint\bT) \sum
_{(s,c,H) \in\sD(\bT)} f{(s,\bT _{s-},c,H)}
\nonumber\\[-8pt]\\[-8pt]
&&\qquad =\int_0^1\dint s\int
\bP_{W,s}(\dint \bT_s)\int\widehat\Phi_W
\bigl(s,\bT_s, \dint(c,H) \bigr) f{(s,\bT_s,c,H)}.\nonumber
\end{eqnarray}
\item[(d)] $\bP_W$ has the infinitesimal generators $\LL_{W,s}^{\Phi_W}$
of (\ref{LocGen}), in that
the forward equation
%
\begin{equation}
\label{eqlocgenerator} \int g \,\dint\bP_{W,t}-\int g \,\dint\bP_{W,0}=
\int_0^t \dint s\int \dint\bP_{W,s}\,
\LL_{W,s}^{\Phi_W}g
\end{equation}
holds for all $t\in[0,1]$ and bounded measurable functions $g$ on $\BB
\TT_W$.
\end{longlist}
\end{lemma}

\begin{pf} (a) \emph{implies} (b). Recall the recursion steps from
Remark~\ref{remconstruction}.
For $i=0$, $T_0=\bT_0$ is chosen according to the distribution $\bP_{W,0}$.
For each $i\ge1$, conditionally on the first $(i-1)$ division
events, the $i$th division event $(s_i,c_i,H_i)$ for a random
tessellation $\bT$ with division rule $\Phi_W$ is chosen according to
the distribution
\[
\exp \bigl[-\hat\phi_W(s_{i-1}, s_i;\bT)
\bigr]\,\dint s_i\, \widehat\Phi _W \bigl(s_i,
\bT_{s_{i-1}},\dint(c_i,H_i) \bigr);
\]
here we have used that $\bT_s=\bT_{s_{i-1}} $ for $s\in[s_{i-1},s_i)$.
So, on the event $\{|\sD(\bT)| = n\}$, the joint distribution of the
elements of $\sD(\bT)$ is the product of these conditional measures
for $i=1,\ldots,n$, times the probability that $s_n$ is the last
division time before 1, which is
$\exp [-\hat\phi_W(s_{n}, 1;\bT) ]$.

(b) \emph{implies} (c).
Fix any initial tessellation $T_0$. On the set of all $\bT$ with fixed
number $n:=|\sD(\bT)|\ge1$ of division events, the measure $\bPi
^{\Phi_W}_W(T_0, \cdot)$ has a product structure. The elements of
$\sD(\bT)$ can be labeled with $i\in\{1,\ldots,n\}$ according to
their temporal order. For each $i$, we extract the $i$th term from the
product, omit its index $i$, and separate the terms concerning the
division events before and after time $s=s_i$. That is, we write $\sD
(\bT)=\sD'\cup\{(s,c,H)\}\cup\sD''$ and separate the respective
conditional measures. By Fubini's theorem, $s$ can be considered as
fixed. For given~$s$, $\sD'$ has the same distribution as $\sD(\bT
_{s})$, but we still have the condition that the extracted division
event $(s,c,H)$ has rank $i$ in $\sD(\bT)$. This condition disappears
by summing over~$i$ and $n\ge i$. Finally, the division events in
$\sD''$ can be integrated out because these do not enter into
$f{(s,\bT_{s-},c,H)}$, and an integration over $T_0$ gives (c).

(c) \emph{implies} (d).
Applying equation (\ref{eqPotimesPhi}) to the function
\[
f(s,\bT,c,H)=\1_{[0,t]}(s) \bigl[g \bigl(\oslash_{s,c,H}(
\bT_{s}) \bigr)-g(\bT _{s}) \bigr],
\]
we find that for each $t\in[0,1]$
%
\begin{eqnarray*}
&& \int_0^t \dint s\int\dint
\bP_{W,s}\, \LL_{W,s}^{\Phi_W}g
\\
&&\qquad = \int\bP_W(\dint\bT_W)\sum
_{(s,c,H) \in\sD(\bT_W)\dvtx  s\le t} \bigl[g(\bT_{W,s})-g(\bT_{W,s-})
\bigr]
\\
&&\qquad = \int\bP_W(\dint\bT_W) \bigl[g(
\bT_{W,t})-g( \bT_{W,0}) \bigr]
\\
&&\qquad = \int g\,\dint\bP_{W,t}-\int g\,\dint\bP_{W,0}.
\end{eqnarray*}

(d) \emph{implies} (a).
In principle, this follows from \cite{Feller40} which, however, makes
use of a time-continuity condition on $\Phi_W$. We thus indicate a
direct argument. For brevity, we omit most indices referring to $W$.
Let $0\le s<t\le1$ and $\bP_{t|\bT_s}$ be a regular version of the
conditional probability of $\bP_{W,t}$ given $\bpr_{W,s}=\bT_s$.
Using~(\ref{eqlocgenerator}) for a function $g$ of the form $g=\1_A\1
_B$ with $A\in\cB_{W,s}$ and $B\in\sigma (\bT\mapsto(T_{W,u})_{s<
u\le1} )$ and varying $A$, one readily finds that
%
\begin{equation}
\label{eqcondgen} \bP_{t|\bT_s}(B)-\delta_{\bT_s}(B)=\int
_s^t \dint u\int\dint\bP _{u|\bT_s}\,
\LL_{u}^{\Phi}\1_B
\end{equation}
for almost all $\bT_s$.
We now fix $\bT_s$ and think of each $\bT_u$ as an element of $\BB
\TT_W$ which is constant on $[u,1]$. Also, for $\bT\in\BB\TT_W$
we let $\tau(\bT)$ be the time of the first jump of $\bT$ after time
$s$, which is set equal to $\infty$ when there is no jump during $[s,1]$.
Setting $B=\{\bT\in\BB\TT_W\dvtx  \tau(\bT)>1\}$, we then find from
(\ref{eqcondgen}) that
\[
\bP_{t|\bT_s}(\tau>t)= 1-\int_s^t \dint
u\, \hat\phi(u,\bT_s) \bP _{u|\bT_s}(\tau>u)
\]
and, therefore, $\bP_{t|\bT_s}(\tau>t)=\exp[-\hat\phi(s, t;\bT
_s)]$. In other words, $\tau$ has the conditional distribution used in
Remark~\ref{remconstruction}.

Next, let $\Gamma\subset\CC\times\HH$ be measurable and
\[
\oslash_\Gamma(T_s)= \bigl\{\oslash_{c,H}(T_s)
\dvtx c\in T_s, (c,H)\in \Gamma \bigr\}.
\]
Consider
the set $B=\{\bT\in\BB\TT_W\dvtx  T_1\in\oslash_\Gamma(T_s)\}$ and let
$\tau_2(\bT)$ the time of the second jump of $\bT$ after $s$ (which
again is set equal to $\infty$ if no second jump exists). Then $\1
_B(\bT_u)=\1_{\{\tau\le u<\tau_2, T_\tau\in\oslash_\Gamma
(T_s)\}}(\bT)$
for $u>s$,
and (\ref{eqcondgen}) implies that
%
\begin{eqnarray*}
&& \bP_{t|\bT_s} \bigl(\tau\le t<\tau_2, T_\tau\in
\oslash _\Gamma(T_s) \bigr)
\\
&&\qquad = \int_s^t
\dint u\, \widehat\Phi(u,\bT_s,\Gamma) \bP_{u|\bT
_s}(\tau>u)
\\
&&\quad\qquad{} -\int_{(s,t]\times\oslash_\Gamma(T_s)} \bP_{t|\bT
_s} \bigl((\tau,
\bT_\tau)\in\dint(v,\bT_v) \bigr) \int
_v^t\dint u\, \hat\phi(u,\bT_v)
\bP_{t|\bT_v}(\tau>u).
\end{eqnarray*}
Using the explicit conditional distribution of $\tau$ derived above, we
thus find that the first term on the right-hand side of the above
equation is equal to
\[
\int_s^t \dint u\, \widehat\Phi(u,
\bT_s, \Gamma) \exp \bigl[-\hat\phi(s, u;\bT_s) \bigr],
\]
whereas the second term equals
%
\begin{eqnarray*}
&&\int_{(s,t]\times\oslash_\Gamma(T_s)} \bP_{t|\bT_s} \bigl((\tau,\bT
_\tau)\in\dint(v,\bT_v) \bigr) \bP_{t|\bT_v}(v<\tau
\le t)
\\
&&\qquad = \bP_{t|\bT_s} \bigl(\tau_2\le t, T_\tau
\in \oslash _\Gamma(T_s) \bigr).
\end{eqnarray*}
We thus arrive at the equation
\[
\bP_{t|\bT_s} \bigl(\tau\le t, T_\tau\in\oslash_\Gamma
(T_s) \bigr)=\int_s^t \dint u\, \exp
\bigl[-\hat\phi(s, u;\bT_s) \bigr] \widehat\Phi(u,\bT _s,
\Gamma ),
\]
which assures that $(\tau,\bT_\tau)$ has the correct conditional
distribution of Remark~\ref{remconstruction}.
\end{pf}

Note that the joint integrating measure on the right-hand side of
(\ref{eqPotimesPhi}) can be written in the concise form
$\obP_W\otimes\widehat\Phi_W$ or, equivalently, $\obC^{\bP_W}
\otimes\Phi$, where $\obP_W=\obP\circ\obpr_W^{-1}$, $\obP$ is
given by (\ref{obP}) and $\obC^{\bP_W}$ by (\ref
{eqextlocCampbell}). We will switch between both representations
according to convenience.

%
\begin{corollary}\label{corPhiW}
Let $\bP_W,\bQ_W\in\sP(\BB\TT_W)$ be two BRTs in $W$. Suppose
$\bQ_W$ admits a division kernel $\Psi_W$, and $\bP_W\ll\bQ_W$.
Then there exists a measurable function $\varphi_W(s,\bT_s,c,H)\ge0$
such that the measure kernel
\[
\Phi_W(s,\bT_s,c,\dint H):= \varphi_W(s,
\bT_s,c,H) \Psi_W(s,\bT _s,c,\dint H)
\]
is a division kernel for $\bP_W$.
\end{corollary}

\begin{pf}
For brevity, we introduce the measure kernel
%
\begin{equation}
\label{eqbDW} \bD_W(\bT, \cdot) = \sum
_{(s,c,H)\in\sD(\bT)}\delta_{(s,\bT
_{s-},c,H)}
\end{equation}
for $\bT\in\BB\TT_W$. The integration on the left-hand side of
(\ref{eqPotimesPhi}) is then with respect to the measure
$\bP_W\bD_W$. Since $\bP_W\ll\bQ_W$ by assumption, it follows that
$\bP_W\bD_W\ll\bQ_W\bD_W$
with a Radon--Nikodym density $f$, say. It also follows that $\obP
_W\ll\obQ_W$ with a density~$g$.
Define
\[
\varphi_W(s,\bT_s,c,H) = f(s,\bT_s,c,H)/g(s,
\bT_s)
\]
if the denominator is positive, and zero otherwise. Then we obtain,
using equation~(\ref{eqPotimesPhi}) for $(\bQ_W,\Psi_W)$ in place
of $(\bP_W,\Phi_W)$,
\[
\bP_W\bD_W = f(\bQ_W\bD_W) =
f(\obQ_W\otimes\widehat\Psi)= \obP _W\otimes(
\varphi_W\widehat\Psi_W).
\]
In view of Lemma~\ref{lemlocalcharacterisation}, this means that
$\bP_W$ admits the division kernel $\Phi_W:=\varphi_W\Psi_W$.\vadjust{\goodbreak}
\end{pf}

Finally, we look at the first-moment condition (\ref{eqfirstmoment}).

%
\begin{lemma}\label{lemmoment-bound} Let $\Phi_W$ be a division
kernel for $W$, and suppose its total mass satisfies
the uniform bound (\ref{eqbdkernel}). Then
\[
\int\bPi^{\Phi_W}_W(T_0,\dint\bT)
|T_{1}| \le e^\phi|T_0|
\]
for all initial tessellations $T_0\in\TT_W$. Moreover, for every
$\varepsilon
>0$, one can find a number $\tau<\infty$ such that
\[
\int\bPi^{\Phi_W}_W(T_0,\dint\bT)
\bigl(|T_{1}|-\tau|T_0| \bigr)_+ \le\varepsilon |T_0|
\]
for all $T_0\in\TT_W$.
\end{lemma}

\begin{pf}
Recall the description of $\bPi^{\Phi_W}_W(T_0, \cdot)$ in Remark
\ref{remconstruction}.
The algorithm there implies that, for each $i$ with $s_i\le1$, the
holding time
$s_i-s_{i-1}$ dominates an exponential time with parameter
$|T_{s_{i-1}}|\phi$, independently of the previous recursion steps.
Hence, the process $|T_s|$ is stochastically dominated by the
Furry--Yule process $Z_s\in\NN$ with birth rate $\phi$,
namely the pure birth Markov process which starts in $k=|T_0|$ and
jumps from any $j\ge1$ to $j+1$ with rate $j\phi$.
Equivalently, $Z_s$ can be described as the branching process in which
each individual, independently of all others, lives for an exponential
time with parameter $\phi$ and then splits into two offspring. In
particular, the descendance trees of each of the $k$ ancestors are
independent, and it is sufficient to look at the number of descendants
at time $s$ in each of these trees.
This number is known to have the geometric distribution with mean
$e^{\phi s}$. A proof of this can be found, for example, in \cite
{Ross}, Examples 6.4, 6.8 or Exercise 6.11.

As for the second assertion, we conclude from the convexity of the
function $a\mapsto(a-\tau)_+$ that
\[
\bigl(|T_{1}|-\tau|T_0| \bigr)_+ \le\sum
_{c\in T_0}\bigl(|T_{c,1}|-\tau \bigr)_+.
\]
Here, $|T_{c,1}|$ is the number of descendants of the initial cell $c$
at time $1$, which is stochastically dominated by the geometric random
variable $Z_1$. As \mbox{$\EE(Z_1-\tau)_+\to0$} as $\tau\to\infty$,
the result
follows immediately.
\end{pf}

\subsection{Significance and construction of global division kernels}

Here, we prove Theorems~\ref{thmequivalence} and~\ref{thmexPhi}.
We begin with the equivalence theorem (Theorem~\ref{thmequivalence}). Most work
will be necessary for deriving the Gibbs property (b) from statement~(c), the characterisation of the jump intensity measure.
To this end, we need to introduce a modification of the outer
projection for a given window $W\in\PP$, which refers to a larger but
bounded window $W'\in\PP$ rather than the full space $\RR^d$.
Namely, for
$W\subset W'\in\PP$ and any $\bT_{W'}\in\BB\TT_{W'}$ let
%
\begin{equation}
\label{eqlocalout} \bT_{W}^{W',\out} = \bigl( \bigl\{c\in
T_{W',s}\dvtx c\not\subset \operatorname{int}(W) \bigr\}
\bigr)_{0\le s\le1}
\end{equation}
be the evolution of the cells hitting $W'\setminus\operatorname{int}(W)$.
In particular, if $W'=W$ then $\bT_{W}^{W,\out}=\bT_W^\partial$.
We also set $\bT_{W}^{W',0,\out}=(T_{W,0}^\inn,\bT_{W}^{W',\out})$
and let $\cB_W^{W',0,\out}$ denote the $\sigma$-field on $\BB\TT_{W'}$
generated by the mapping $\bT_{W'}\mapsto\bT_{W}^{W',0,\out}$.

\begin{pf*}{Proof of Theorem~\ref{thmequivalence}}
We establish the circle $\mbox{(a)} \Rightarrow\mbox{(c)}
\Rightarrow\mbox{(b)} \Rightarrow\mbox{(a)}$.

(a) \emph{implies} (c).
Let $f$ be a bounded nonnegative measurable function on $\oBT\times
\CC\times\HH$, and suppose there is some $W\in\PP$ such that
(i) $f(s,\bT_s, c, H)$ is $\cB_W$-measurable as a function of $\bT$,
and (ii) $f(s,\bT_s, c, H)=0$ unless $c\subset W$ and $|\sD(\bT
_{W,s})|\le K$ for some $K<\infty$. Define
\[
g(\bT) =\sum_{(s,c,H)\in\sD(\bT)}f(s,\bT_{s-}, c, H).
\]
By assumption, $g$ is bounded and local, and $g(\bT_0)=0$ for every
$\bT$. Moreover, 
if $s$ is not a jump time of $\bT_{W,s}$ then
\[
g \bigl(\oslash_{s,c,H}(\bT_s) \bigr)-g(\bT_s)
=f(s,\bT_s,c,H)
\]
and, therefore, $\LL^\Phi_s g(\bT_s)=\int\dint\widehat\Phi(s,\bT
_s,\cdot,\cdot) f(s,\bT_s,\cdot,\cdot)$.
The forward equation thus shows that
\[
\int g\,\dint\bP= \int_0^1\dint s\int\dint
\bP_s\, \LL^\Phi_s g = \int f \dint(\obP\otimes
\widehat\Phi),
\]
which is (c) for our particular $f$. The case of general $f$ now
follows by letting $K\to\infty$ and using a monotone class argument.

(c) \emph{implies} (b).
Fix a window $W\in\PP$ and let $\bP_W^\inn( \cdot|\bT
_W^{0,\out})$ be a regular version of the conditional distribution of
$\bpr_W^\inn$ under the condition $\bpr_W^{0,\out}=\bT_W^{0,\out
}$. We need to show that this probability kernel almost surely
coincides with
$\bG_W^\Phi( \cdot|\bT_W^{0,\out})$. (In particular, this will
imply that the latter is almost surely well defined.)
Pick any two nonnegative measurable functions
$g(\bT_W^{0,\out})$ and $h(s,\bT_{W,s}^\inn,c,H)$ of the indicated
arguments. We suppose $g$ is local, in that $g(\bT_W^{0,\out})=g(\bT
_W^{W',0,\out})$ for some $W'\in\PP$ containing $W$.
Consider the integral
%
\begin{equation}
\label{eqGibbs1} \int\bP(\dint\bT) g \bigl(\bT_W^{0,\out}
\bigr) \int\bP_W^\inn \bigl(\dint\bS |\bT_W^{0,\out}
\bigr) \sum_{(s,c,H)\in\sD(\bS)}h(s,\bS_{s-},c,H).
\end{equation}
By the definition of conditional distribution, this is equal to
\[
\int\bP(\dint\bT) g \bigl(\bT_W^{0,\out} \bigr) \sum
_{(s,c,H)\in\sD(\bT
_{W}^\inn)}h \bigl(s,\bT_{W,s-}^\inn,c,H \bigr).
\]
In view of the locality assumption on $g$, the integrand actually only
depends on $\bT_{W'}$,
which is a pure jump process and therefore strongly Markov. Writing
$(s_i,c_i,H_i)$ for the $i$th division event of $\bT_{W'}$ in temporal
order, we can rewrite the last expression in the form
%
\begin{eqnarray}\label{eqGibbs2}
&& \sum_{i\ge1}\int \bP_{W'}(
\dint\bT_{W'}) g \bigl(\bT_W^{W',0,\out
} \bigr)
\1_{\{s_i<1, c_i\subset W\}}(\bT_{W'}) h \bigl(s_i,
\bT_{W,s_i-}^\inn,c_i,H_i \bigr).\hspace*{-20pt}
\end{eqnarray}
Now, both $h$ and the indicator function in the integrand are
measurable with respect to the $\sigma$-field $\cB_{W',s_i}$
of all events $A\in\cB_{W'}$ with $A\cap\{s_i\le t\}\in\cB_{W',t}$
for all $t$. By the strong Markov property, we can therefore replace
the function $g(\bT_W^{W',0,\out})$ by its conditional expectation
$g(s_i,\bT_{W',s_i})$ relative to
$\cB_{W',s_i}$. Furthermore, the process $\bT_W^{W',0,\out}$ is
itself a Markov jump process. [In fact, it can be considered
as the BRT in $W'$ for the division kernel which equals $\Phi
_{W'}(s,\bT_{W'},c, \cdot)$ if $c\not\subset W$
and is identically zero otherwise.] This means that
\[
g(s_i,\bT_{W',s_i})=g \bigl(s_i,
\bT_{W,s_i}^{W',0,\out} \bigr)=g(s_i,\bT_{W',s_i-})
\]
when $c_i\subset W$. Altogether, we find that the expression (\ref
{eqGibbs2}) is equal to
\[
\int\bP(\dint\bT) \sum_{(s,c,H)\in\sD(\bT)\dvtx  c\subset
W}g(s,
\bT_{W',s-}) h \bigl(s,\bT_{W,s-}^\inn,c,H \bigr).
\]
By statement (c), this in turn coincides with
\[
\int_0^1\dint s\int\bP_s(\dint
\bT_s) \sum_{c\in T_s\dvtx  c\subset W} \int\Phi(s,
\bT_s,c,\dint H) g(s,\bT_{W',s}) h \bigl(s,
\bT_{W,s}^\inn,c,H \bigr),
\]
which by the Markov property is equal to
\[
\int_0^1\dint s\int\bP(\dint\bT) g \bigl(
\bT_W^{0,\out} \bigr) \sum_{c\in
T_{W,s}^\inn}
\int\Phi(s,\bT_s,c,\dint H) h \bigl(s,\bT_{W,s}^\inn,c,H
\bigr).
\]
Taking conditional expectation with respect to $\bT_W^{0,\out}$ and
using the conditional division kernel $ \Phi_{W}^\inn( \cdot|\bT
_W^\out) $ from Definition~\ref{defcondGibbs}, we can rewrite this as
%
\begin{eqnarray*}
&&\int\bP(\dint\bT) g \bigl(\bT_W^{0,\out} \bigr) \int
_0^1\dint s \int\bP_{W,s}^\inn
\bigl(\dint\bS_{s}|\bT_W^{0,\out} \bigr)
\\
&&\quad{}\times \int\widehat\Phi_{W}^\inn \bigl(s,
\bS_s,c, \dint H |\bT_W^\out \bigr) h(s,\bS
_s,c,H).
\end{eqnarray*}
Since the underlying spaces are Borel, a comparison of (\ref
{eqGibbs1}) with the preceding expression
shows that, for almost all $\bT_W^{0,\out}$,
\[
\obP_{W}^\inn \bigl( \cdot|\bT_W^{0,\out}
\bigr)\otimes\widehat\Phi_{W}^\inn \bigl( \cdot|
\bT_W^\out \bigr) =\int\bP_W^\inn
\bigl(\dint\bS|\bT_W^{0,\out} \bigr) \sum
_{(s,c,H)\in\sD(\bS)}\delta_{(s,\bS_{s-}^\inn,c,H)},
\]
which corresponds to (\ref{eqPotimesPhi}). Lemma~\ref
{lemlocalcharacterisation} therefore implies that
$\bP_{W}^\inn( \cdot|\bT_W^{0,\out})$ coincides with $\bG
_W^\Phi( \cdot|\bT_W^{0,\out})$.
This completes the proof of the Gibbs property (b).

(b) \emph{implies} (a).
Let $g$ be a bounded function which is $\cB_W$-measurable for some
$W\in\PP$.
For any $n$ with $W\subset[n]$ let
%
\begin{equation}
\label{eqAn} A_n= \bigl\{\bT\in\BB\TT\dvtx c\cap W=\varnothing
\mbox{ for all }c\in T_0\mbox{ with }c\cap\partial[n]\ne
\varnothing \bigr\}.
\end{equation}
Obviously, $A_n\in\cB^\partial_{[n],0}$. Also, since for each $T\in
\TT$ the union of all cells hitting $W$ is contained in some $[n]$, we
have $A_n\uparrow\BB\TT$ as $n\to\infty$. Furthermore, if $\bT\in
A_n$ then the inner window
$\inn_{[n]}(s,\bT_{[n]}^\partial)$ [defined in (\ref{eqinnW})]
contains $W$ for all $s$.
Using Definition~\ref{defcondGibbs} and Lemma~\ref
{lemlocalcharacterisation}, we thus obtain that
\[
\int\bG_W^\Phi \bigl(\dint\bS|\bT_W^{0,\out}
\bigr) \biggl[g(\bS_t) - g(\bS_0) -\int
_0^t\dint s\, \LL_s^{ \Phi_W^\inn( \cdot|\bT_{W}^\out)} g(
\bS_s) \biggr] =0
\]
for all $0< t\le1$ and $\bT\in A_n$. Integrating this over $\int_{A_n}\bP(\dint\bT)$,
applying the Gibbs property (b) and letting $n\to\infty$ we arrive
at (a).
\end{pf*}

Before turning to the proof of Corollary~\ref{corPhiWfromPhi} it
is worthwhile to introduce a condensed notation for property (c) of
Theorem~\ref{thmequivalence}. So, we introduce the measure kernel
\[
\bD(\bT, \cdot) = \sum_{(s,c,H)\in\sD(\bT)}\delta_{(s,\bT
_{s-},c,H)}
\]
from $\BB\TT$ to $\oBT\times\CC\times\HH$, which catches the
behaviour of $\bT$ at all cell division events; it is analogous to the
kernel $\bD_W$ within a window $W$, which was defined at~(\ref{eqbDW}). Statement (c) of Theorem~\ref{thmequivalence} can then be
written in the concise form
%
\begin{equation}
\label{eqcshort} \bP\bD:=\int\bP(\dint\bT) \bD(\bT, \cdot) = \obP\otimes \widehat
\Phi.
\end{equation}

\begin{pf*}{Proof of Corollary~\ref{corPhiWfromPhi}}
Fix some $W\in\PP$, recall the definitions of $\tpr_W$ and $\Delta
_W$ at (\ref{eqtpr}), and note that
$\bD_W$ is supported on $\oBT_W\times\Delta_W$. Since
\[
\sD(\bT_W) = \bigl\{(s,c\cap W,H)\dvtx (s,c,H)\in\sD(\bT), (c,H)\in
\tpr_W^{-1}\Delta_W \bigr\},
\]
we have $\bD_W(\bT_W,A\times B) =\bD(\bT, \obpr_W^{-1}A \times
\tpr_W^{-1}B)$ for all $\bT\in\BB\TT$ and all events $A\subset
\oBT_W$ and $B\subset\Delta_W$ and, therefore,
by (\ref{eqcshort})
%
\begin{eqnarray*}
\bP_W\bD_W(A\times B) &=& \bP\bD \bigl(
\obpr_W^{-1}A \times\tpr_W^{-1}B
\bigr) = \obP\otimes\widehat\Phi \bigl(\obpr_W^{-1}A \times
\tpr_W^{-1}B \bigr)
\\
&=& \int_{\obpr_W^{-1}A} \dint\obP(s,\bT_s) \widehat\Phi
\bigl(s,\bT _s,\tpr_W^{-1}B \bigr).
\end{eqnarray*}
So, if $\widehat\Phi_W$ is defined as in the corollary then $\bP_W\bD
_W=\obP_W\otimes\widehat\Phi_W$.
Lemma~\ref{lemlocalcharacterisation} thus shows that $\bP_W$ admits
the kernel $\Phi_W$.
\end{pf*}

Finally, we turn to the construction of division kernels for BRTs
satisfying (\ref{LAC}).

\begin{pf*}{Proof of Theorem~\ref{thmexPhi}}
\emph{Part} 1: \emph{Extension of local division kernels}.
By condition (\ref{LAC}), Corollary~\ref{corPhiW} implies that for each
$W\in\PP$
there exists a cell division kernel $\Phi_W$ in $W$ such that $\bP_W$
is a BRT for $\Phi_W$. (In fact, $\Phi_W$ is absolutely continuous
with respect to $\Lambda^{*}$, but we do not need this here.) So, it merely
remains to construct a global common extension $\Phi$ of these kernels
$\Phi_W$.

By Lemma~\ref{lemlocalcharacterisation} and the preceding proof of
Corollary~\ref{corPhiWfromPhi}, we know that
%
\begin{equation}
\label{eqBTCampbell} \obP_W\otimes\widehat\Phi_W =
\bP_W\bD_W =\bP\bD\circ(\obpr _W\otimes
\tpr_W)^{-1}
\end{equation}
on $\oBT_W\times\Delta_W$.
As a consequence, the measures $\obP_W\otimes\widehat\Phi_W$ with $W\in
\PP$ are consistent in the sense that
%
\begin{equation}
\label{consistent} (\obP_{W'}\otimes\widehat\Phi_{W'})\circ(
\obpr_W\otimes\tpr_W)^{-1} =
\obP_W\otimes\widehat\Phi_W \qquad\mbox{on $
\oBT_W\times\Delta_W$}
\end{equation}
for $W\subset W'\in\PP$. To see that these measures admit a common
extension, we first localise to a fixed window $V\in\PP$. For
$W\supset V$, we write
$\obP_W\otimes\1_V\widehat\Phi_W$ for the restriction of $\obP
_W\otimes\widehat\Phi_W$ to the set $\oBT_W\times\tpr_V^{-1}\Delta_V$.
These measures have a finite total mass that does not depend on $W$.
Indeed, (\ref{eqBTCampbell}) and the first-moment condition (\ref
{eqfirstmoment})
imply that
%
\begin{eqnarray}
\label{eqkernelmass} \bigl\|\obP_W\otimes\1_V\widehat
\Phi_W\bigr\|&=& \int\bP(\dint\bT) \bigl| \bigl\{ (s,c,H) \in\sD(\bT)\dvtx
(c,H)\in \tpr_V^{-1}\Delta_V \bigr\}\bigr|
\nonumber\\[-8pt]\\[-8pt]
&\le&\int\bP(\dint\bT) |T_{V,1}|<\infty.\nonumber
\end{eqnarray}
Since all spaces under consideration are Borel spaces, we can thus
apply an abstract version of the Kolmogorov extension theorem \cite{KP}, Corollary~6.15, to obtain a finite measure on $\oBT\times\tpr
_V^{-1}\Delta_V$,
to be denoted by $\obP\otimes\1_V\widehat\Phi$, which satisfies
\[
(\obP\otimes\1_V\widehat\Phi)\circ(\obpr_W\otimes
\tpr_W)^{-1}= \obP _W\otimes\1_V
\widehat\Phi_W
\]
for all $W\in\PP$ with $W\supset V$.
Since $V$ is arbitrary and
\[
\bigcup_{V\in\PP}\tpr_V^{-1}
\Delta_V=\Delta:= \bigl\{(c,H)\in\CC \times\HH\dvtx H\in\langle c
\rangle \bigr\},
\]
the measures $\obP\otimes\1_V\widehat\Phi$ can be glued together to a
locally finite
measure $\obP\otimes\widehat\Phi$ on $\oBT\times\Delta$ satisfying
%
\begin{equation}
\label{eqconsistent2} (\obP\otimes\widehat\Phi)\circ(\obpr_W\otimes
\tpr_W)^{-1}= \obP _W\otimes\widehat
\Phi_W\qquad\mbox{on $\oBT_W\times\Delta_W$}
\end{equation}
for all $W\in\PP$. As we have indicated by the notation,
disintegration shows that this measure is indeed
the product of $\obP$ with a locally finite measure kernel $\widehat\Phi$.

We next need to show that this $\widehat\Phi$ is really a (cumulative)
division kernel.
By construction, each $\widehat\Phi(s,\bT_s, \cdot)$ is supported on
$\Delta$, which means that each $\Phi(s,\bT_s,c, \cdot)$
is supported on $\langle c\rangle$.
In fact, considering the set $\Delta(T_s)=\{(c,H)\dvtx c\in T_s, H\in
\langle
c\rangle\}$ and its complement $\neg\Delta(T_s)$, we can write
%
\begin{eqnarray*}
&& \int\dint\obP(s,\bT_s) \widehat\Phi \bigl(s,\bT_s,
\neg \Delta (T_s) \bigr)
\\
&&\qquad = \int\dint(\obP\otimes\widehat\Phi) (s,\bT_s,c,H) \lim
_{W\uparrow\RR^d} \bigl(1-\1_{\tpr_W^{-1}\Delta
(T_{W,s})}(c,H) \bigr),
\end{eqnarray*}
and the last term vanishes by (\ref{eqconsistent2}) and Fatou's
lemma. So, we can conclude that, for $\obP$-almost all $(s,\bT_s)$,
$\widehat\Phi(s,\bT_s, \cdot)$ is indeed supported on $\Delta
(T_s)$, as
required.

Finally, combining (\ref{eqBTCampbell}) and (\ref{eqconsistent2})
we find that statement (c) of Theorem~\ref{thmequivalence} holds for
all $f$ of the form
\[
f(s,\bT,c,H)=\1_{\langle c\cap W\rangle}(H) f_W(s,\bT_W,c\cap
W,H)
\]
with some $W\in\PP$ and a measurable function $f_W$. As this can be
extended to general $f$
by a monotone class argument, it follows that $\Phi$ is a division
kernel for~$\bP$.

\emph{Part} 2: \emph{Averaging over translations}.
Suppose now that $\bP$ is invariant under translations, and let $\Phi
$ be a global division kernel for $\bP$, which exists by part~1 of the proof.
For $x\in\RR^d$, let $\vartheta_x$ be the spatial\vspace*{1pt} translation by
$-x$, which acts on $\oBT$ via
$\vartheta_x\dvtx  (s,\bT_s)\mapsto(s,\bT_s-x)$ and, by hypothesis,
leaves $\obP$ invariant. As before, we use the same symbol for the translation
$\vartheta_x\dvtx  (c,H)\mapsto(c-x,H-x)$ acting on $\CC\times\HH$.
For $x\in\RR^d$ let $\widehat\Phi^x(s,\bT_s, \cdot):=
\widehat\Phi(s,\bT_s+x,\vartheta_x^{-1} \cdot)$.
We first claim that
%
\begin{equation}
\label{eqinvariant} \obP\otimes\widehat\Phi= \obP\otimes\widehat\Phi^x
\end{equation}
for all $x$. Indeed, let $A\in\cB(\oBT)$ and
$B\in\cB(\CC\times\HH)$ be arbitrarily given. Then we can write,
using the $\vartheta_x$-invariance of $\bP$ in the first step,
%
\begin{eqnarray*}
\bigl(\obP\otimes\widehat\Phi^x \bigr) (A\times B)&=& (\obP\otimes
\widehat\Phi ) \bigl(\vartheta_x^{-1}A\times
\vartheta_x^{-1}B \bigr)
\\
&=&\int\bP(\dint\bT)\sum_{(s,c,H)\in\sD(\bT)\dvtx  (c-x,H-x)\in B} \1_A(s,
\bT_{s-})
\\
&=&\int\bP(\dint\bT)\sum_{(s,c,H)\in\sD(\bT)\dvtx  (c,H)\in B} \1
_A(s, \bT_{s-})
\\
&=&(\obP\otimes\widehat\Phi) (A\times B),
\end{eqnarray*}
proving (\ref{eqinvariant}). The second and the fourth step come from
Theorem~\ref{thmequivalence}(c), and in the third step
we observed that $\sD(\bT)$ consists of the shifted elements of $\sD
(\bT-x)$ and then used again
the translation invariance of $\bP$. Equation (\ref{eqinvariant})
shows that $\widehat\Phi^x( \cdot, \cdot,B)=\widehat\Phi( \cdot, \cdot,B)$ $\obP$-almost surely for each $B$.

To obtain an everywhere covariant version of $\Phi$, we pick a
countable generator $\cG$ of $\cB(\CC\times\HH)$ which is stable
under intersections. We also let $\Gamma$ be the set of all $(s,\bT
_{s})\in\oBT$ which are such that
$\widehat\Phi(s,\bT_{s},B)=\widehat\Phi^y(s,\bT_{s},B)$ for all $B\in\cG
$ [and thus all $B\in\cB(\CC\times\HH)$] and the countably many
lattice elements $y\in\ZZ^d$. Then $\obP(\Gamma)=1$ by (\ref
{eqinvariant}). We further define
the kernel
\[
\widetilde\Phi(s,\bT_s, \cdot )=\cases{ \widehat\Phi(s,
\bT_s, \cdot), &\quad if $(s,\bT_s)\in\Gamma$,
\vspace*{3pt}\cr
\widehat\Lambda^{ *}(s, \bT_{s}, \cdot), &\quad otherwise,}
\]
where $\widehat\Lambda^{ *}$ is the cumulative STIT kernel of Example
\ref
{exSTIT}. It is then clear that $\widetilde\Phi$ is a version of
$\widehat\Phi$ which satisfies $\widetilde\Phi^y=\widetilde\Phi$ for all
$y\in\ZZ^d$.

To achieve the covariance under the full translation group, we finally define
\[
\overline\Phi= \int_{[1]} \widetilde\Phi^x\, \dint
x,
\]
where $[1]$ is the centred unit cube in $\RR^d$. Then for each $y\in
\RR^d$, we have
\[
\overline\Phi^y= \int_{[1]+y} \widetilde
\Phi^x\, \dint x=\overline\Phi
\]
because $[1]+y$ can be decomposed into finitely many pieces which are
lattice translations of corresponding pieces of $[1]$.
On the other hand, since
\[
\obP\otimes\overline\Phi=\int_{[1]} \obP\otimes\widetilde
\Phi^x\, \dint x=\int_{[1]} \obP\otimes\widehat
\Phi^x\, \dint x=\obP\otimes\widehat\Phi
\]
by (\ref{eqinvariant}), $\overline\Phi$ is also a version of $\widehat\Phi$.
\end{pf*}

We conclude this subsection with two supplements to the preceding
proofs. The first deals with the consistency properties (\ref
{consistent}), respectively, (\ref{eqconsistent2}), and the second with
a localised version of the Gibbs property.

%
\begin{remark}[({Consistency of kernel densities})]\label{remphiconsistency}
Consider two windows $W,W'\in\PP$ with $W\subset W'$ and a BRT $\bP
\in\sP$ satisfying $\bP_{W'}\ll\bP_{W',0}\bPi_{W'}^\Lambda$.
Let $\varphi
_W$ and $\varphi_{W'}$ be the $\Lambda$-densities of the division kernels
$\Phi_{W}$ and~$\Phi_{W'}$ of $\bP_{W}$ and $\bP_{W'}$, which exist
by Corollary~\ref{corPhiW}. The consistency equation (\ref
{consistent}) then means that
\[
\varphi_W(s,\bT_{W,s},c,H) = \int\bP_{W',s|\bT_{W,s}}(
\dint\bT _{W',s}) \varphi_{W'} \bigl(s,\bT_{W',s},
\pi_{W}^{-1}( c,T_{W',s}),H \bigr)
\]
for $\obP_W\otimes\Lambda^{*}_W$ almost all arguments. Here,
$\bP_{W',s|\bT_{W,s}}$ stands for a regular version of the
conditional distribution of $\bT_{W',s}$ under $\bP_s$
given $\bT_{W,s}$, and $c'=\pi_{W}^{-1}( c,T_{W',s})$ is the unique
element of $T_{W',s}$ with $c'\cap W=c$.
An analogous statement holds for $W'=\RR^d$ when $\bP$ admits a
global division kernel with a
$\Lambda$-density.
\end{remark}

%
\begin{remark}[(Conditional BRTs with finite horizon)]\label{remlocalisedGibbs}
Fix two windows $W,W'\in\PP$ with $W\subset W'$ and let $\bP_{W'}$
be a BRT in $W'$ for a division kernel
$\Phi_{W'}$. Furthermore, replace $\RR^d$ by $W'$ in Definition~\ref
{defcondGibbs} and use the conditional division kernel
\[
\Phi_W^\inn(s,\bS_s,c, \cdot|
\bT_{W'}):=\Phi_{W'} \bigl(s,\bS _s\cup
\bT_{W,s}^{W',\out},c, \cdot \bigr),
\]
to obtain a conditional BRT $\bG^{\Phi_{W'}}_W( \cdot|\bT
_{W}^{W',0,\out})$ in $W$; here we use the notation
introduced in and after (\ref{eqlocalout}).
The arguments in the proof of Theorem~\ref{thmequivalence}, (c)${}\Rightarrow{}$(b), then show that
the kernel $\bG^{\Phi_{W'}}_W$ is a regular version of the
conditional distribution of $\bpr_W^\inn$ for $\bP_{W'}$
under the condition $\cB_W^{W',0,\out}$.
\end{remark}

\subsection{On the inner entropy density}\label{secEntropyProof}

We first recall some standard properties of relative entropy. A basic
fact is
the variational formula, which states that
%
\begin{equation}
\label{eqvarform} \cH(\mu,\nu) = \sup_g \biggl[\int g\,\dint
\mu- \log\int e^g\,\dint \nu \biggr],
\end{equation}
for any two probability measures $\mu,\nu$ on a common measurable
space. Here,
the supremum extends over all bounded measurable functions on this
space; see \cite{Varadhan}, Theorem~4.1.
On the one hand, the variational formula implies the useful estimate
%
\begin{equation}
\label{eqvarform1} \int g\,\dint\mu\le\cH(\mu,\nu) + \log\int e^g\,\dint
\nu
\end{equation}
for any nonnegative measurable $g$. On the other hand, using Jensen's
inequality it follows immediately that $\cH(\mu,\nu)$ is jointly
measure convex in both arguments simultaneously. Also, it is jointly
lower semi-continuous in $(\mu,\nu)$ in the topology generated by the
integrals of bounded measurable functions. Finally, if $\mu$ and $\nu
$ are restricted to a sub-$\sigma$-field $\cA$ then relative entropy is
increasing in $\cA$. Alternative proofs of these facts can be found in
\cite{GEO}, Section~15.1, for example.
Since $\cH(a\mu;b\nu) =a\cH(\mu,\nu)+b \varrho(a/b)$ for
$a,b>0$ and
normalised $\mu,\nu$ [recall (\ref{rhofunctiondef})], the last facts
extend directly to the case of finite measures,
except that the convexity then holds in the first argument only.

Now, turning to the proof of Theorem~\ref{thmentropy} we proceed with
a series of lemmas.
Let $\bP\in\sP_\Theta$ be arbitrarily given and $P=\bP\circ\bpr
_0^{-1}$ its initial distribution. We can clearly assume that
%
\begin{equation}
\label{eqentr<infty} \liminf_{n\to\infty}n^{-d}
\cH_{[n]}^\inn(\bP) <\infty
\end{equation}
because otherwise there is nothing to show.

%
\begin{lemma}\label{<infty=>LAC} Condition (\ref{eqentr<infty})
implies condition (\ref{LAC}).
\end{lemma}

\begin{pf} We fix a window $W\in\PP$ and consider the sets $A_n$
defined in (\ref{eqAn}).
Recall that for $\bT\in A_n$ the inner window
$\inn_{[n]}(s,\bT_{[n]}^\partial)$ contains $W$ for all $s$, so that
\[
\bG^\Lambda_{[n]} \bigl(B|\bT_{[n]}^{0,\partial}
\bigr)=\bPi^\Lambda_{W}(T_{W,0},B)
\]
for all $B \in\cB_{W}$.
Suppose now that $\cH_{[n]}^\inn(\bP) <\infty$ and recall the
notation of Definition~\ref{definnerentropy}.
Writing $\bP_{[n]}^\inn( \cdot|\bT_{[n]}^{0,\partial})$ for a
regular\vspace*{1pt} conditional distribution of
$\bpr_{[n]}^\inn$ under the condition $\bpr_{[n]}^{0,\partial}=\bT
_{[n]}^{0,\partial}$, we then have
\[
\bP_{[n]}^\inn \bigl( \cdot|\bT_{[n]}^{0,\partial}
\bigr)\ll\bG^\Lambda _{[n]} \bigl( \cdot|\bT_{[n]}^{0,\partial}
\bigr)\qquad\mbox{for $\bP $-almost all $\bT$.}
\]
Therefore, if $B \in\cB_{W}$ is such that $P\bPi^\Lambda(B)=0$ then
$\bPi^\Lambda_{W}(T_{W,0},B)=0$ for almost all $\bT$,
and thus $\bP_{[n]}^\inn(B|\bT_{[n]}^{0,\partial})=0$ for almost
all $\bT\in A_n$. Hence,
\[
\bP(B\cap A_n)=\int_{A_n}\bP(\dint\bT)
\bP_{[n]}^\inn \bigl(B|\bT _{[n]}^{0,\partial}
\bigr)=0.
\]
Letting $n\to\infty$ through the integers $n$ with $\cH_{[n]}^\inn
(\bP) <\infty$, we thus obtain that $\bP(B)=0$.
So, we have shown that $\bP\ll P\bPi^\Lambda$ on $\cB_{W}$, and the
proof is complete.
\end{pf}

Combining the preceding lemma with Theorem~\ref{thmexPhi}, we can
conclude that $\bP$ admits a global division kernel
$\Phi$. Hence, for each window $W\in\PP$, the conditional
distribution of $\bpr_W^\inn$ given $\cB_W^{0,\partial}$ under $\bP
$, respectively, $P\bPi^\Lambda$ are equal to
the localised conditional BRTs $\bG^{\Phi_{W}}_W( \cdot|\bT
_{W}^{0,\partial})$, respectively, $\bG^{\Lambda}_W( \cdot|\bT
_{W}^{0,\partial})$ introduced
in Remark~\ref{remlocalisedGibbs}, respectively, Example~\ref
{excondSTIT}. It follows that
%
\begin{equation}
\label{eqinnerGibbsentr} \cH_W^\inn(\bP)= \int\bP(\dint\bT) \cH
\bigl(\bG^{\Phi
_{W}}_W \bigl( \cdot|\bT_{W}^{0,\partial}
\bigr);\bG^{\Lambda}_W \bigl( \cdot |\bT _{W}^{0,\partial}
\bigr) \bigr).
\end{equation}
This expression can be specified as follows.

%
\begin{lemma}\label{lemWentropy}
Under (\ref{eqentr<infty}), we have for each $W\in\PP$
%
\begin{equation}
\label{eqinnentrW} \cH_{W}^\inn(\bP) = \int\dint
\obP_W(s,\bT_W) \sum_{c\in
T_{W,s}^\inn}
\cH \bigl(\Phi_W(s,\bT_{W,s},c, \cdot); \1_{\langle c\rangle
}
\Lambda \bigr).
\end{equation}
\end{lemma}

\begin{pf}
By Lemma~\ref{<infty=>LAC} and Corollary~\ref{corPhiW}, $\Phi_W$
is absolutely continuous
with respect to $\Lambda$; we write $\varphi_W$ for the associated
Radon--Nikodym density. Using the equivalence of Lemma
\ref{lemlocalcharacterisation}(a) and (b) separately for the
intervals between the ``immigration times'' (\ref{eqjumptimes}), we
obtain the following identity for the
Radon--Nikodym density of $\bG^{\Phi_{W}}_W( \cdot|\bT
_{W}^{0,\partial})$ relative to $\bG^{\Lambda}_W( \cdot|\bT
_{W}^{0,\partial})$:
%
\begin{eqnarray}
\label{eqlocaldensity} \log\frac{\dint\bG^{\Phi_{W}}_W( \cdot|\bT_{W}^{0,\partial
})}{\dint\bG^{\Lambda}_W( \cdot|\bT_{W}^{0,\partial})} \bigl(\bT _W^\inn
\bigr) &=& \bigl[ \hat\lambda_W^\inn \bigl(0,1;
\bT_W^\inn \bigr)-\hat\phi_W^\inn
(0,1;\bT_W) \bigr]
\nonumber\\[-8pt]\\[-8pt]
&&{} + \sum_{(s,c,H) \in\sD(\bT_W^\inn)} \log\varphi _W(s,\bT
_{W,s-},c,H),\nonumber
\end{eqnarray}
where $\bT_W\in\BB\TT_W$, $\sD(\bT_W^\inn)$ is the associated
set of division events with $c\subset W$, and similarly to (\ref
{eqtotalmass}),
\[
\hat\phi_W^\inn(0, 1;\bT_W)=\int
_{0}^{1}\dint s \sum
_{c\in
T_{W,s}^\inn}\Phi_W \bigl(s,\bT_{W,s},c,
\langle c\rangle \bigr)
\]
and $\hat\lambda_W^\inn(0, 1;\bT_W^\inn)=\int_{0}^{1}\dint s
\sum_{c\in T_{W,s}^\inn}\Lambda(\langle c\rangle) $.

Now, $\cH_{W}^\inn(\bP)$ is simply the integral of (\ref{eqlocaldensity})
over
\[
\bP_W^{0,\partial} \bigl(\dint\bT_W^{0,\partial}
\bigr) \bG^{\Phi
_{W}}_W \bigl(\dint\bT_W^\inn|
\bT_{W}^{0,\partial} \bigr)=\bP_W(\dint\bT
_W).
\]
By the equation in Lemma~\ref{lemlocalcharacterisation}(c),
integration of the last term in (\ref{eqlocaldensity}) gives the contribution
\[
\int\dint\obP_W(s,\bT_{W,s}) \sum
_{c\in T_{W,s}^\inn} \int_{\langle
c\rangle} \Phi_W(s,
\bT_{W,s},c,\dint H) \log\varphi_W(s,\bT_{W,s},c,H).
\]
In terms of the measure $\obC^{\bP_W,\inn}$ which is defined by
restricting the sum in (\ref{eqextlocCampbell})
to the cells $c\in T_{W,s}^\inn$, this can be rewritten in the concise form
$ \int\dint\obC^{\bP_W,\inn} \otimes\Lambda^{*} \varphi
_W\log\varphi_W$.
Likewise, we have
%
\begin{eqnarray*}
&&\int\bP_W(\dint\bT_W) \bigl[ \hat
\lambda_W^\inn \bigl(0,1;\bT _W^\inn
\bigr)-\hat\phi_W^\inn(0,1;\bT_W) \bigr]
\\
&&\qquad=\int\dint\obC^{\bP_W,\inn} \otimes\Lambda^{*} [1-\varphi
_W ].
\end{eqnarray*}
Consequently, the $\bP_W$-integral of (\ref{eqlocaldensity}) is
equal to
$ \int\dint\obC^{\bP_W,\inn} \otimes\Lambda^{*} \varrho(
\varphi_W)$, and~(\ref{eqinnentrW}) follows by recalling (\ref{relent}).
\end{pf}

The final step in the proof of Theorem~\ref{thmentropy} is the following.
Let $h^\inn(\bP)$ be defined by (\ref{eqentropy}).
For brevity, we write $h_n^\inn(\bP)=n^{-d} \cH_{[n]}^\inn(\bP)$.

%
\begin{lemma}\label{lementrlb} Under (\ref{eqentr<infty}),
\[
\lim_{n\to\infty}h_n^\inn(\bP)= \sup
_{n\ge1}h_n^\inn(\bP)= h^\inn (
\bP).
\]
\end{lemma}

\begin{pf} We claim first that $\liminf_{n\to\infty}h_n^\inn(\bP
)\ge
h^\inn(\bP)$.
We pick any $0<\eta<1$ and $\ell<\infty$
and restrict the sum in (\ref{eqinnentrW}) for $W=[n]$ to the cells
of $T_s$ with midpoint in
$[n\eta]$ and radius at most $\ell$.
More precisely, we let $L_n\subset\RR^d$ be such that the set $\{
[\eta]+x\dvtx x\in L_n\}$
is a tessellation of the cube $[n\eta]$. (Note that $|L_n| = n^d$.)
Also, we take any $0<\varepsilon<1-\eta$ and let $n$ be so large
that $\eta
+\ell<\varepsilon n$. Then we can write
\[
h_n^\inn(\bP)\ge n^{-d}\sum
_{x\in L_n}h_{n,\eta,\ell}(x)
\]
with
\[
h_{n,\eta,\ell}(x):=\int\dint\obP(s,\bT_{s})\sum
_{c\in T_{s}\cap
\Gamma_{\eta,\ell}(x)} \cH \bigl(\Phi_{[n]}(s,\bT_{[n],s},c,
\cdot); \1_{\langle
c\rangle}\Lambda \bigr),
\]
where $\Gamma_{\eta,\ell}(x)$ is the set of all cells $c$ satisfying
$m(c)\in[\eta]+x$ and $r(c)\le\ell$.
Now, whether or not a cell $c\in T_{[n],s}$ belongs to $\Gamma_{\eta,\ell
}(x)$ can be decided by looking at the restriction $T_{[\varepsilon n]+x,s}$.
So, using Remark~\ref{remphiconsistency} and Jensen's inequality
together with the translation invariance of $\bP$ and the covariance
equation (\ref{eqlocalhomogen}) we find for each $x\in L_n$,
%
\begin{eqnarray*}
h_{n,\eta,\ell}(x)&=&\int\dint\obP_{[\varepsilon n]+x}(s,\bT _{[\varepsilon
n]+x,s})
\\
&&{}\times \sum
_{c\in T_{s}\cap\Gamma_{\eta,\ell}(x)} \int\bP_{[n],s|\bT_{[\varepsilon n]+x,s}}(\dint
\bT_{[n],s})
\\
&&\hspace*{60pt}{} \times\int_{\langle c\rangle}\Lambda(\dint H) \varrho \bigl( \varphi
_{[n]}(s,\bT_{[n],s},c,H) \bigr)
\\
&\ge&\int\dint\obP_{[\varepsilon n]}(s,\bT_{[\varepsilon
n],s})\sum
_{c\in T_{s}\cap
\Gamma_{\eta,\ell}} \int_{\langle c\rangle}\Lambda(\dint H)
\varrho \bigl(\varphi _{[\varepsilon n]}(s,\bT_{[\varepsilon
n],s},c,H) \bigr)
\\
&=& \cH \bigl(\1_{\Gamma_{\eta,\ell}}\obC^{\bP} \otimes\Phi\mid
_{\cB_{[\varepsilon n]}}; \1_{\Gamma_{\eta,\ell}}\obC^{\bP} \otimes
\Lambda^{*}\mid _{\cB_{[\varepsilon n]}} \bigr).
\end{eqnarray*}
In the last expression, $\cB_{[\varepsilon n]}$ is identified with
the $\sigma
$-field that is generated by the projection
$\obpr_{[\varepsilon n]}\otimes\mathrm{id}$,
and $\Gamma_{\eta,\ell}:=\Gamma_{\eta,\ell}(0)$ is viewed as a
set in the
product space $\oBT\times\CC\times\HH$. In the limit as $n\to
\infty$,
Perez' continuity theorem for relative entropies (cf.~\cite{GEO}, Proposition~15.6) implies that the last relative entropy
converges to
%
\begin{eqnarray*}
&&\cH \bigl(\1_{\Gamma_{\eta,\ell}}\obC^{\bP}\otimes\Phi; \1
_{\Gamma
_{\eta,\ell}} \obC^{\bP}\otimes\Lambda^{*} \bigr)
\\
&&\qquad = \int\dint\obP(s,\bT_s)\sum_{c\in T_s\cap\Gamma_{\eta,\ell}}
\cH \bigl(\Phi(s,\bT_{s},c, \cdot); \1_{\langle c\rangle}\Lambda \bigr).
\end{eqnarray*}
By the time-integrated version of the Palm formula (\ref{Palm}) and
the shift covariance of $\Phi$ and $\Lambda$, the last integral is
equal to
\[
h_{\eta,\ell}:=\eta^d \int_{\{r(c)\le\ell\}} \dint
\obP^0(s,\bT _s,c) \cH \bigl(\Phi(s,\bT_{s},c,
\cdot); \1_{\langle c\rangle
}\Lambda \bigr).
\]
Altogether, we find that $\liminf_{n\to\infty}h_n^\inn(\bP) \ge
h_{\eta,\ell}$, and the claim follows by letting
$\eta\to1$ and $\ell\to\infty$.

It remains to show that $h_n^\inn(\bP)\le h^\inn(\bP)$. By (\ref
{eqentr<infty}) and the above, $h^\inn(\bP)<\infty$. This implies
that the kernel $\Phi$ admits a Radon--Nikodym density relative to
$\Lambda$.
Applying Remark~\ref{remphiconsistency} and Jensen's inequality as
above, we conclude from (\ref{eqinnentrW})
that
\[
h_n^\inn(\bP)\le n^{-d}\int\dint\obP(s,
\bT_s)\sum_{c\in T_s\dvtx
c\subset[n]} \cH \bigl(\Phi(s,
\bT_{s},c, \cdot); \1_{\langle
c\rangle
}\Lambda \bigr).
\]
The condition under the sum above implies that $m(c)\in[n]$. Using
again (\ref{Palm}) in its time-integrated version,
we thus find that the last expression is not larger than $h^\inn(\bP)$.
The proof is thus complete.
\end{pf}

%
\begin{remark}\label{rementropyrange}
The inner entropy of a BRT
$\bP$ in a window $W\in\PP$ can be defined by
considering the tessellations not only in $W$ but also in some
neighborhood of $W$. Namely, if $\Phi$ is a division kernel for $\bP$
and $\br>0$, one can introduce the quantity
\[
\cH_W^{\br,\inn}(\bP) = \int\bP(\dint\bT) \cH \bigl(\bG
^{\Phi_{W+B_\br}}_W \bigl( \cdot|\bT^{W+B_\br,0,\out}_W
\bigr); \bG^{\Lambda}_W \bigl( \cdot|\bT_W^{0,\partial}
\bigr) \bigr),
\]
which is called the \emph{inner entropy of $\bP$ in $W$ with horizon
$\br$}.
Here, the first of the conditional BRTs $\bG$ is as in Remark~\ref
{remlocalisedGibbs}. A glance at the preceding proof then shows that
Lemma~\ref{lementrlb} can be extended to yield
\[
\lim_{n\to\infty}n^{-d} \cH_{[n]}^{\br,\inn}(
\bP)= h^\inn(\bP).
\]
\end{remark}

Next, we turn to the proof of Theorem~\ref{thmlevelsets}, which is
split into two lemmas.

%
\begin{lemma}\label{lemaffine} The inner entropy density $h^\inn$ is affine.
\end{lemma}

\begin{pf}
As noticed after (\ref{eqvarform1}), relative entropy is a jointly
convex function of probability measures.
This shows that the inner entropies $\cH^\inn_W(\bP)=\cH(\bP_W;\bP
_W^{0,\partial}\otimes\bG_W^\Lambda)$ are convex in $\bP$,
and so is their limit $h^\inn(\bP)$.
The proof is therefore completed by showing\vspace*{2pt} that this limit is also concave.
So, let $\bP,\bP'\in\sP_\Theta$, $0<a<1$, $\widehat\bP=a\bP
+(1-a)\bP
'$, and assume without loss of generality that
$h^\inn(\widehat\bP)<\infty$. By Lemma~\ref{lementrlb}, it follows
that $h_n^\inn(\widehat\bP)<\infty$ for all $n$.
In particular, $\widehat\bP_{[n]}\ll\widehat\bP^{0,\partial}_{[n]}\otimes
\bG^{\Lambda}_{[n]}$ with a Radon--Nikodym density $g_n$.
The Radon--Nikodym theorem further implies that $\bP_{[n]}\ll\widehat\bP
_{[n]}$ and $\bP_{[n]}'\ll\widehat\bP_{[n]}$ with densities
$f_n$ and $f_n'$, respectively. It is clear that $af_n+(1-a)f_n'=1$
almost surely for $\widehat\bP_{[n]}$.
Moreover, it follows that $\bP_{[n]}^{0,\partial}=f_{n}^{0,\partial
}\widehat\bP_{[n]}^{0,\partial}$
for a suitable Radon--Nikodym density $f_{n}^{0,\partial}$. We
conclude that
$\bP_{[n]} = (f_ng_n/f_n^{0,\partial}) \bP^{0,\partial
}_{[n]}\otimes\bG^{\Lambda}_{[n]}$.
Since $f_n\le1/a$ and $\int\dint\bP_{[n]} \log f_n^{0,\partial}=
\cH(\bP_{[n]}^{0,\partial};\widehat\bP_{[n]}^{0,\partial})\ge0$,
this gives
\[
n^d h_n^\inn(\bP)=\int\dint\bP_{[n]}
\log\frac
{f_ng_n}{f_n^\partial} \le\int\dint\bP_{[n]} \log g_n + \log
\frac{1}a.
\]
Together with the analogous inequality for $\bP'$, we finally end up
with the estimate
\[
a h_n^\inn(\bP) + (1-a) h_n^\inn
\bigl(\bP' \bigr)\leq n^{-d}\int\dint \widehat
\bP_{[n]}\log g_n+o(1)= h_n^\inn(
\widehat\bP) + o(1).
\]
The result thus follows from Lemma~\ref{lementrlb} by letting ${n\to
\infty}$.
\end{pf}

As for the topological properties of $h^\inn$, we note first that its
lower semi-continuity is a direct consequence of Lemma
\ref{lementrlb} and the lower semi-continuity of relative entropy;
recall the discussion below (\ref{eqvarform1}).
Since $\bT\mapsto|T_{[1],1}|$ is the supremum of bounded local
functions, it is also evident that the hitting intensity $i_1( \cdot)$
is lower semi-continuous. It follows that the restricted level sets
$\sP_{\Theta,P,\beta,\gamma}$ (as introduced in Theorem~\ref
{thmlevelsets})
are closed. The following lemma, which can be viewed as a refinement of
Lemma~\ref{<infty=>LAC}, will imply that they are in fact compact; as
the intensity bound is not needed here, we put $\beta=\infty$.

%
\begin{lemma}\label{lemlocequicontinuous} The restricted level sets
$\sP_{\Theta,P,\infty,\gamma}$ are locally equi-continuous in the
following sense: for each $W\in\PP$ and $0\le\gamma<\infty$ and
every sequence $B_k\in\cB_{W}$ with $B_k\downarrow\varnothing$ as
$k\to\infty$, one has
\[
\lim_{k\to\infty}\sup_{\bP\in\sP_{\Theta,P,\infty,\gamma}} \bP(B_k)=0.
\]
\end{lemma}

\begin{pf}
Let $W\in\PP$ and a sequence $B_k\in\cB_W$ with $B_k\downarrow
\varnothing$ be given. Pick some $\varepsilon>0$ and consider the
events $A_n$
defined in (\ref{eqAn}). Recall that $A_n\uparrow\BB\TT$. Since
$A_n\in\cB^\partial_{[n],0}$,
$\bP(A_n)$ depends only on the initial distribution of $\bP$, which
is $P$ for all $\bP\in\sP_{\Theta,P,\infty,\gamma}$.
So, there is an $n$ with $\bP(A_n)\ge1-\varepsilon$ for all $\bP
\in\sP
_{\Theta,P,\infty,\gamma}$.

Next, each $\bP\in\sP_{\Theta,P,\infty,\gamma}$ admits some
division kernel $\Phi$, and $\cH_{[n]}^\inn(\bP)\le\eta:=
n^d\gamma$ by Lemma~\ref{lementrlb}.
Since
\[
\cH_{[n]}^\inn(\bP)=\int\bP(\dint\bT) \cH \bigl(\bG
_{[n]}^{\Phi_{[n]}} \bigl( \cdot|\bT_{[n]}^{0,\partial}
\bigr);\bG _{[n]}^{\Lambda} \bigl( \cdot|\bT_{[n]}^{0,\partial}
\bigr) \bigr)
\]
by definition and Remark~\ref{remlocalisedGibbs}, we can conclude
that the set
\[
H_n:= \bigl\{\bT\in\BB\TT\dvtx \cH \bigl(\bG_{[n]}^{\Phi_{[n]}}
\bigl( \cdot|\bT_{[n]}^{0,\partial} \bigr);\bG_{[n]}^{\Lambda}
\bigl( \cdot |\bT _{[n]}^{0,\partial} \bigr) \bigr) \le\eta/
\varepsilon \bigr\}
\]
in $\cB_{[n]}^{0,\partial}$ has measure at least $1-\varepsilon$
for $\bP$.
It follows that
\[
\bP(B_k) \le2\varepsilon+\bP(B_k\cap A_n
\cap H_n) = 2\varepsilon+\int_{A_n\cap H_n}\bP(\dint\bT)
\bG_{[n]}^{\Phi
_{[n]}} \bigl(B_k|\bT_{[n]}^{0,\partial}
\bigr)
\]
because $\inn_{[n]}(s,\bT_{[n]}^\partial)\supset W$ for all $\bT\in
A_n$ and all $s$; recall (\ref{eqinnW}).
The next step is to use the inequality (\ref{eqvarform1}). For $\bT
\in A_n\cap H_n$,
this inequality shows that
\[
\bigl(\eta/{\varepsilon^2} \bigr) \bG_{[n]}^{\Phi_{[n]}}
\bigl(B_k|\bT _{[n]}^\partial \bigr)\le (\eta/{
\varepsilon}) + \log\int\dint\bPi^\Lambda_W(T_{W,0},
\cdot) \exp \bigl[ \bigl(\eta/{\varepsilon ^2} \bigr)
\1_{B_k} \bigr]
\]
since $\bG_{[n]}^{\Lambda}( \cdot|\bT_{[n]}^\partial)=\bPi
^\Lambda
_W(T_{W,0}, \cdot)$ on $\cB_W$ when $\bT\in A_n$.
Inserting this into the previous inequality, we find
%
\begin{eqnarray*}
&&\sup_{\bP\in\sP_{\Theta,P,\infty,\gamma}} \bP(B_k)
\\
&&\qquad \le3\varepsilon+ \bigl({\varepsilon^2}/\eta \bigr) \int
P_W(\dint T_W) \log\int\dint\bPi^\Lambda_W(T_W,
\cdot ) \exp \bigl[ \bigl(\eta/{\varepsilon^2} \bigr)
\1_{B_k} \bigr].
\end{eqnarray*}
Letting $k\to\infty$, using the dominated convergence theorem, and noting
that $\varepsilon$ was chosen arbitrarily, we
arrive at the lemma.
\end{pf}

The preceding lemma verifies the conditions of Propositions~4.9~and~4.15 of~\cite{GEO}, which imply that
$\sP_{\Theta,P,\infty,\gamma}$ is relatively compact and relatively
sequentially compact within the class of all translation invariant
BRTs. However, this does not yet imply that each limit of a net in $\sP
_{\Theta,P,\infty,\gamma}$ also satisfies the first-moment condition.
(This is because $h^\inn$ is the limit of \emph{conditional}
entropies which do not allow to control the number of cells that hit
the boundary. But this number enters into the hitting intensity $i_1$.)
The simplest way to deal with this problem is to add the bound $i_1\le
\beta$ which trivially implies (\ref{eqiP}) also for all limiting
BRTs. The proof of Theorem~\ref{thmlevelsets} is therefore complete.

\subsection{Free energy density, variational principle, existence}

Throughout this section, we fix a moderate division kernel $\Psi$. Our
first item is the existence of the energy density.

\begin{pf*}{Proof of Theorem~\ref{thmenergydensity}}
Let $\bP\in\sP_\Theta$ be a BRT with a covariant division kernel
$\Phi$.
By property (c) of Theorem~\ref{thmequivalence}, the inner energy of
$\bP$ in a window $W\in\PP$ can be written in the form
%
\begin{eqnarray}
\label{eqUW^in} \qquad&&\cU^\inn_W(\bP;\Psi)
\nonumber\\[-8pt]\\[-8pt]
&&\qquad= \int\dint\obP(s, \bT_s)\sum_{c\in T_s\dvtx  c\subset W}
\int\Phi (s,\bT_s,c,\dint H) \log\psi(s,\bT_s,c,H).\nonumber
\end{eqnarray}
Since $\Phi$ and $\psi$ are covariant, the time-integrated version of
the Palm formula~(\ref{Palm}) shows that the last term can be written
in the form
\[
\int\dint\obP^0(s, \bT_s,c) \Vol (x\dvtx c+x\subset W )
\int \Phi(s,\bT_s,c,\dint H) \log\psi(s,\bT_s,c,H).
\]
Hence,
\[
n^{-d} \cU^\inn_{[n]}(\bP;\Psi) =
u^\inn(\bP;\Psi) + \delta _n(\bP;\Psi)
\]
with
\[
\bigl|\delta_n(\bP;\Psi)\bigr|\le\kappa_\Psi\int\dint
\obP^0(s, \bT _s,c) \Vol \bigl(x\in[1]\dvtx c/n+x\not
\subset[1] \bigr) \Phi \bigl(s,\bT _s,c,\langle c\rangle \bigr)
\]
by (M3). The volume term above is bounded by $1$ and tends to $0$ as
$n\to\infty$. To apply the dominated convergence theorem, we thus
need to
show that the total mass of $\obP^0\otimes\Phi$ is finite. But the
Palm formula and
(\ref{eqkernelmass}) show that this mass is at most $i_1(\bP)$.
This completes the proof of the first part of Theorem~\ref
{thmenergydensity} and implies the bound on $|u^\inn(\bP;\Psi)|$.

The proof of the second part is similar: the Palm formula gives
%
\begin{eqnarray*}
&&\cV^\inn_W(\bP;\Psi)
\\
&&\qquad= \int\dint\obP^0(s, \bT_s,c) \Vol (x\dvtx c+x
\subset W ) \int_{\langle c\rangle}\Lambda(\dint H) \bigl(\psi(s,\bT
_{s},c,H)-1 \bigr)
\end{eqnarray*}
and thus $n^{-d} \cV^\inn_{[n]}(\bP;\Psi) = v^\inn(\bP;\Psi) +
\delta_n'(\bP;\Psi)$ with a remainder term $\delta_n'$ which, by
assumption (M4), is bounded in modulus by $\kappa_\Psi'$ times
\[
\int\dint\obP^0(s, \bT_s,c) \Vol \bigl(x\in[1]\dvtx
c/n+x\not \subset [1] \bigr).
\]
By (\ref{eqextPalmmass}) and the dominated convergence theorem,
this bound vanishes in the limit $n\to\infty$. The proof of Theorem
\ref
{thmenergydensity} is therefore complete.
\end{pf*}

%
\begin{remark}\label{remenergydensity}
Exploiting Theorem~\ref{thmequivalence}(c) and the Palm formula
(\ref{Palm}) in the same way as
in the first part of the preceding proof, one finds that the energy
density can be written in the alternative form
\[
u^\inn(\bP;\Psi)=\int\bP(\dint\bT)\sum_{(s,c,H)\in\sD(\bT
)\dvtx  m(c)\in[1]}
\log\psi(s,\bT_{s-},c,H),
\]
in which the division kernel of $\bP$ does not appear. In particular,
it follows that $u^\inn( \cdot;\Psi)$ is affine.
\end{remark}

Turning to the proof of the variational principle, Theorem~\ref
{thmvarprinciple}, we introduce an \emph{inner relative entropy of a
BRT $\bP$ in a window $W\in\PP$ with horizon $\br=\br_\Psi$
relative to $\Psi$} as follows: if $\bP$ admits a division kernel
$\Phi$ we set, using the notation of Remark~\ref{remlocalisedGibbs},
%
\begin{eqnarray}
\label{eqinnrelentr} \qquad &&\cH_W^{\br,\inn}(\bP;\Psi)
\nonumber\\[-8pt]\\[-8pt]
&&\qquad = \int\bP(\dint\bT) \cH \bigl(\bG^{\Phi_{W+B_\br}}_W \bigl(
\cdot| \bT^{W+B_\br,0,\out}_W \bigr); \bG^{\Psi_{W+B_\br}}_W
\bigl( \cdot|\bT^{W+B_\br,0,\out}_W \bigr) \bigr);\nonumber
\end{eqnarray}
otherwise we set $\cH_W^{\br,\inn}(\bP;\Psi)=\infty$.
(Compare this definition with Remark~\ref{rementropyrange}, where
$\Psi=\Lambda^{*}$.)
By the bounded-range property (M2) of $\Psi$ and Corollary~\ref
{corPhiWfromPhi}, the conditional BRT $\bG^{\Psi_{W+B_\br}}_W$
in (\ref{eqinnrelentr})
actually coincides with $\bG^{\Psi}_W$. We then have the following
convergence to the quantity $h^\inn(\bP;\Psi)$
in (\ref{eqfreeenergy}).

%
\begin{corollary}\label{correlentrhorizon}
Let $\Psi$ be a moderate division kernel and $\br=\br_\Psi$ its
range. Then
\[
h^\inn(\bP;\Psi)=\lim_{n\to\infty}n^{-d}
\cH_{[n]}^{\br,\inn
}(\bP;\Psi)
\]
for all $\bP\in\sP_\Theta$. The limit is finite if and only if
$h^\inn
(\bP)<\infty$, and then equation~(\ref{eqrelentrdensityPsi}) holds.
\end{corollary}

\begin{pf}
An analog of equation (\ref{eqlocaldensity}) gives for each $n$ the identity
%
\begin{equation}
\label{eqHn^r,in} \cH_{[n]}^{\br,\inn}(\bP;\Psi)=
\cH_{[n]}^{\br,\inn}(\bP)-\cU ^\inn_{[n]}(\bP;
\Psi)+\cV^\inn_{[n]}(\bP;\Psi),
\end{equation}
which is a counterpart to (\ref{eqfreeenergy}). Also, the estimates
in the proof of Theorem~\ref{thmenergydensity} show that the second
and third term on the right-hand side are bounded in modulus by a
finite constant times $n^d$.
The convergence result thus follows directly from Remark~\ref
{rementropyrange} and Theorem~\ref{thmenergydensity}
(together with Lemma~\ref{<infty=>LAC} and Theorem~\ref{thmexPhi}).

Next, suppose that $h^\inn(\bP)<\infty$ and let $\varphi$ and
$\psi$
be the Radon--Nikodym densities of $\Phi$ and $\Psi$ with respect to
$\Lambda^{*}$. Inserting
the explicit expressions for all quantities, we then obtain
%
\begin{eqnarray*}
&& h^\inn(\bP)-u^\inn(\bP;\Psi)+v^\inn(\bP;\Psi)
\\
&&\qquad = \int\dint \obP^0 \otimes\Lambda^{*} [-\varphi+
\varphi \log\varphi-\varphi\log\psi+ \psi ]
\\
&&\qquad =\int\dint\obP^0 \otimes\Lambda^{*} \psi\varrho(
\varphi /\psi) =\int\dint\obP^0\, \cH(\Phi;\Psi),
\end{eqnarray*}
which is (\ref{eqrelentrdensityPsi}).
\end{pf}

The variational principle, Theorem~\ref{thmvarprinciple}, follows
directly from equation (\ref{eqrelentrdensityPsi}) and thus from
the preceding corollary.

Next we address the existence problem for BRTs with given division
kernel, as stated in Theorem~\ref{thmexistence}.
We still keep a moderate $\Psi$ fixed and let $\br=\br_\Psi$ be its
range. We also fix an initial distribution $P\in\sP_\Theta(\TT)$.
We will construct a BRT $\bP$ with initial distribution $\bpr_0(\bP
)=P$ and division kernel $\Psi$ as a cluster point of some
approximating measures $\bP^{n,\av} $.

Specifically, for any $n$ we let $\bar n=n+\br$ and consider the
shifted cubes $[n]_i=[n]+\bar n i$, $i\in\ZZ^d$, which are separated
by a grid of corridors of width $\br$. Let $[n]_\bullet=\bigcup_{i\in\ZZ^d}[n]_i$ be their union.
We introduce a BRT $\bP^{n}$ for which the cells that hit the
corridors between the boxes $[n]_i$
evolve according to $\Lambda^{*}$ and, conditioned on this STIT
evolution, the
cells inside these boxes evolve independently according to $\Psi$.
(This is inspired by the familiar construction of independent
repetitions in disjoint blocks,
which is often used in large deviation theory; see \cite{GEO}, (15.52), for example. Using the STIT process in the corridors
between the blocks, we avoid an artificial cutting of cells at the
block boundaries.)
Formally, we introduce the projection
\[
\bpr_{[n]_\bullet}^{0,\out}\dvtx \bT\mapsto\bT_{[n]_\bullet}^{0,\out}:=
\biggl(\bigcup_{i\in\ZZ^d}T_{[n]_i,0}^\inn,
\bigcap_{i\in\ZZ^d}\bT_{[n]_i}^\out
\biggr),
\]
and define
%
\begin{equation}
\label{eqP^n} \bP^{n} = \bigl(P\bPi^\Lambda
\bigr)_{[n]_\bullet}^{0,\out}\otimes \bigotimes
_{i\in\ZZ^d} \bG^\Psi_{[n]_i}.
\end{equation}
More explicitly, $\bP^{n}$ is defined by its integrals
\[
\int f \dint\bP^{n} =\int P\bPi^\Lambda(\dint\bT) \prod
_{i\in
\ZZ
^d}\int\bG^\Psi_{[n]_i} \bigl(
\dint\bS_i|\bT_{[n]_i}^{0,\out} \bigr) f \biggl(
\bT_{[n]_\bullet}^{0,\out}\cup\bigcup_i
\bS _i \biggr)
\]
for measurable functions $f\ge0$ on $\BB\TT$. By the bounded-range
property (M2), the conditional BRTs
$\bG^\Psi_{[n]_i}(\cdot|\bT_{[n]_i}^{0,\out})$
 depend only on $\bT_{[n]_\bullet}^{0,\out}$, so
that $\bP^{n}$ is\vspace*{2pt} well defined.
It is easily seen that $\bP^{n}$ is a BRT with initial distribution
$P$ and division kernel
%
\begin{equation}
\label{eqPsi^n} \Psi^n(s,\bT,c, \cdot)=\cases{ \Psi(s,\bT,c, \cdot),
&\quad if $c\subset[n]_i$ for some $i\in \ZZ^d$,
\vspace*{3pt}\cr
\Lambda \bigl(\langle c\rangle\cap\cdot \bigr), &\quad otherwise.}
\end{equation}
To achieve translation invariance, we introduce the average
%
\begin{equation}
\label{eqP^n,av} \bP^{n,\av} = \bar n ^{-d}\int
_{[\bar n ]}\dint x\, \bP^{n}\circ \vartheta_x^{-1}.
\end{equation}
The next two lemmas show that the BRTs $\bP^{n,\av}$ belong to a
restricted level set of the inner entropy density,
and thus have a cluster point.

%
\begin{lemma}\label{lemuniformmomentbound}
\textup{(a)}~There exists a constant $\beta<\infty$ such that\newline\mbox
{$i_1(\bP^{n,\av})\le\beta$} for all $n$.

\textup{(b)}~For every $\varepsilon>0$, there exists some $\tau<\infty$
such that
\[
\int\bP^{n,\av}(\dint\bT) |T_{[1],1}| \1_{\{|T_{[1],1}|\ge
\tau\}
}\le
\varepsilon\qquad\mbox{for all }n.
\]
\end{lemma}

\begin{pf}Let $\kappa=\kappa_\Psi$. Since $\psi\le e^\kappa$ by
(M3), it
follows that each kernel $\Psi^n$ also has a $\Lambda$-density $\psi^n$
satisfying $\psi^n\le e^\kappa$ for all $n$. With the help of Remark
\ref
{remphiconsistency}, we can further conclude that
this bound remains true after localisation to a window $W\in\PP$
(relative to $\bP^n$); that is,
the localised kernel $\Psi^n_W$ has a $\Lambda$-density $\psi^n_W$ with
$\psi^n_W\le e^\kappa$. In particular, if $W=[1]+x$ is a translate
of the
unit cube, then
%
\begin{equation}
\label{eqalpha} \Psi^n_W \bigl(s,\bT_{W,s},c,
\langle c\rangle \bigr)\le e^\kappa\Lambda \bigl( \bigl\langle[1] \bigr
\rangle \bigr)=:\alpha <\infty
\end{equation}
for all possible arguments.
In view of Lemma~\ref{lemmoment-bound}, it follows that
%
\begin{equation}
\label{equniformmomentbound} \int\bP^n(\dint\bT) |T_{[1]+x,1}| \le
\beta:=i_0(P) e^\alpha,
\end{equation}
and statement (a) follows by averaging over $x$.

To prove (b), we still let $W=[1]+x$ and define $\varepsilon_1
=\varepsilon/4i_0(P)$.
By Lemma~\ref{lemmoment-bound}, there exists a number $\tau_1$ with
\[
\sup_n\int\bP^n(\dint\bT) \bigl(|T_{W,1}|-
\tau_1|T_{W,0}| \bigr)_+ \le\varepsilon_1
i_0(P) =\varepsilon/4.
\]
For any $\tau_2$ we then find (by distinguishing whether or not $\tau
_1|T_{W,0}|\le\tau_2$) that
\[
\int\bP^n(\dint\bT) \bigl(|T_{W,1}|-\tau_2 \bigr)_+\le
\varepsilon/4+ e^\alpha \int P(\dint T) |T_{[1]}|
\1_{\{|T_{[1]}|>\tau_2/\tau_1\}},
\]
which is at most $\varepsilon/2$ for suitable choice of $\tau_2$.
Setting $\tau
=2\tau_2$ and using that
\[
|T_{W,1}|\le2 \bigl(|T_{W,1}|-\tau_2 \bigr)_+\qquad
\mbox{on }\bigl\{ |T_{W,1}|\ge\tau\bigr\},
\]
we then see that
\[
\int\bP^{n}(\dint\bT) |T_{[1]+x,1}| \1_{\{|T_{[1]+x,1}|\ge\tau
\}
}\le
\varepsilon
\]
for all $n$. Statement (b) thus follows by taking the average over
$x\in[\bar n]$.
\end{pf}

%
\begin{lemma}\label{lemPsi-entropy->0}
$h^\inn(\bP^{n,\av};\Psi)\to0$ as $n\to\infty$.
\end{lemma}

\begin{pf}Fix any $n$ and let $\bP^n$ and $\Psi^n$ be given by (\ref
{eqP^n}) and (\ref{eqPsi^n}). Consider the inner relative entropy of
$\bP^{n,\av}$ relative to $\Psi$ in a large cube $W=[k]$ and with
horizon $\br=\br_\Psi$, as defined in
(\ref{eqinnrelentr}). In concise notation, (\ref{eqinnrelentr}) reads
\[
\cH_W^{\br,\inn}(\bP;\Psi) =\cH \bigl(\bP_{W+B_\br}; \bP
^{W+B_\br,0,\out}_W\otimes\bG^{\Psi}_W \bigr).
\]
As relative entropy is jointly measure convex, we have
\[
\cH_{W}^{\br,\inn} \bigl(\bP^{n,\av};\Psi \bigr) \le\bar n
^{-d}\int_{[\bar
n ]}\dint x\, \cH_{W+x}^{\br,\inn}
\bigl(\bP^{n};\Psi \bigr).
\]
To estimate this further, we note that $\bP^n$ has the division kernel
$\Psi^n$.
A combination of (\ref{eqHn^r,in}), (\ref{eqUW^in}), (\ref
{eqpressureW}) and an analog of (\ref{eqinnentrW})
thus gives the formula
\[
\cH_{W+x}^{\br,\inn} \bigl(\bP^{n};\Psi \bigr)= \int
\dint \obP^n(s,\bT_s) \sum_{c\in T_s\dvtx  c\subset W+x}
\cH \bigl(\Psi^n; \Psi|s,\bT_{s},c \bigr),
\]
where $\cH(\Psi^n; \Psi|s,\bT_{s},c)=\cH (\Psi^n(s,\bT
_{s},c, \cdot); \Psi(s,\bT_{s},c, \cdot) )$ for brevity.
We can further use that $W+x\subset[k+\bar n]$ when $x\in[\bar n]$.
Altogether, we obtain
\[
\cH_{W}^{\br,\inn} \bigl(\bP^{n,\av};\Psi \bigr) \le\int
\dint\obP^n(s,\bT _s) \sum_{c\in T_s\dvtx  c\subset[k+\bar n]}
\cH \bigl(\Psi^n; \Psi|s,\bT_{s},c \bigr).
\]

Next, it is clear from (\ref{eqPsi^n}) that $\cH(\Psi^n; \Psi|
\cdot, \cdot,c)=0$ when $c\subset[n]_i$ for some~$i$.
On the other hand, for any cell $c$ hitting the corridors between the
boxes $[n]_i$ we have $\cH(\Psi^n; \Psi| \cdot, \cdot,c)=\cH
(\Lambda^{*}; \Psi| \cdot, \cdot,c)$, which is bounded by a constant.
Indeed, the function $\varrho(a)$ defined in (\ref{rhofunctiondef}) is
bounded by a multiple of $|a-1|$ as long as $a\le e^{\kappa_\Psi}$.
Assumptions (M3) and (M4) therefore imply that
\[
\cH \bigl(\Lambda^{*}; \Psi| \cdot, \cdot,c \bigr) =\int
_{\langle c\rangle}\Lambda(\dint H) \psi( \cdot, \cdot,c,H) \varrho
\bigl(1/\psi( \cdot, \cdot,c,H) \bigr) \le\tilde \kappa_\Psi
\]
for some constant $\tilde\kappa_\Psi<\infty$ and all $c$ hitting
the corridors.

Now let $k=(\ell-1)\bar n$ for some integer $\ell$ and $L_\ell$ be
such that $\{[1]+x\dvtx  x\in L_\ell\}$ is a tessellation
of $[\ell\bar n]\setminus[n]_\bullet$. The preceding estimates then
show that
%
\begin{eqnarray*}
\cH_{[k]}^{\br,\inn} \bigl(\bP^{n,\av};\Psi \bigr) &\le&
\sum_{x\in L_\ell} \int\dint\obP^n(s,
\bT_s) \sum_{c\in T_s\dvtx  c\cap([1]+x)\ne
\varnothing} \cH \bigl(
\Psi^n; \Psi|s,\bT_{s},c \bigr)
\\
&\le&\tilde\kappa_\Psi\sum_{x\in L_\ell}\int\dint
\obP^n(s,\bT _s) |T_{[1]+x,1}| \le\tilde
\kappa_\Psi\beta |L_\ell|.
\end{eqnarray*}
The last inequality comes from (\ref{equniformmomentbound}).
Letting $\ell\to\infty$ and applying Corollary~\ref{correlentrhorizon},
we finally see that
\[
h^{\inn} \bigl(\bP^{n,\av};\Psi \bigr) \le\tilde
\kappa_\Psi\beta\lim_{\ell\to\infty
}\frac{(\ell\bar n)^d-\ell^d n^d}{((\ell-1)\bar n)^d} = \tilde
\kappa_\Psi\beta \bigl(1-({n}/{\bar n})^d \bigr).
\]
This proves the lemma.
\end{pf}

Combining equation (\ref{eqfreeenergy}) with the last lemma and the
bounds in Theorem~\ref{thmenergydensity} and Lemma~\ref
{lemuniformmomentbound}(a), one finds that
\[
h^\inn \bigl(\bP^{n,\av} \bigr)\le \bigl(\kappa_\Psi+ \kappa_\Psi' \bigr)\beta +1=:\gamma <\infty
\]
when $n$ is large enough. That is, the measures $\bP^{n,\av}$
eventually belong to the sequentially compact level set
$\sP_{\Theta, P,\beta,\gamma}$ of Theorem~\ref{thmlevelsets}. This
means that a subsequence converges in $\tau_\loc$ to some $\bP$ in
this set. We need to show that $\bP$ has the \mbox{division} kernel $\Psi$.
In view of Theorem~\ref{thmvarprinciple}, this will follow
once we have shown that $h^\inn(\bP;\Psi)=0$. By the last lemma, it
is therefore sufficient to verify that $h^\inn( \cdot;\Psi)$
is lower semi-continuous on the closure of the sequence $\{\bP^{n,\av
}\dvtx  n\ge1\}$.
In view of equation (\ref{eqfreeenergy}) and Theorem~\ref{thmlevelsets}, this follows from the next lemma, which completes the
proof of Theorem~\ref{thmexistence}.

%
\begin{lemma}\label{lemenergycont}
The functionals $u^\inn( \cdot;\Psi)$ and $v^\inn( \cdot;\Psi)$ are continuous on the closure $\sC$ of the sequence
$\{\bP^{n,\av}\dvtx  n\ge1\}$.
\end{lemma}

\begin{pf}
First, we observe that the estimate in Lemma~\ref
{lemuniformmomentbound}(b) holds not only for all $\bP^{n,\av}$,
but even for all $\bP\in\sC$. This is because the integral there is
a lower semi-continuous function of the integrating measure.
We further know from Lemma~\ref{<infty=>LAC} that each $\bP\in\sC$
satisfies (\ref{LAC}).
Hence, Theorem~\ref{thmenergydensity} and Remark~\ref
{remenergydensity} can be applied.

It follows that $u^\inn(\bP;\Psi)= \int u \,\dint\bP$ for the function
\[
u(\bT) = \sum_{(s,c,H)\in\sD(\bT)\dvtx  m(c)\in[1]} \log\psi(s,\bT
_{s-},c,H)
\]
on $\BB\TT$, which in general is neither bounded nor local. We will
therefore replace $u$ by a truncated version
\[
u_{\tau,\ell}(\bT) = \1_{\{|\bT_{[1],1}| \le\tau\}} \sum_{(s,c,H)\in
\sD(\bT)\dvtx  m(c)\in[1], r(c)\le\ell}
\log\psi(s,\bT_{s-},c,H),
\]
for suitable numbers $\tau$ and $\ell$; $r(c)$ is again the radius of
$c$. The function $u_{\tau,\ell}$ is bounded in modulus by $\kappa
_\Psi\tau
$ and also local because of (M2). It differs from $u$ by at most
$\kappa
_\Psi(\delta_\tau+\delta_\ell)$ with the error functions
\[
\delta_\tau(\bT)= \1_{\{|\bT_{[1],1}| > \tau\}} |\bT_{[1],1}|,\qquad
\delta_\ell(\bT)= \sum_{(s,c,H)\in\sD(\bT)\dvtx  m(c)\in
[1]}
\1_{\{r(c)>\ell\}}.
\]
As noticed at the beginning of this proof, we have $\sup_{\bP\in\sC
}\int\delta_\tau\, \dint\bP\to0$ as $\tau\to\infty$. On the
other hand,
the function $\delta_\ell$ is not larger than
\[
\delta_\ell'(\bT) = \sum_{c_0\in T_0\dvtx  c_0\cap[1]\ne\varnothing,
r(c_0)>\ell}
|T_{c_0\cap[1],1}|,
\]
and Lemma~\ref{lemmoment-bound} gives the estimate
\[
\sup_{\bP\in\sC}\int\delta_\ell' \,\dint\bP
\le e^\alpha\int P(\dint T) \sum_{c_0\in T_0\dvtx  c_0\cap[1]\ne\varnothing}
\1_{\{
r(c_0)>\ell\}}
\]
for the constant $\alpha$ in (\ref{eqalpha}) because each $\bP\in
\sC$
has initial distribution $P$.
This bound does not depend on $n$ and tends to $0$ as $\ell\to\infty$
because $i_0(P)<\infty$. We have thus shown
that the restriction of $u^\inn( \cdot;\Psi)$ to $\sC$ is the
uniform limit of the functions $\bP\mapsto\int u_{\tau,\ell}
\,\dint
\bP$,
which are continuous in $\tau_\loc$.

The analogous result for $v^\inn(\bP^{n,\av};\Psi)$ is achieved in
a similar way by truncating the function
\[
v(\bT)= \int_0^1\dint s \sum
_{c\in T_s\dvtx  m(c)\in[1]}\int_{\langle
c\rangle}\Lambda(\dint H) \bigl(
\psi(s,\bT_{s},c,H)-1 \bigr)
\]
and using (M4).
\end{pf}

As the proof of Theorem~\ref{thmexistence} is now complete, we turn
to its corollary.

\begin{pf*}{Proof of Corollary~\ref{corface}}
Suppose $\bP\in\sG_\Theta(\Psi)$
is not extremal in $\sP_\Theta$. Then $\bP= a \bP^1+(1-a)\bP^2$ for
some $0<a<1$ and two distinct BRTs $ \bP^1, \bP^2\in\sP_\Theta$. By
Theorem~\ref{thmlevelsets}, Remark~\ref{remenergydensity} and
Theorem~\ref{thmvarprinciple}, it follows that
\[
0=h^\inn(\bP;\Psi) = a h^\inn \bigl(\bP^1;\Psi
\bigr)+(1-a) h^\inn \bigl(\bP ^2;\Psi \bigr),
\]
so that $ \bP^1, \bP^2$ both belong to $\sG_\Theta(\Psi)$. Hence,
$\bP$ is not extremal in $\sG_\Theta(\Psi)$.
\end{pf*}

Our final observation concerns the uniqueness problem discussed in
Remark~\ref{remuniqueness}. We will exploit the fact that, in one
space dimension,
we always have that $\sum_{c\in T_W}\Lambda(\langle c\rangle
)=\Lambda(\langle W\rangle)$
when $W\in\PP$ and $T_W\in\TT_W$.
Consider the following variants of conditions (M3) and (M4):

\begin{longlist}[(M4$'$)]
\item[(M3$'$)] $\Psi$ is \emph{STIT-bounded}, in that $\Psi\le K_\Psi
\Lambda^{*}$ for some
constant $K_\Psi<\infty$.
\item[(M4$'$)]
$\Psi$ is \emph{STIT for large cells}, in that $\Psi( \cdot,
\cdot,c, \cdot)=\Lambda(\langle c\rangle\cap \cdot
)$ whenever
$\diam(c)\ge r_\Psi'$ for some constant $r_\Psi'<\infty$.
\end{longlist}

%
\begin{proposition}\label{prop1duniqueness}
Suppose that the space dimension is $d=1$. Let $P\in\sP(\TT)$ and
$\Psi$ be a division kernel satisfying
\textup{(M2)}, \textup{(M3$'$)} and \textup{(M4$'$)}.
Then there exists at most one BRT for $\Psi$ with initial distribution $P$.
\end{proposition}

\begin{pf}
Suppose there exist two distinct BRTs $\bP$, $\bP'$ for $\Psi$ with
the same initial distribution~$P$. Consider the difference measure $\bP
^\delta=\bP-\bP'$ and fix an interval $[k]\in\PP$. Let $g$ be
$\cB_{[k]}$-measurable with $|g|\le1$.
Using property (a) of Theorem~\ref{thmequivalence}, we obtain for
each $0<t\le1$ the identity
\[
\int g\,\dint\bP^\delta_t=\int_0^t
\dint s\int\bP^\delta _s(\dint\bT_s)
\LL_s^\Psi g(\bT_s)
\]
with
\[
\LL_s^\Psi g(\bT_s)=\sum
_{c\in T_s\dvtx  c\cap[k]\ne\varnothing}\int_{\langle c\cap[k]\rangle} \Psi(s,
\bT_s,c,\dint H) \bigl[g \bigl(\oslash_{s,c,H}(
\bT_s) \bigr)-g(\bT_s) \bigr].
\]
Now, (M2) and (M4$'$) imply that $\LL_s^\Psi g(\bT_s)$ depends only on
$\bT_{[k+r],s}$ with $r=2(r_\Psi+r_\Psi')$.
On the other hand, using (M3$'$) and the additivity of $c\mapsto\Lambda
(\langle
c\rangle)$ we find that
\[
\bigl|\LL_s^\Psi g(\bT_s)\bigr|\le2 K_\Psi\sum
_{c\ni T_s\dvtx c\cap[k]\neq
\varnothing}\Lambda \bigl( \bigl\langle c\cap[k] \bigr
\rangle \bigr)=2 K_\Psi\Lambda \bigl( \bigl\langle[k] \bigr\rangle
\bigr)=:\alpha k.
\]
The total variation norm $\delta_{k}(t):=\|\bP^\delta_{[k],t}\|$
thus satisfies the inequality
%
\begin{equation}
\label{eqGronwall} \delta_k(t)\le\alpha k\int_0^t
\delta_{k+r}(s)\,\dint s
\end{equation}
of Gronwall type. (Note that $\delta_k$ is increasing and, therefore,
measurable.)
Since $\delta_{k+nr}(s)\le2$, we obtain by $n$-fold iteration
\[
\delta_k(t)\le2 \alpha^n (k+nr)^n
t^n/n! \le2 e^k \bigl(\alpha te^r
\bigr)^n
\]
and thus, in the limit as $n\to\infty$, $\delta_k(t)=0$ for all $t<
\varepsilon:=1/(\alpha e^r)$ and all $k$. Inserting this into (\ref{eqGronwall})
and repeating the estimate, we obtain that $\delta_k(t)=0$ for all
$t< 2\varepsilon$ and all $k$. Continuing in this way, we finally find
that $\delta_k(1)=0$ for all $k$, which means that $\bP=\bP'$.
\end{pf}

\section*{Acknowledgements}
We would like to thank Claudia Redenbach for providing the simulation
pictures of Figure~\ref{Bild}, and an anonymous referee for helpful
hints and suggestions.




\printaddresses


\begin{thebibliography}{31}
\bibitem{ArakSurgailis89}
\begin{barticle}[mr]
\bauthor{\bsnm{Arak},~\bfnm{T.}\binits{T.}} \AND
\bauthor{\bsnm{Surgailis},~\bfnm{D.}\binits{D.}}
(\byear{1989}).
\btitle{Markov fields with polygonal realizations}.
\bjournal{Probab. Theory Related Fields}
\bvolume{80}
\bpages{543--579}.
\bid{doi={10.1007/BF00318906}, issn={0178-8051}, mr={0980687}}
\end{barticle}
\bptok{imsref}%
\endbibitem

\bibitem{ArakSurgailis91}
\begin{barticle}[mr]
\bauthor{\bsnm{Arak},~\bfnm{T.}\binits{T.}} \AND
\bauthor{\bsnm{Surgailis},~\bfnm{D.}\binits{D.}}
(\byear{1991}).
\btitle{Consistent polygonal fields}.
\bjournal{Probab. Theory Related Fields}
\bvolume{89}
\bpages{319--346}.
\bid{doi={10.1007/BF01198790}, issn={0178-8051}, mr={1113222}}
\end{barticle}
\bptok{imsref}%
\endbibitem

\bibitem{BertinBilliotDrouilhet}
\begin{barticle}[mr]
\bauthor{\bsnm{Bertin},~\bfnm{Etienne}\binits{E.}},
\bauthor{\bsnm{Billiot},~\bfnm{Jean-Michel}\binits{J.-M.}} \AND
\bauthor{\bsnm{Drouilhet},~\bfnm{R{\'e}my}\binits{R.}}
(\byear{1999}).
\btitle{Existence of {D}elaunay pairwise {G}ibbs point process with superstable component}.
\bjournal{J. Stat. Phys.}
\bvolume{95}
\bpages{719--744}.
\bid{doi={10.1023/A:1004551527790}, issn={0022-4715}, mr={1700922}}
\end{barticle}
\bptok{imsref}%
\endbibitem

\bibitem{Bertoin}
\begin{bbook}[mr]
\bauthor{\bsnm{Bertoin},~\bfnm{Jean}\binits{J.}}
(\byear{2006}).
\btitle{Random Fragmentation and Coagulation Processes}.
\bpublisher{Cambridge Univ. Press},
\blocation{Cambridge}.
\bid{doi={10.1017/CBO9780511617768}, mr={2253162}}
\end{bbook}
\bptok{imsref}%
\endbibitem

\bibitem{CattiauxRZ}
\begin{barticle}[mr]
\bauthor{\bsnm{Cattiaux},~\bfnm{P.}\binits{P.}},
\bauthor{\bsnm{Roelly},~\bfnm{S.}\binits{S.}} \AND
\bauthor{\bsnm{Zessin},~\bfnm{H.}\binits{H.}}
(\byear{1996}).
\btitle{Une approche gibbsienne des diffusions browniennes infini-dimensionnelles}.
\bjournal{Probab. Theory Related Fields}
\bvolume{104}
\bpages{147--179}.
\bid{doi={10.1007/BF01247836}, issn={0178-8051}, mr={1373374}}
\end{barticle}
\bptok{imsref}%
\endbibitem

\bibitem{DaiPra}
\begin{barticle}[mr]
\bauthor{\bsnm{Dai Pra},~\bfnm{Paolo}\binits{P.}}
(\byear{1993}).
\btitle{Large deviations and stationary measures for interacting particle systems}.
\bjournal{Stochastic Process. Appl.}
\bvolume{48}
\bpages{9--30}.
\bid{doi={10.1016/0304-4149(93)90105-D}, issn={0304-4149}, mr={1237166}}
\end{barticle}
\bptok{imsref}%
\endbibitem

\bibitem{DaiPraRZ}
\begin{barticle}[mr]
\bauthor{\bsnm{Dai Pra},~\bfnm{Paolo}\binits{P.}},
\bauthor{\bsnm{Roelly},~\bfnm{Sylvie}\binits{S.}} \AND
\bauthor{\bsnm{Zessin},~\bfnm{Hans}\binits{H.}}
(\byear{2002}).
\btitle{A {G}ibbs variational principle in space-time for infinite-dimensional diffusions}.
\bjournal{Probab. Theory Related Fields}
\bvolume{122}
\bpages{289--315}.
\bid{doi={10.1007/s004400100170}, issn={0178-8051}, mr={1894070}}
\end{barticle}
\bptok{imsref}%
\endbibitem

\bibitem{DereudreDrouilhetGeorgii}
\begin{barticle}[mr]
\bauthor{\bsnm{Dereudre},~\bfnm{David}\binits{D.}},
\bauthor{\bsnm{Drouilhet},~\bfnm{Remy}\binits{R.}} \AND
\bauthor{\bsnm{Georgii},~\bfnm{Hans-Otto}\binits{H.-O.}}
(\byear{2012}).
\btitle{Existence of {G}ibbsian point processes with geometry-dependent interactions}.
\bjournal{Probab. Theory Related Fields}
\bvolume{153}
\bpages{643--670}.
\bid{doi={10.1007/s00440-011-0356-5}, issn={0178-8051}, mr={2948688}}
\end{barticle}
\bptok{imsref}%
\endbibitem

\bibitem{DereudreGeorgii}
\begin{barticle}[mr]
\bauthor{\bsnm{Dereudre},~\bfnm{David}\binits{D.}} \AND
\bauthor{\bsnm{Georgii},~\bfnm{Hans-Otto}\binits{H.-O.}}
(\byear{2009}).
\btitle{Variational characterisation of {G}ibbs measures with {D}elaunay triangle interaction}.
\bjournal{Electron. J. Probab.}
\bvolume{14}
\bpages{2438--2462}.
\bid{doi={10.1214/EJP.v14-713}, issn={1083-6489}, mr={2563247}}
\end{barticle}
\bptok{imsref}%
\endbibitem

\bibitem{Deuschel}
\begin{barticle}[mr]
\bauthor{\bsnm{Deuschel},~\bfnm{J.-D.}\binits{J.-D.}}
(\byear{1986}).
\btitle{Nonlinear smoothing of infinite-dimensional diffusion processes}.
\bjournal{Stochastics}
\bvolume{19}
\bpages{237--261}.
\bid{doi={10.1080/17442508608833427}, issn={0090-9491}, mr={0872463}}
\end{barticle}
\bptok{imsref}%
\endbibitem

\bibitem{Feller40}
\begin{barticle}[mr]
\bauthor{\bsnm{Feller},~\bfnm{Willy}\binits{W.}}
(\byear{1940}).
\btitle{On the integro-differential equations of purely discontinuous {M}arkoff processes}.
\bjournal{Trans. Amer. Math. Soc.}
\bvolume{48}
\bpages{488--515}.
\bid{issn={0002-9947}, mr={0002697}}
\end{barticle}
\bptok{imsref}%
\endbibitem

\bibitem{FoeSn}
\begin{barticle}[mr]
\bauthor{\bsnm{F{\"o}llmer},~\bfnm{H.}\binits{H.}} \AND
\bauthor{\bsnm{Snell},~\bfnm{J.~L.}\binits{J.~L.}}
(\byear{1977}).
\btitle{An ``inner'' variational principle for {M}arkov fields on a graph}.
\bjournal{Z. Wahrsch. Verw. Gebiete}
\bvolume{39}
\bpages{187--195}.
\bid{mr={0445642}}
\end{barticle}
\bptok{imsref}%
\endbibitem

\bibitem{GEO}
\begin{bbook}[mr]
\bauthor{\bsnm{Georgii},~\bfnm{Hans-Otto}\binits{H.-O.}}
(\byear{2011}).
\btitle{Gibbs Measures and Phase Transitions},
\bedition{2nd} ed.
\bpublisher{de Gruyter},
\blocation{Berlin}.
\bid{doi={10.1515/9783110250329}, mr={2807681}}
\end{bbook}
\bptok{imsref}%
\endbibitem

\bibitem{Israel}
\begin{bbook}[mr]
\bauthor{\bsnm{Israel},~\bfnm{Robert~B.}\binits{R.~B.}}
(\byear{1979}).
\btitle{Convexity in the Theory of Lattice Gases}.
\bpublisher{Princeton Univ. Press},
\blocation{Princeton, NJ}.
\bid{mr={0517873}}
\end{bbook}
\bptok{imsref}%
\endbibitem

\bibitem{KRM}
\begin{bbook}[mr]
\bauthor{\bsnm{Kallenberg},~\bfnm{Olav}\binits{O.}}
(\byear{1983}).
\btitle{Random Measures},
\bedition{3rd} ed.
\bpublisher{Akademie Verlag},
\blocation{Berlin}.
\bid{mr={0818219}}
\end{bbook}
\bptok{imsref}%
\endbibitem

\bibitem{KP}
\begin{bbook}[mr]
\bauthor{\bsnm{Kallenberg},~\bfnm{Olav}\binits{O.}}
(\byear{2002}).
\btitle{Foundations of Modern Probability},
\bedition{2nd} ed.
\bpublisher{Springer},
\blocation{ New York}.
\bid{mr={1876169}}
\end{bbook}
\bptok{imsref}%
\endbibitem

\bibitem{MNW}
\begin{barticle}[mr]
\bauthor{\bsnm{Mecke},~\bfnm{J.}\binits{J.}},
\bauthor{\bsnm{Nagel},~\bfnm{W.}\binits{W.}} \AND
\bauthor{\bsnm{Weiss},~\bfnm{V.}\binits{V.}}
(\byear{2008}).
\btitle{A global construction of homogeneous random planar tessellations that are stable under iteration}.
\bjournal{Stochastics}
\bvolume{80}
\bpages{51--67}.
\bid{doi={10.1080/17442500701605403}, issn={1744-2508}, mr={2384821}}
\end{barticle}
\bptok{imsref}%
\endbibitem

\bibitem{NW05}
\begin{barticle}[mr]
\bauthor{\bsnm{Nagel},~\bfnm{Werner}\binits{W.}} \AND
\bauthor{\bsnm{Weiss},~\bfnm{Viola}\binits{V.}}
(\byear{2005}).
\btitle{Crack {STIT} tessellations: Characterization of stationary random tessellations stable with respect to iteration}.
\bjournal{Adv. in Appl. Probab.}
\bvolume{37}
\bpages{859--883}.
\bid{doi={10.1239/aap/1134587744}, issn={0001-8678}, mr={2193987}}
\end{barticle}
\bptok{imsref}%
\endbibitem

\bibitem{OkaBooSugChi00}
\begin{bbook}[mr]
\bauthor{\bsnm{Okabe},~\bfnm{Atsuyuki}\binits{A.}},
\bauthor{\bsnm{Boots},~\bfnm{Barry}\binits{B.}},
\bauthor{\bsnm{Sugihara},~\bfnm{Kokichi}\binits{K.}} \AND
\bauthor{\bsnm{Chiu},~\bfnm{Sung~Nok}\binits{S.~N.}}
(\byear{2000}).
\btitle{Spatial Tessellations: Concepts and Applications of {V}oronoi Diagrams}.
\bpublisher{Wiley},
\blocation{Chichester}.
\bid{mr={1770006}}
\end{bbook}
\bptok{imsref}%
\endbibitem

\bibitem{RedenbachThaele}
\begin{barticle}[mr]
\bauthor{\bsnm{Redenbach},~\bfnm{Claudia}\binits{C.}} \AND
\bauthor{\bsnm{Th{\"a}le},~\bfnm{Christoph}\binits{C.}}
(\byear{2013}).
\btitle{On the arrangement of cells in planar {STIT} and {P}oisson line tessellations}.
\bjournal{Methodol. Comput. Appl. Probab.}
\bvolume{15}
\bpages{643--654}.
\bid{issn={1387-5841}, mr={3085884}}
\end{barticle}
\bptok{imsref}%
\endbibitem

\bibitem{Ross}
\begin{bbook}[auto:STB|2014/02/12|14:17:21]
\bauthor{\bsnm{Ross},~\bfnm{S.~M.}\binits{S.~M.}}
(\byear{2003}).
\btitle{Introduction to Probability Models},
\bedition{8th} ed.
\bpublisher{Academic Press},
\blocation{Amsterdam}.
\end{bbook}
\bptok{imsref}%
\endbibitem

\bibitem{SW}
\begin{bbook}[mr]
\bauthor{\bsnm{Schneider},~\bfnm{Rolf}\binits{R.}} \AND
\bauthor{\bsnm{Weil},~\bfnm{Wolfgang}\binits{W.}}
(\byear{2008}).
\btitle{Stochastic and Integral Geometry}.
\bpublisher{Springer},
\blocation{Berlin}.
\bid{doi={10.1007/978-3-540-78859-1}, mr={2455326}}
\end{bbook}
\bptok{imsref}%
\endbibitem

\bibitem{ST2}
\begin{barticle}[mr]
\bauthor{\bsnm{Schreiber},~\bfnm{Tomasz}\binits{T.}} \AND
\bauthor{\bsnm{Th{\"a}le},~\bfnm{Christoph}\binits{C.}}
(\byear{2010}).
\btitle{Second-order properties and central limit theory for the vertex process
of iteration infinitely divisible and iteration stable random tessellations in the plane}.
\bjournal{Adv. in Appl. Probab.}
\bvolume{42}
\bpages{913--935}.
\bid{doi={10.1239/aap/1293113144}, issn={0001-8678}, mr={2796670}}
\end{barticle}
\bptok{imsref}%
\endbibitem

\bibitem{ST3}
\begin{barticle}[mr]
\bauthor{\bsnm{Schreiber},~\bfnm{Tomasz}\binits{T.}} \AND
\bauthor{\bsnm{Th{\"a}le},~\bfnm{Christoph}\binits{C.}}
(\byear{2011}).
\btitle{Intrinsic volumes of the maximal polytope process in higher dimensional {STIT} tessellations}.
\bjournal{Stochastic Process. Appl.}
\bvolume{121}
\bpages{989--1012}.
\bid{doi={10.1016/j.spa.2011.01.001}, issn={0304-4149}, mr={2775104}}
\end{barticle}
\bptok{imsref}%
\endbibitem

\bibitem{ST7}
\begin{barticle}[mr]
\bauthor{\bsnm{Schreiber},~\bfnm{Tomasz}\binits{T.}} \AND
\bauthor{\bsnm{Th{\"a}le},~\bfnm{Christoph}\binits{C.}}
(\byear{2012}).
\btitle{Second-order theory for iteration stable tessellations}.
\bjournal{Probab. Math. Statist.}
\bvolume{32}
\bpages{281--300}.
\bid{issn={0208-4147}, mr={3021459}}
\bptnote{check year}%
\end{barticle}
\bptok{imsref}%
\endbibitem

\bibitem{ST4}
\begin{barticle}[mr]
\bauthor{\bsnm{Schreiber},~\bfnm{Tomasz}\binits{T.}} \AND
\bauthor{\bsnm{Th{\"a}le},~\bfnm{Christoph}\binits{C.}}
(\byear{2013}).
\btitle{Shape-driven nested {M}arkov tessellations}.
\bjournal{Stochastics}
\bvolume{85}
\bpages{510--531}.
\bid{doi={10.1080/17442508.2011.654344}, issn={1744-2508}, mr={3176472}}
\end{barticle}
\bptok{imsref}%
\endbibitem

\bibitem{ST5}
\begin{barticle}[mr]
\bauthor{\bsnm{Schreiber},~\bfnm{Tomasz}\binits{T.}} \AND
\bauthor{\bsnm{Th{\"a}le},~\bfnm{Christoph}\binits{C.}}
(\byear{2013}).
\btitle{Limit theorems for iteration stable tessellations}.
\bjournal{Ann. Probab.}
\bvolume{41}
\bpages{2261--2278}.
\bid{doi={10.1214/11-AOP718}, issn={0091-1798}, mr={3098072}}
\end{barticle}
\bptok{imsref}%
\endbibitem

\bibitem{ST6}
\begin{barticle}[mr]
\bauthor{\bsnm{Schreiber},~\bfnm{Tomasz}\binits{T.}} \AND
\bauthor{\bsnm{Th{\"a}le},~\bfnm{Christoph}\binits{C.}}
(\byear{2013}).
\btitle{Geometry of iteration stable tessellations: Connection with {P}oisson hyperplanes}.
\bjournal{Bernoulli}
\bvolume{19}
\bpages{1637--1654}.
\bid{doi={10.3150/12-BEJ424}, issn={1350-7265}, mr={3129028}}
\end{barticle}
\bptok{imsref}%
\endbibitem

\bibitem{SKM}
\begin{bbook}[auto:STB|2014/02/12|14:17:21]
\bauthor{\bsnm{Stoyan},~\bfnm{D.}\binits{D.}},
\bauthor{\bsnm{Kendall},~\bfnm{D.~G.}\binits{D.~G.}} \AND
\bauthor{\bsnm{Mecke},~\bfnm{J.}\binits{J.}}
(\byear{1995}).
\btitle{Stochastic Geometry},
\bedition{2nd} ed.
\bpublisher{Wiley},
\blocation{Chichester}.
\end{bbook}
\bptok{imsref}%
\endbibitem

\bibitem{TWN}
\begin{barticle}[mr]
\bauthor{\bsnm{Th{\"a}le},~\bfnm{Christoph}\binits{C.}},
\bauthor{\bsnm{Weiss},~\bfnm{Viola}\binits{V.}} \AND
\bauthor{\bsnm{Nagel},~\bfnm{Werner}\binits{W.}}
(\byear{2012}).
\btitle{Spatial {STIT} tessellations: Distributional results for {I}-segments}.
\bjournal{Adv. in Appl. Probab.}
\bvolume{44}
\bpages{635--654}.
\bid{doi={10.1239/aap/1346955258}, issn={0001-8678}, mr={3024603}}
\end{barticle}
\bptok{imsref}%
\endbibitem

\bibitem{Varadhan}
\begin{bincollection}[mr]
\bauthor{\bsnm{Varadhan},~\bfnm{Srinivasa~R.~S.}\binits{S.~R.~S.}}
(\byear{1988}).
\btitle{Large deviations and applications}.
In \bbooktitle{\'{E}cole D'\'{E}t\'e de {P}robabilit\'es de {S}aint-{F}lour {XV}--{XVII}, 1985--87}.
\bseries{Lecture Notes in Math.}
\bvolume{1362}
\bpages{1--49}.
\bpublisher{Springer},
\blocation{Berlin}.
\bid{doi={10.1007/BFb0086178}, mr={0983371}}
\bptnote{check year}%
\end{bincollection}
\bptok{imsref}%
\endbibitem

\end{thebibliography}
\end{document}